\documentclass[11pt,reqno]{amsart}
\usepackage{amsmath, amsfonts, amsthm, amssymb,color}
\usepackage{setspace}
\usepackage{amsthm,amsmath,amssymb}
\usepackage{mathrsfs}
\usepackage{geometry}
\usepackage{amsfonts}
\usepackage[utf8]{inputenc}
\usepackage{amssymb}
\usepackage{amsmath}
\usepackage{subfigure}
\usepackage{graphicx}
\usepackage{amsmath,amscd}
\everymath{\displaystyle}
\usepackage{indentfirst} 
\usepackage{enumerate}
\usepackage[colorlinks]{hyperref}
\usepackage{bm}

\usepackage{algorithm}
\usepackage{algorithmic}
\usepackage{mathrsfs}
\usepackage[toc,page,title,titletoc,header]{appendix}
\usepackage{fancyhdr}

\numberwithin{equation}{section}

\newtheorem{thm}{Theorem}[section]
\newtheorem{lemma}[thm]{Lemma}
\newtheorem{cor}[thm]{Corollary}
\newtheorem{prop}[thm]{Proposition}
\newtheorem{rmk}[thm]{Remark}

\usepackage{graphicx}
\numberwithin{equation}{section}

\newcommand{\ud}{\,\mathrm{d}}
\newcommand{\uid}{\,\mathrm{Id}}

\newcommand{\eps}{\varepsilon}

\newcommand{\R}{{\mathbb{R}}}

\newcommand{\E}{{\mathbb{E}}}

\def\Id{\operatorname{Id}}
\def\p{\partial}

\definecolor{darkgreen}{rgb}{0.0, 0.42, 0.0}

\author[X. Feng]{Xuanrui Feng}
\address{School of Mathematical Sciences, Peking University, Beijing 100871, China.}
\email{pkufengxuanrui@stu.pku.edu.cn}

\author[Z. Wang]{Zhenfu Wang}
\address{Beijing International Center for Mathematical Research, Peking University, Beijing 100871, China}
\email{zwang@bicmr.pku.edu.cn}

\title[Kac's Program for the Landau Equation]
{Kac's Program for the Landau Equation}
\date{\today}

\begin{document}

\begin{abstract}

    We study the derivation of the spatially homogeneous Landau equation from the mean-field limit of a conservative $N$-particle system, obtained by passing to the grazing limit on Kac's walk in his program for the Boltzmann equation. Our result covers the full range of interaction potentials, including the physically important Coulomb case. This provides the first resolution of propagation of chaos for a many-particle system approximating the Landau equation with Coulomb interactions, and the first extension of Kac's program to the Landau equation in the soft potential regime. The convergence is established in weak, Wasserstein, and entropic senses, together with strong $L^1$ convergence. To handle the singularity of soft potentials, we extend the duality approach of Bresch-Duerinckx-Jabin \cite{bresch2024duality} and establish key functional inequalities, including an extended commutator estimate and a new second-order Fisher information estimate.
    
\end{abstract}

\maketitle

\tableofcontents

\section{Introduction}

The celebrated Landau equation, lying at the core of kinetic theory and plasma physics, describes the time evolution of the density of some dilute monatomic gas molecules undergoing short-range potentials. The systematic study of the dynamics of gases can be traced back to 1867, when Maxwell \cite{Maxwell1867} wrote down the (later called) Boltzmann equation and discovered its Maxwellian equilibrium. Later in 1872, Boltzmann \cite{boltzmann1872weitere} explained the convergence to the Maxwellian equilibrium for the solutions by establishing the famous Boltzmann $H$-Theorem, proving that the entropy increases along the Boltzmann equation. It was not until 1936 when Landau \cite{landau1936kinetische} derived his equation as an alternative to the Boltzmann equation when the charged particles perform Coulomb collisions in a plasma, the case which the Boltzmann equation barely fails to explain. It serves as a formal limit of the Boltzmann equation when grazing collisions prevail, that is, most of the collisions asymptotically result in small changes in the momentum of the particles, or the collision kernel concentrates on small angles.

In his seminal paper, Boltzmann assumed the famous \textit{molecular chaos}, or in his words \textit{Stosszahlansatz}, property, which means that any two given particles are assumed to be uncorrelated/independent before the binary collision happens between them. This results in the form of the Boltzmann collisional operator, which is fully determined by two-body collisions. However, the rigorous justification of this molecular chaos assumption and the derivation of the Boltzmann equation from Newton dynamics undergoing the low-density limit (also known as the Boltzmann-Grad limit \cite{grad1949kinetic}) have been long-standing open problems. The short-time result was first obtained by Lanford \cite{lanford1975time} and developed by a series of works, for instance Illner-Pulvirenti \cite{illner1989global}, Gallagher-Saint-Raymond-Texier \cite{gallagher2014newton} and Bodineau-Gallagher-Saint-Raymond-Simonella \cite{bodineau2023statistical}. It was only recently that Deng-Hani-Ma \cite{deng2024long,deng2025hilbert} settled on the long-time derivation of the Boltzmann equation from the hard sphere dynamics, provided that the Boltzmann equation possesses a classical solution. However, the rigorous derivation of the Landau equation from Newton dynamics undergoing the weak-coupling limit is still open. A consistency result was given by Bobylev-Pulvirenti-Saffirio \cite{bobylev2013particle}.

Alternatively, Kac \cite{kac1956foundations} introduced his seminal program to derive the spatially homogeneous Boltzmann equation from a large conservative $N$-particle system undergoing binary collisions. Kac has justified the \textit{propagation of (molecular) chaos} in his one-dimensional caricature and derived the Boltzmann equation from the \textit{mean-field limit} of the master equation of the particle system. In the past decades, Kac's program has been deeply developed further by many experts in kinetic theory, including for instance McKean \cite{mckean1967propagation}, Carlen-Carvalho-Le Roux-Loss-Villani \cite{carlen2010entropy}, Mischler-Mouhot \cite{mischler2013kac}, Hauray-Mischler \cite{hauray2014kac} and Carrapatoso \cite{carrapatoso2015quantitative}.

We summarize Kac's idea in a short paragraph. Assume that there are $N$ indistinguishable particles distributed in velocity independently and identically according to the initial distribution $f_0$. Denote the particle system by the initial velocity vector $V_N=(v_1, \cdots, v_N)$. Suppose that the Boltzmann collision kernel is in the form of $B(|z|,\cos \theta)=\Gamma(|z|) \cdot b(\cos\theta)$. For each pair of indices $(i,j)$ where $i < j$, we choose a random collision time $T(\Gamma(|v_i-v_j|))$ according to the exponential distribution with parameter $\Gamma(|v_i-v_j|)$. Let $(i,j)$ be the pair of indices such that the collision time is minimum among all the choices and choose the collision direction $\sigma \in \mathbb{S}^2$ according to the distribution $b\Big(\frac{(v_i-v_j) \cdot \sigma}{|v_i-v_j|}\Big)$. Then the post-collisional velocity vector becomes $V^\ast_N=(v_1,\cdots,v^\ast_i,\cdots,v^\ast_j,\cdots,v_N)$, where the new velocities are determined by the conservation law of momentum and kinetic energy, with the explicit expression
\begin{equation*}
    v^\ast_i=\frac{v_i+v_j}{2}+\frac{|v_i-v_j|}{2}\sigma, \quad v^\ast_j=\frac{v_i+v_j}{2}-\frac{|v_i-v_j|}{2}\sigma.
\end{equation*}
Iterating the collision process above gives our desired Markovian jump process. Kac conjectures that as $N \rightarrow \infty$, the chaotic property of the joint law of any given $k$ particles in the particle system can propagate to any positive time.

Accordingly, Kac's program has been carried out for the Landau equation by passing the master equation of the stochastic particle system to the grazing limit. The first such attempt in the literature appears to be Miot-Pulvirenti-Saffirio \cite{miot2011kac}, where the authors proved the convergence of the BBGKY hierarchy solved by the marginal distributions of the truncated Kac's particle system to the Landau hierarchy. The truncation condition has been removed by the recent work of Carrillo-Guo \cite{carrillo2025fisher}. However, the uniqueness of solutions to the Landau hierarchy remains open, and thus propagation of chaos is not fully derived. Carrapatoso \cite{carrapatoso2016propagation} successfully constructed Kac's particle system by passing Kac's walk to the grazing limit and proved the mean-field convergence for the Landau equation with Maxwellian molecules, using the abstract analytic framework and the semigroup method developed in \cite{mischler2013kac}. Fournier-Guillin \cite{fournier2017kac} improved the convergence rate and extended the result to hard potentials. However, up to the authors' knowledge, there is still no complete mean-field convergence result of Kac's particle system for the Landau equation in the regime of soft potentials, leaving alone the most interesting case with Coulomb interactions.

After the completion of the first version of this manuscript, Tabary proved in his preprint \cite{tabary2025propagation} propagation of chaos for the truncated Kac's particle system in the regime of very soft potentials including the Coulomb interactions, but mainly based on a very different tightness-uniqueness approach.

\subsection{Landau Equation}

In this article, we consider the spatially homogeneous Landau equation in $\R^3$ in the following form:
\begin{equation}\label{landau}
   \begin{cases}
\frac{\partial f}{\partial t}=Q(f, f)\triangleq \frac{\partial}{\partial v_\alpha}\int_{\mathbb{R}^3} a_{\alpha \beta}(v-w)\Big(f(w)\frac{\partial f(v)}{\partial v_\beta}-f(v)\frac{\partial f(w)}{\partial w_{\beta}}\Big) \ud w,\\
f(0)=f_{0}, \end{cases}
\end{equation}
where $f: [0,T]\times\R^3\to [0,\infty)$ describes the density of the gas molecules with velocity $v\in \R^3$ at time $t\in[0,T]$. Hereinafter, we shall generally adopt Einstein's summation convention for subindices such as $\alpha, \beta$ running from $1$ to $3$, representing the components of vectors or the entries of matrices.

The matrix-valued function $a$ appearing in the Landau equation \eqref{landau} is defined by
\begin{equation*}
    a(z)=\alpha(|z|)P(z),
\end{equation*}
where the non-negative definite matrix is given by $P(z)=\big(|z|^2 \operatorname{Id}-z\otimes z\big)$, and the interaction potential $\alpha(r)$ represents the strength of pairwise interactions between molecules. The most interesting and physically meaningful case corresponds to the Coulomb interactions with $\alpha(r)=r^{-3}$, but for the interest of mathematical study, we can also consider a wider range of potentials. Generally, in this article we consider power-law interactions $\alpha(r)=r^\gamma$ with the parameter $\gamma \in [-3,1]$ measuring the strength of the pairwise interaction. Usually, we speak of hard potentials for $\gamma \in (0,1]$, Maxwellian molecules for the critical case $\gamma=0$, and soft potentials for $\gamma \in [-3,0)$. When $\gamma \in [-2,0)$, which is often referred as moderately soft potentials, the interaction matrix $a$ shows no singularity; while for $\gamma \in [-3,-2)$, which falls into the regime of very soft potentials, the essential singularity appears. The critical case $\gamma=-3$ exactly corresponds to the Coulomb interactions.

The nonlinear operator $Q$ stands for the Landau collision operator, which models the effect of binary collisions between molecules. One understands the structure of the collision operator by observing that the change of number of molecules with velocity $v$ is equal to the difference of two effects, one being the gain of post-collisional velocity, the other being the loss of pre-collisional velocity. This fact is reflected in the integral-differential equation. The explicit form of the Landau collision operator is derived by formally passing to the grazing limit in the Boltzmann collision operator.

Using integration by parts and the integral-differential structure, the Landau equation can be reformulated into the divergence form:
\begin{equation}\label{compactlandau}
    \p_tf =\nabla\cdot [(a\ast f)\nabla f- (b\ast f)f],
\end{equation}
where the vector-valued function $b$ is defined by the divergence of $a$:
\begin{equation*}
    b(z)=\nabla\cdot a(z)=-2|z|^{\gamma} z.
\end{equation*}
We can also expand the bracket and rewrite it into the non-divergence form:
\begin{equation*}
    \p_t f=(a\ast f):\nabla^2 f-(c \ast f)f,
\end{equation*}
where the coefficient $c$ is given by
\begin{equation*}
    c(z)=\begin{cases}
        -2(\gamma+3)|z|^\gamma \quad \gamma>-3;\\
        -8\pi \delta_0 \qquad \, \qquad \gamma=-3,
    \end{cases}
\end{equation*}
with $\delta_0$ denoting the Dirac mass at $0$, in the sense that
\begin{equation*}
    c \ast f =\nabla \cdot (b \ast f).
\end{equation*}
We introduce the simplified notions to represent the convolution coefficients:
\begin{equation*}
    \bar a=a \ast f, \quad \bar b =b \ast f, \quad \bar c =c \ast f.
\end{equation*}

It is well known that the Landau equation formally conserves the total mass, momentum and kinetic energy of the gas molecules, and has monotonically decreasing Boltzmann entropy, known as the analog of the famous Boltzmann $H$-Theorem. Therefore, we shall always assume that the initial data $f_0$ has finite moments up to second order and has finite entropy. For simplicity, we normalize the conservative quantities at any time $t$ to be
\begin{equation}\label{normalize}
    \int_{\R^3} f(v) \ud v=1, \quad \int_{\R^3} vf(v) \ud v=0, \quad \int_{\R^3} |v|^2f(v) \ud v=3.
\end{equation}

The well-posedness and regularity of the Landau equation have been extensively studied over the past few decades. For the case of Coulomb interactions, which is the most relevant model in physics, the first version of generalized solutions is given by Arsen'ev-Peskov \cite{arsenev77existence}. Villani \cite{villani1998new} defined the notion of $H$-solutions and proved its global existence using the entropy dissipation formula. Desvillettes \cite{desvillettes2015entropy} established the $L_{-3}^3$ regularity of the $H$-solutions (later improved by Ji \cite{ji2025entropy}) and proved that they are actually weak solutions through estimates of the entropy dissipation functional. Fournier \cite{fournier2010uniqueness} proved the uniqueness of bounded smooth solutions. Local existence, uniqueness and stability results for classical smooth solutions with general initial data were investigated by Fournier-Gu\'erin \cite{fournier2009well}, Henderson-Snelson-Tarfulea \cite{henderson2020local} and Golding-Loher \cite{golding2024local}. Partial regularity results concerning the singularity set were obtained by Golse-Gualdani-Imbert-Vasseur \cite{golse2022partial} and Golse-Imbert-Ji-Vasseur \cite{golse2024local}. The recent breakthrough by Guillen-Silvestre \cite{guillen2023landau} established the monotonic decrease of the Fisher information of classical solutions and thus proved that the Landau equation does not blow up. The main idea was summarized in a recent short note \cite{guillen2025landau} by the authors themselves. Global existence and uniqueness of classical solutions were improved by Golding-Gualdani-Loher \cite{golding2025global}, Desvillettes-Golding-Gualdani-Loher \cite{desvillettes2024production},  Ji \cite{ji2024dissipation} and He-Ji-Luo \cite{he2024existence} with more general initial data. We also mention the recent preprint by Tabary \cite{tabary2025weak} where a weak-strong uniqueness result in the sense of relative entropy was obtained.

For the cases of hard potentials and moderately soft potentials, classical well-posedness results have been established much earlier. Among the literature we mention the following: the Maxwellian molecule case was studied by Villani \cite{villani1998spatially}; the cases of hard potentials were summarized by Desvillettes-Villani \cite{desvillettes2000spatiallyI,desvillettes2000spatiallyII}  and Fournier-Heydecker \cite{fournier2021stability}; the cases of moderately soft potentials were studied by Silvestre \cite{silvestre2017upper} and Cameron-Silvestre-Snelson \cite{cameron2018global}. We also refer to the reviews by Villani \cite{villani2002review} and Silvestre \cite{silvestre2023regularity} and the references therein.

\subsection{Particle System and Master Equation}

The master equation of Kac's particle system for the Landau equation can be derived by formally passing the Boltzmann master equation to the grazing limit. This procedure was first formulated in Carrapatoso \cite{carrapatoso2016propagation}, but earlier studies of this master equation can be found in Balescu-Prigogine \cite{prigogine1959irreversibleI,prigogine1959irreversibleII}, Kiessling-Lancelloti \cite{kiessling2004master} and Miot-Pulvirenti-Saffirio \cite{miot2011kac}. The weak formulation of the Landau master equation is given by
\begin{align*}
    \p_t \langle F_N, \varphi \rangle&=\frac{1}{N}\sum_{i,j=1}^N \int_{\R^{3N}} b(v^i-v^j) \cdot (\nabla_{v^i} \varphi -\nabla_{v^j} \varphi) F_N \ud v^{[N]}\\
    &+\frac{1}{2N} \sum_{i,j=1}^N \int_{\R^{3N}} a(v^i-v^j):(\nabla_{v^iv^i}^2 \varphi-\nabla_{v^iv^j}^2 \varphi-\nabla_{v^jv^i}^2 \varphi+\nabla_{v^jv^j}^2 \varphi)F_N \ud v^{[N]},
\end{align*}
where $\varphi \in C_b^2(\R^{3N})$ is any test function whose derivatives up to second order are continuous and bounded. Hereinafter, we adopt the notation $[N]=\{1,\cdots,N\}$ and $\ud v^{[N]}=\ud v^1 \cdots \ud v^N$. We also use the convention $a(0)=0$ and $b(0)=0$ to avoid self-interactions. Using integration by parts, we can also derive the following Landau master equation (Liouville equation or forward Kolmogorov equation):
\begin{equation}\label{master}
    \partial_t F_N=\frac{1}{2N} \sum_{i,j=1}^N (\nabla_{v^i}-\nabla_{v^j}) \cdot \Big( a(v^i-v^j) \cdot (\nabla_{v^i}-\nabla_{v^j})F_N\Big).
\end{equation}
The following study will always focus on \eqref{master}. Due to the singularity and degeneracy of the coefficient matrix $a$ in the regime of very soft potentials, some effort is required to justify the well-posedness of the master equation. We understand $F_N$ as the unique bounded weak solution of \eqref{master}. The existence, uniqueness and regularity of the weak solution have been established in the literature and follow from standard arguments. We give a brief review in Appendix \ref{appendixB}.

We can also construct a particle system described by a large stochastic differential equation (SDE) system to meet the master equation \eqref{master} using the It\^o formula. Consider the time evolution of the velocities of $N$ indistinguishable particles given by the following SDE system:
\begin{equation}\label{SDE}
    \ud V_t^i=\frac{2}{N}\sum_{j=1}^N b(V_t^i-V_t^j) \ud t +\frac{\sqrt{2}}{\sqrt{N}} \sum_{j=1}^N \sigma(V_t^i-V_t^j) \ud Z_t^{i,j}.
\end{equation}
Here, particles are labeled by the index $i$ ranging from $1$ to $N$ and $V_t^i$ represents the velocity of the $i$-th particle at time $t$. The vector field $\sigma$ is given by the square root of the non-negative definite coefficient matrix $a$. The random factor $Z_t^{i,j}$ is defined in the following anti-symmetric sense: for all pairs of indices $i<j$,  $\{Z_t^{i,j}\}=\{B_t^{i,j}\}$ are $N(N-1)/2$ independent $3$-dimensional standard Brownian motions, and $Z_t^{j,i}=-B_t^{i,j}$. It is straightforward to check that the joint law $F_N$ of the velocities of the particles satisfies the master equation \eqref{master}.

Since we have assumed that the particles are indistinguishable, the joint law $F_N$ remains symmetric or exchangeable with respect to the $N$ variables at any time $t$. Moreover, we assume for simplicity that initially the particles are $f_0$-chaotic with fully factorized joint law
\begin{equation*}
    F_N(0, \cdot)=f_0^{\otimes N},
\end{equation*}
with $f_0$ being the initial data of the Landau equation and satisfying some regularity assumptions. From the statistical point of view, the macroscopic information of the particle system is included in the $k$-marginals of the joint law $F_N$:
\begin{equation*}
    F_{N,k}(v^{[k]})=\int_{\R^{3(N-k)}} F_N(v^{[N]}) \ud v^{k+1} \cdots \ud v^N.
\end{equation*}
Kac's propagation of chaos argument is to establish the conjectured weak convergence of $F_{N,k}$ to $f^{\otimes k}$ for any fixed $k$ at any time $t$, once it holds initially, which is satisfied by our chaotic initial data.

We mention that apart from Kac's model \eqref{SDE}, there are also other kinds of particle systems which have been used to approximate the Landau equation. The first such attempt belongs to Fontbona-Gu\'erin-M\'el\'eard \cite{fontbona2009measurability}, where the authors constructed a particle system in the sense of probabilistic transformation and proved the convergence in the Wasserstein distance using optimal transport techniques. The SDE system of their model is given by
\begin{equation}\label{SDE_new_1}
    \ud V_t^i=\frac{2}{N}\sum_{j=1}^N b(V_t^i-V_t^j) \ud t +\frac{\sqrt{2}}{\sqrt{N}} \sum_{j=1}^N \sigma(V_t^i-V_t^j) \ud B_t^{i,j}.
\end{equation}
Compared to Kac's particle system given by \eqref{SDE}, the randomness here is described by $N^2$ independent standard Brownian motions, rather than in the sense of anti-symmetry. Another type of particle system considered by Fournier in \cite{fournier2009particle} reads 
\begin{equation}\label{SDE_new_2}
    \ud V_t^i=\frac{2}{N}\sum_{j=1}^N b(V_t^i-V_t^j) \ud t +\sqrt{\frac{2}{N} \sum_{j=1}^N a(V_t^i-V_t^j)} \ud B_t^{i},
\end{equation}
which is an alternative of the Nanbu particle system approximating the Boltzmann equation. Fournier has improved the decay rate in the Wasserstein-$2$ distance via a classical coupling approach. The two SDEs above share the same Liouville equation, which, however, differs from \eqref{master}. The recent work Carrillo-Feng-Guo-Jabin-Wang \cite{carrillo2024relative} revisited the latter system and established the quantitative mean-field convergence in relative entropy for the first time. All these mentioned results apply to the Landau equation with Maxwellian molecules. For the regime of moderately soft potentials, Fournier-Hauray \cite{fournier2016propagation} used the same particle system to prove the qualitative and quantitative mean-field convergence by the weak-strong stability argument and McKean's classical weak martingale formulation. However, none of these methods was extended to the regime of very soft potentials, especially for the case with Coulomb interactions. In particular, it remains open wheter the relative entropy method can be further extended to Kac's model \eqref{SDE} or other singular interacting particle systems to approximate the Landau equation with Coulomb interactions.

The most significant difference between the particle systems \eqref{SDE_new_1} \eqref{SDE_new_2} and Kac's particle system \eqref{SDE} is that, Kac's particle system preserves momentum and kinetic energy point-wise:
\begin{equation*}
    \sum_{i=1}^N V_t^i=\sum_{i=1}^N V_0^i, \quad \sum_{i=1}^N |V_t^i|^2=\sum_{i=1}^N |V_0^i|^2, \quad \text{a.s.};
\end{equation*}
while the particle systems \eqref{SDE_new_1} and \eqref{SDE_new_2}  only preserve the expectations of momentum and kinetic energy:
\begin{equation*}
    \E \sum_{i=1}^N V_t^i= \E \sum_{i=1}^N V_0^i, \quad \E \sum_{i=1}^N |V_t^i|^2=\E \sum_{i=1}^N |V_0^i|^2,
\end{equation*}
which makes the latter cases enjoy less physical relevance and properties. Therefore, the mean-field approximation of Kac's particle system has been considered a more physically relevant topic.

Finally, we mention that a collisional-oriented particle system approximating the Landau equation was constructed in the recent works of Du-Li \cite{du2024collision} and Du-Li-Xie-Yu \cite{du2024structure}, in the sense of the random batch method, which is effective for numerical simulations.

\subsection{Main Results}

Throughout this article, we adopt the notation of the Japanese bracket $\langle v \rangle=\sqrt{1+|v|^2}$ and introduce the weighted $L^p$ norm given by
\begin{equation*}
    \|f\|_{L_m^p}= \Big( \int_{\R^3} \langle v \rangle^{m} |f|^p  \ud v \Big)^{\frac{1}{p}}.
\end{equation*}

Our first main result is on the mean-field limit in the language of {\em propagation of chaos} of Kac's model \eqref{SDE} (or \eqref{master}) towards the spatially homogeneous Landau equation \eqref{compactlandau}.

\begin{thm}[Propagation of Chaos]\label{thm:poc}
    Let $f_0 \in L^1 \cap L^\infty(\R^3)$ be a probability density satisfying the normalized condition \eqref{normalize}. For any $\gamma \in [-3,1]$, let $F_N \in L^\infty([0,T];L^1 \cap L^\infty(\R^{3N}))$ be the unique bounded weak solution to the Landau master equation \eqref{master} with initial data $f_0^{\otimes N}$. Let $f \in C^1((0, \infty);\mathcal{S}(\R^3)) \cap L^\infty([0,\infty);L^1 \cap L^\infty(\R^3))$ be the unique bounded smooth solution to the Landau equation \eqref{compactlandau} with initial data $f_0$. 
    
    Assume that $f_0$ has finite weighted Fisher information and high-order $L^1$ moment:
    \begin{equation*}
        \int_{\R^3} \langle v \rangle^{\max(-\gamma,2\gamma+6)} |\nabla \log f_0(v)|^2f_0(v) \ud v <\infty,
    \end{equation*}
    \begin{equation*}
        \|f_0 \|_{L_m^1} <\infty \text{ for some } m>\max(6,2\gamma+8).
    \end{equation*}
    Then for any $T>0$, we have propagation of chaos: for any $k \geq 1$, the $k$-marginal $F_{N,k}$ converges weakly to $f^{\otimes k}$ as $N \rightarrow \infty$ on $[0,T] \times \R^{3k}$, i.e. for any smooth and compactly supported test function $\varphi \in C_c^\infty(\R^3)$ and for any $t \in [0,T]$, it holds that 
    \begin{equation*}
        \int_{\R^{3k}} (F_{N,k}-f^{\otimes k}) \varphi^{\otimes k} \ud v^{[k]} \rightarrow 0
    \end{equation*}
    as $N \rightarrow \infty$.
\end{thm}

\begin{rmk}\label{sketch}
    Under the above assumptions, the existence and uniqueness of the bounded smooth solution $f$ of the Landau equation with initial data $f_0$ can be obtained by combining the existing results in the literature. We give a quick sketch here. All constants in this remark shall be independent of time $t$.

    By \cite[Theorem 1.2]{guillen2023landau} and \cite[Theorem 1.5]{ji2024dissipation}, the initial data $f_0$ yields a unique global smooth solution with monotonically decreasing Fisher information. Since $f_0 \in L^\infty$, we have $f$ bounded in a short time interval. The long time boundedness has been derived by \cite[Theorem 5]{desvillettes2000spatiallyI} for hard potentials, by \cite[Proposition 4]{villani1998spatially} and by \cite[Theorem 3.7]{silvestre2017upper}. For very soft potentials, notice that since the initial Fisher information $I(f_0)= \int_{\R^3} |\nabla \log f_0|^2 f_0 \ud v$ has been assumed to be bounded, we have $\sup_{t \geq 0} I(f_t) \leq I(f_0)\leq C$. Then by the Sobolev embedding theorem we have $\|f\|_{L^3} \leq C$. By Lemma \ref{moment} below we also have the high-order moment estimate $\|f\|_{L_m^1} \leq C(1+t)$. The interpolation inequality implies that $\| f\|_{L_n^p}$ is finite for any finite time horizon $[0,T]$, where $p \in (1,3)$ and $n=m(3-p)/2$. By \cite[Theorem 2.11]{ji2024dissipation}, as long as $m>6$, we can choose $p \in (3/2,3)$ appropriately such that $\|f\|_{L^\infty} \leq C$ for large time. Combining these arguments, we have $f$ uniformly bounded.
\end{rmk}

Recall the standard equivalent conditions of convergence in the Wasserstein-$2$ distance:
\begin{align*}
    W_2(F_{N,k},f^{\otimes k}) \rightarrow 0 \Longleftrightarrow &\, F_{N,k} \text{ converges weakly to } f^{\otimes k}\\
    &\, \text{and } \int_{\R^{3k}} F_{N,k}\sum_{i=1}^k |v^i|^2 \ud v^{[k]} \longrightarrow \int_{\R^{3k}} f^{\otimes k} \sum_{i=1}^k |v^i|^2 \ud v^{[k]},
\end{align*}
which we refer to the classical textbooks by Villani \cite[Theorem 7.12]{villani2021topics} or \cite[Theorem 6.9]{villani2008optimal}. Combining with the conservation law of the kinetic energy for the master equation \eqref{master} and the Landau equation \eqref{compactlandau}, and the fact that $F_{N,k}(0,\cdot)=f_0^{\otimes k}$, we deduce the mean-field convergence in the Wasserstein-$2$ distance.

\begin{cor}[Convergence in Wasserstein-$2$ Distance]\label{cor_poc_wass}
    Under the same assumptions as in Theorem \ref{thm:poc}, for any $T>0$ and $k \geq 1$, we have
    \begin{equation*}
        W_2(F_{N,k},f^{\otimes k}) \rightarrow 0
    \end{equation*}
    as $N \rightarrow \infty$ on $[0,T]$.
\end{cor}

Beyond the chaotic property given by the weak convergence, a stronger notion of chaos called entropic chaos was introduced in \cite{carlen2010entropy} and studied in subsequent works \cite{mischler2013kac,hauray2014kac,carrapatoso2015quantitative,carrapatoso2016propagation}. Roughly speaking, entropic chaos ensures that the normalized entropy is asymptotically conserved in the mean-field limit. Using the standard entropy dissipation formula (see for instance \cite{villani1998new,desvillettes2015entropy,carrapatoso2017estimates}) and the lower semi-continuity, we establish the entropic chaos for Kac's particle system. For any probability measure  $g \in \mathcal{P}({\mathbb{R}^{k}})$, its  entropy is defined as 
\[
H(g) = \begin{cases} 
\int_{\mathbb{R}^{k}} g \log g  \ud v^{[k]}, & \mbox{if} \,\, g \mbox{ is absolutely continuous w.r.t. the Lebesgue measure};  \\
+ \infty, & \mbox{otherwise}. 
\end{cases}
\]
Moreover, as done by some convexity argument in Fournier-Hauray-Mischler \cite{fournier2014propagation}, we deduce the strong mean-field convergence in $L^1$. 

\begin{thm}[Entropic Chaos and Convergence in $L^1$]\label{thm:entropic}
    Under the same assumptions as in Theorem \ref{thm:poc}, for any $T>0$, we have the entropic chaos:
    \begin{equation*}
        \lim_{N \rightarrow \infty} \frac{1}{N}H(F_N)=H(f)
    \end{equation*}
    on $[0,T]$. Consequently, we have strong $L^1$ propagation of chaos: for any $k \geq 1$,
    \begin{equation*}
        F_{N,k} \rightarrow f^{\otimes k} \text{ strongly in } L^1(\R^{3k})
    \end{equation*}
    as $N \rightarrow \infty$ on $[0,T]$.
\end{thm}

\begin{rmk}
    A stronger notion of chaos called Fisher information chaos was introduced in \cite{hauray2014kac} similarly as the entropic chaos. It is well known that Fisher information is lower semi-continuous with respect to the weak convergence. If we can deduce that the Fisher information dissipation functional is also lower semi-continuous, we can follow the same strategy as in Section \ref{entropicchaos} to prove the Fisher information chaos. The explicit Fisher information dissipation formula is derived in \cite{guillen2023landau}.
\end{rmk}

The main new feature of our main results is that they have been successfully applied to Kac's particle system to approximate the Landau equation in the regime of soft potentials for the first time. In particular, our results unlock the particle approximation problem for the Landau equation with very soft potentials, especially the most important case of Coulomb interactions. The major weapons of which we have taken advantage, including the duality approach and the key functional estimates, have seen great effectiveness of treating Coulomb-like singular interacting kernels. Nevertheless, we state our main results for the entire range of power-law potentials $\gamma \in [-3,1]$ in this article for the sake of completeness, showing that our weapons possess the full power to treat all kinds of pairwise interactions.

All our main results, including Theorem \ref{thm:poc}, Corollary \ref{cor_poc_wass} and Theorem \ref{thm:entropic} on propagation of chaos in the classical Kac's sense, in the Wasserstein metric sense and in the entropic sense, are {\em qualitative} now. But the duality method we apply in this article has the potential to reach a convergence rate as in \cite{bresch2024duality}. We are content with the first qualitative mean-field limit result towards the Landau equation in the regime of very soft potentials, especially with the Coulomb interactions, and leave the quantitative propagation of chaos for the future study. 

\subsection{Methodology}\label{subsection:methodology}

In this subsection, we give a quick sketch of our main methods for proving Theorem \ref{thm:poc} and summarize the main novelties of this article. For simplicity, we only consider the case of Coulomb interactions here, which is the major target of our results. The cases of other potentials do not present further essential difficulties.

Compared to the classical theory of mean-field limit for first-order and second-order systems (related literature is reviewed in Subsection \ref{subsection:literature}), the rigorous derivation of the Landau equation from many-particle systems has been considerably more challenging and of high importance in physics. When starting from particle systems such as \eqref{SDE} or \eqref{SDE_new_2}, the counterpart of the diffusion coefficient $\sigma$ in \eqref{cannSDE} below becomes a matrix that is not only distribution-dependent but also degenerate. This significantly complicates the passage from the particle system to the Landau equation. In particular, the traditional method which relies on the analysis of the BBGKY hierarchy solved by the marginal distributions seems to be difficult to attack the Landau master equation, since the dissipation cannot well control the next-order derivative.

This essential difficulty enlightens us to consider the duality approach for the theory of mean-field limit, which was recently introduced by Bresch-Duerinckx-Jabin \cite{bresch2024duality} for canonical first-order and second-order systems towards the McKean-Vlasov limit equations. Roughly speaking, instead of directly considering the solution $F_N$ of the \textit{forward} Kolmogorov equation \eqref{master} starting from the factorized density $f_0^{\otimes N}$, we turn to the solution $\Phi_N$ of the \textit{backward} dual Kolmogorov equation ending at the $U$-statistics of the test function $$(C_N^k)^{-1} \sum_{1 \leq i_1 < \cdots < i_k \leq N} \varphi(v^{i_1}) \cdots \varphi(v^{i_k})$$ with fixed $k$ and any test function $\varphi$. Then, up to some small error term, propagation of Kac's chaos can be shown to be equivalent to the asymptotic vanishing of the integral:
\begin{equation}\label{vanishingintegral}
    \int_{\R^{3N}} N V_f(v^1,v^2) f^{\otimes N} \Phi_N \ud v^{[N]},
\end{equation}
where $V_f$ is some test function that will be explicitly computed later. We only emphasize now that $V_f$ has zero expectation with respect to $f$ on each variable.

Our first key observation is that, due to the special divergence structure of the Landau master equation \eqref{master}, the dual solution $\Phi_N$ coincides exactly with the time reverse of a solution of the master equation, with the initial data being the final data of $\Phi_N$. By examining the construction process of the weak solution $F_N$ and applying the auxiliary function method, we also establish the following parabolic maximum principle for the Landau master equation:
\begin{equation*}
    \|F_N\|_{L^\infty([0,\infty);L^\infty(\R^{3N}))} \leq \|F_N(0,\cdot) \|_{L^\infty(\R^{3N})},
\end{equation*}
which provides us with a nice $L^\infty$ \textit{a priori} estimate for $\Phi_N$ by the previous observation:
\begin{equation*}
    \|\Phi_N\|_{L^\infty((-\infty,T];L^\infty(\R^{3N}))} \leq \|\varphi\|_{L^\infty(\R^3)}^k.
\end{equation*}

Next, in order to simplify the target integral \eqref{vanishingintegral}, we perform the cluster expansion formula on the dual solution $\Phi_N$ to decompose it into the sum of local correlation functions $C_{N,n}$:
\begin{equation*}
    \Phi_N(v^{[N]})=\sum_{n=0}^N \sum_{\{i_1,\cdots,i_n\} \subset [N]} C_{N,n}(v^{i_1}, \cdots, v^{i_n}).
\end{equation*}
The key property of the correlation function $C_{N,n}$ is that it has zero expectation with respect to $f$ on each variable. Applying the above formula to the integral \eqref{vanishingintegral}, it suffices to check the following two conditions:
\begin{itemize}
    \item $V_f \in L^1([0,T];L^2 (f^{\otimes 2}\ud v \ud w));$
    \item $N C_{N,2} \rightharpoonup^\ast 0 \text{ in } L^\infty([0,T];L^2(f^{\otimes 2}\ud v \ud w))$.
\end{itemize}

The first square-integrability condition gathers all the difficulties of functional estimates in this article, since the test function $V_f$ performs simultaneously singularity caused by the Coulomb interactions and the unboundedness related to the Fokker-Planck quantities $\nabla \log f$ and $\nabla^2 \log f$. We divide the proof into two steps. The first step contains estimates of Fokker-Planck type, which reduce the problem to a second-order Fisher information estimate. The major novelty of this part is an extended version of the second-order commutator estimate in Nguyen-Rosenzweig-Serfaty \cite{nguyen2022mean}. We have to control the double integral
\begin{equation*}
    \iint_{\R^6}  \frac{|\nabla \log f(v)-\nabla \log f(w)|^2}{|v-w|^4} f(v)f(w) \ud v \ud w,
\end{equation*}
which has the same form of the integral that appears in \cite[Theorem 3]{bresch2024duality}. The Coulomb singularity goes beyond the locally $L^2$ assumption in \cite{bresch2024duality} and prevents the singular term $|v-w|^{-4}$ from being locally integrable. Moreover, due to the lack of a $W^{1,\infty}$ bound of the vector field $\nabla \log f$, we cannot rigorously integrate by parts as in \cite{nguyen2022mean} or directly apply the mean-value inequality as in \cite{bresch2024duality}. Instead, we expand the square and combine the terms with the density $f(v)f(w)$, and transform the whole into some squares of difference of vector fields:
\begin{align*}
    &\, |\nabla \log f(v)-\nabla \log f(w)|^2 f(v)f(w)\\
    \lesssim &\, \Big( |\nabla \sqrt{f}(v)|^2-|\nabla \sqrt{f}(w)|^2 \Big)^2+|f(v)-f(w)|^2+|\nabla f(v)-\nabla f(w)|^2.
\end{align*}
Insert this inequality into the double integral, and we are able to apply the integral expression for the fractional Sobolev norm
\begin{equation*}
    \|g\|_{\Dot{H}^{1/2}(\R^3)}^2 \sim \iint_{\R^6} \frac{|g(v)-g(w)|^2}{|v-w|^4} \ud v \ud w.
\end{equation*}
This would reduce the target integral to the Sobolev form $\|\sqrt{f}\|_{H^2}$ and thus the second-order Fisher information. The second step, being relatively standard, is to estimate the second-order Fisher information by propagating the weighted Fisher information along the Landau equation and applying the diffusive structure of the Landau equation.

The second weak-$\ast$ convergence condition follows from hierarchy estimates for the correlation functions. Thanks to the previously stated \textit{a priori} $L^\infty$ estimate for $\Phi_N$ and the orthogonality of the correlation functions, it is straightforward to obtain the \textit{a priori} boundedness condition for $N^{n/2} C_{N,n}$ under a Hilbert norm, which allows us to extract a weakly-$\ast$ convergent subsequence by the Banach-Alaoglu theorem. It suffices to check the uniqueness of the weak limit $\bar C_n$, and $\bar C_2=0$ follows immediately from the existence of trivial solutions to the limit hierarchy.

To this end, we write down the explicit hierarchy solved by $C_{N,n}$ and pass to the weak-$\ast$ limit along the convergent subsequence. This would lead to the limit hierarchy solved by the weak limit $\bar C_n$. Compared to the classical BBGKY hierarchy, this new limit hierarchy has two significant advantages, which make the duality approach effective in attacking the Landau equation and overcoming the degeneracy of the dissipation. Firstly, the limit hierarchy does not contain derivatives of higher-order terms, since they have all vanished when passing to the $N$-limit. Secondly, although the energy dissipation term possesses some degeneracy, the other terms involving derivatives also degenerate in the same rate, which again allows us to control them with the dissipation term by Cauchy-Schwarz inequality. The desired uniqueness result is thus straightforward from the energy method.

\subsection{Related Literature}\label{subsection:literature}

Recent progress on the mean-field limit for first-order systems has primarily focused on interacting particle systems described in the following canonical form:
\begin{equation}\label{cannSDE}
\ud X_t^i = \frac{1}{N} \sum_{j \ne i}^N K(X_t^i - X_t^j) \ud t + \sqrt{2 \sigma} \ud B_t^i, \quad i = 1, 2, \cdots, N,
\end{equation}
where particles $\{X_t^i\}$ evolve in the underlying space given by the whole space $\mathbb{R}^d$ or the torus $\mathbb{T}^d$, with various choices of the singular two-body interaction kernels $K$ and independent standard $d$-dimensional Brownian motions $\{B_t^i\}_{1 \leq i \leq N}$. The diffusion coefficient $\sigma$ is typically some non-negative constant.

In the 1980s, Osada \cite{osada1986propagation} established propagation of chaos for the viscous point vortex model, which corresponds to \eqref{cannSDE} with $K(x)$ given by the Biot-Savart kernel in $\mathbb{R}^2$, $K(x)= \frac{1}{2 \pi} \frac{x^\perp}{|x|^2}$, under the assumption of largeness on $\sigma$. Later, Fournier-Hauray-Mischler \cite{fournier2014propagation} obtained an entropic propagation of chaos result for the same model with arbitrary positive $\sigma$, using the entropy methods and compactness arguments. However, both results are qualitative in nature.

The relative entropy method for establishing quantitative propagation of chaos was introduced by Jabin and Wang: first for second-order systems with bounded kernels \cite{jabin2016mean}, and later for general first-order systems with singular kernels in $W^{-1,\infty}$ \cite{jabin2018quantitative}, including the 2D point vortex approximation of the Navier-Stokes equation on the torus. Building upon this framework, several notable recent developments include: Guillin-Le Bris-Monmarché \cite{guillin2024uniform}, who proved uniform-in-time propagation of chaos for the viscous vortex model on the torus using the logarithmic Sobolev inequality (LSI) for the limit density; Feng-Wang \cite{feng2023quantitative,feng2024quantitative}, who obtained quantitative propagation of chaos for 2D point vortex models on the whole space, possibly under general circulation conditions; and also Huang \cite{huang2023entropy}, who obtained results in the flavor of {\em strengthening chaos}, showing how weakly chaotic initial data can evolve into entropic chaos at positive times. 

For interacting particle systems in the deterministic setting ($\sigma = 0$), Serfaty \cite{serfaty2020mean} introduced the modulated energy method to study the mean-field limit for Coulomb and Riesz flows. By defining and estimating a Coulomb/Riesz-based distance between the empirical measure and the limit density, which was referred to the modulated energy, Serfaty established quantitative convergence rates of the empirical measure (tested against smooth functions) for Coulomb and super-Coulomb interactions in the full Euclidean space without additive noise. This method was later extended to general Riesz-type singular flows by Nguyen-Rosenzweig-Serfaty \cite{nguyen2022mean}, and further to global-in-time results by Rosenzweig-Serfaty \cite{rosenzweig2023global}. Combining the modulated energy with the viscosity-weighted relative entropy functional, Bresch-Jabin-Wang \cite{bresch2019modulated, bresch2019mean,bresch2023mean} developed the modulated free energy method to handle repulsive and attractive singular kernels with additive noise, notably including the two-dimensional Patlak-Keller-Segel model. See also \cite{cai2024propagation} for a recent result on propagation of chaos for the 2D log gas with additive noise on the whole space. De Courcel-Rosenzweig-Serfaty \cite{de2023sharp, de2023attractive} also extended the modulated free energy method to the periodic Riesz-type flows and the attractive log gas model, establishing uniform-in-time propagation of chaos results.

We also mention the recent work of Lacker \cite{lacker2021hierarchies}, where the focus is on the local relative entropy of order $k$.  By leveraging the BBGKY hierarchy and a carefully constructed iterative scheme, Lacker established the optimal convergence rate for the relative entropy between the $k$-particle marginals and the factorized law $f_t^{\otimes k}$, in the setting of weakly interacting diffusion with bounded interaction kernels. Building on the analysis developed in \cite{jabin2018quantitative}, Wang \cite{wang2024sharp} extended the sharp convergence rate in $N$ to the two-dimensional Navier-Stokes system on the torus, under the assumption of high viscosity. Its counterpart on the whole space setting has been obtained by Feng-Wang \cite{feng2024quantitative}. In a related direction, Bresch-Jabin-Soler \cite{bresch2022new} used the BBGKY hierarchy in combination with compactness arguments to derive the mean-field limit for second-order singular Vlasov-Fokker-Planck systems, particularly in the two-dimensional Coulomb interaction case. Finally, we highlight the new approach introduced in Bresch-Duerinckx-Jabin \cite{bresch2024duality}, which is likewise based on a hierarchy estimate, but employs a novel concept of dual cumulants to derive qualitative and quantitative estimates. 

\subsection{Structure of the Article}

The rest of this article is organized as follows. In Section \ref{dualityapproach} we establish the basic framework of the duality approach, defining the backward dual solution and performing the cluster expansion formula. This approach reduces our first main result Theorem \ref{thm:poc} to two steps, which are presented respectively in Proposition \ref{squareintegrable} and Proposition \ref{weakstarconvergence}. In Section \ref{keyfunctional} we prove the key functional estimate Proposition \ref{squareintegrable}. In Section \ref{sec:hierarchy} we proceed with the duality approach and prove the hierarchy estimate Proposition \ref{weakstarconvergence}, completing the proof of Theorem \ref{thm:poc}. In Section \ref{entropicchaos} we prove our second main result Theorem \ref{thm:entropic}, establishing the entropic chaos and propagation of chaos in $L^1$. We present the classical proof of the well-posedness of the Landau master equation in Appendix \ref{appendixB} for completeness.

\section{Duality Approach}\label{dualityapproach}

Motivated by the recent work of Bresch-Duerinckx-Jabin \cite{bresch2024duality}, who introduced a powerful duality framework to study the mean-field limit for singular interacting particle systems, we develop and adapt this approach in a novel way to tackle Kac’s program for the Landau equation. While our analysis shares certain structural similarities with \cite{bresch2024duality}, our contribution goes beyond a direct application: we provide a fully self-contained proof tailored to the specific challenges of the Landau setting, and introduce several new ideas and computations that are essential for handling the delicate structure of the Landau collision operator. 

\subsection{Duality Formulation}

Recall that our Landau master equation or forward Kolmogorov equation of Kac's particle system is given by
\begin{equation*}
    \p_t F_N=\frac{1}{2N}\sum_{i,j=1}^N (\nabla_{v^i}-\nabla_{v^j}) \cdot \Big( a(v^i-v^j) \cdot (\nabla_{v^i}-\nabla_{v^j})F_N \Big),
\end{equation*}
with the initial data $F_N(0,\cdot)=f_0^{\otimes N}$. Here $F_N$ is the unique bounded weak solution to the above equation with regularity $F_N \in L^\infty([0,\infty);L^1 \cap L^\infty(\R^{3N}))$, which exists by the classical approximation strategy.

Now for any fixed final time $T>0$ and any fixed local number $k \in \mathbb{Z}^+$, and for any test function $\varphi \in C_c^\infty(\R^3)$ as in Theorem \ref{thm:poc}, we construct the corresponding backward dual Kolmogorov equation:
\begin{equation*}
    \p_t \Phi_N=-\frac{1}{2N}\sum_{i,j=1}^N (\nabla_{v^i}-\nabla_{v^j}) \cdot \Big( a(v^i-v^j) \cdot (\nabla_{v^i}-\nabla_{v^j})\Phi_N \Big),
\end{equation*}
with final data given by the $k$-th order $U$-statistics of $\varphi$:
\begin{equation*}
    \Phi_N(T,v^{[N]})=(C_N^k)^{-1} \sum_{1 \leq i_1< \cdots <i_k \leq N} \varphi(v^{i_1}) \cdots \varphi(v^{i_k}), 
\end{equation*}
where $C_N^k$  denotes the binomial coefficient $C_N^k = \binom{N}{k} = \frac{N!}{k!(N-k)!}. $

Thanks to the special divergence structure of the Landau master equation \eqref{master}, we observe that $\Phi_N$ is exactly the time reverse of the unique bounded weak solution $\Tilde{F}_N$ of the Landau master equation:
\begin{equation*}
    \Phi_N(t, \cdot)=\Tilde{F}_N (T-t, \cdot),
\end{equation*}
where $\Tilde{F}_N$ is given by the Landau master equation with initial data $\Tilde{F}_N(0,\cdot)=\Phi_N(T, \cdot)$. Therefore, the existence, uniqueness and regularity of the bounded weak solution $\Phi_N$ to this backward dual equation is straightforward. 

Using the regularized master equation in Appendix \ref{appendixB} and integrating by parts, it is straightforward to check that
\begin{equation*}
    \frac{\ud}{\ud t} \int_{\R^{3N}} F_{N,\eps}(t) \cdot \Phi_{N,\eps}(t) \ud v^{[N]}=0.
\end{equation*}
Comparing the integrals at time $t=0$ and $t=T$ and passing to the limit $\eps \rightarrow 0$, and using the symmetry property of $F_N$, we write that
\begin{equation*}
    \int_{\R^{3k}} F_{N,k}(T) \varphi^{\otimes k} \ud v^{[k]}=\int_{\R^{3N}} f_0^{\otimes N} \Phi_N(0) \ud v^{[N]}.
\end{equation*}
Moreover, by the symmetry property and Newton-Leibniz formula, we can express the integral on the right-hand side with
\begin{equation*}
    \int_{\R^{3k}} f(T)^{\otimes k} \varphi^{\otimes k} \ud v^{[k]}-\int_0^T \Big( \frac{\ud}{\ud t} \int_{\R^{3N}} f_t^{\otimes N} \Phi_N(t) \ud v^{[N]} \Big) \ud t. 
\end{equation*}
Hence we conclude that
\begin{equation*}
    \int_{\R^{3k}} \Big(f(T)^{\otimes k}-F_{N,k}(T) \Big) \varphi^{\otimes k} \ud v^{[k]}=\int_0^T \Big( \frac{\ud}{\ud t} \int_{\R^{3N}} f_t^{\otimes N} \Phi_N(t) \ud v^{[N]} \Big) \ud t. 
\end{equation*}
To prove propagation of chaos in Kac's sense (Theorem \ref{thm:poc}), it suffices to prove the asymptotic vanishing of the integral on the right-hand side.

Denote by $S$ the time derivative of the integral inside the bracket on the right-hand side. Applying the weak-strong argument and using the backward dual Kolmogorov equation and the Landau equation, we compute explicitly the time derivative that
\begin{align*}
    S = \int_{\R^{3N}} &\, \sum_{i=1}^N \nabla_{v^i} \cdot \Big( a \ast f(v^i) \cdot \nabla_{v^i} f^{\otimes N}-b \ast f(v^i) \, f^{\otimes N} \Big) \Phi_N \\
    &\, -\frac{1}{2N} \sum_{i,j=1}^N (\nabla_{v^i}-\nabla_{v^j}) \cdot \Big( a(v^i-v^j) \cdot (\nabla_{v^i}-\nabla_{v^j}) \Phi_N \Big) f^{\otimes N} \ud v^{[N]}.
\end{align*}
Using integration by parts and the symmetry property of $\Phi_N$, we rewrite it into
\begin{align*}
    S = \int_{\R^{3N}} &\, \sum_{i=1}^N \Big( a \ast f(v^i):\frac{\nabla^2 f}{f}(v^i)-c \ast f(v^i) \Big) f^{\otimes N} \Phi_N\\
    -\frac{1}{N} &\, \sum_{i,j=1}^N b(v^i-v^j) \cdot (\nabla \log f(v^i)-\nabla \log f(v^j)) f^{\otimes N} \Phi_N\\
    -\frac{1}{N} &\, \sum_{i,j=1}^N a(v^i-v^j): \Big( \frac{\nabla^2 f}{f}(v^i)-\nabla \log f(v^i) \otimes \nabla \log f(v^j) \Big) f^{\otimes N} \Phi_N \ud v^{[N]}\\
    = \frac{1}{N} &\, \sum_{i,j=1}^N \int_{\R^{3N}} V_f(v^i,v^j) f^{\otimes N} \Phi_N \ud v^{[N]}\\
    = (N&\,-1) \int_{\R^{3N}} V_f(v^1,v^2) f^{\otimes N} \Phi_N \ud v^{[N]}+\int_{\R^{3N}} V_f(v^1,v^1) f^{\otimes N} \Phi_N \ud v^{[N]},
\end{align*}
where the test function $V_f$ is defined by
\begin{equation*}
    \begin{aligned}
        V_f(v,w)=&\, -a(v-w): \Big( \frac{\nabla^2 f}{f}(v)-\nabla \log f(v) \otimes \nabla \log f(w) \Big)\\
        &\, -b(v-w) \cdot (\nabla \log f(v)-\nabla \log f(w))\\
        &\, +a \ast f(v):\frac{\nabla^2 f}{f}(v)-c \ast f(v).
    \end{aligned}
\end{equation*}
It is straightforward to check that the test function $V_f(v,w)$ satisfies the two-side cancellation condition:
\begin{equation*}
    \int_{\R^3} V_f(v,w) f(v) \ud v=0, \quad \int_{\R^3} V_f(v,w) f(w) \ud w=0.
\end{equation*}
\begin{rmk}
    This cancellation condition seems to be related to the same condition in \cite{jabin2018quantitative}. In fact, if we compute the time derivative formula of the relative entropy $H(F_N|f^{\otimes N})=\int_{\R^{3N}} F_N \log \frac{F_N}{f^{\otimes N}} \ud v^{[N]}$ as in \cite{jabin2018quantitative}, we will obtain a non-positive dissipation term given by the opposite of the relative version of the entropy dissipation functional, together with an error term given by the expectation of some test function with respect to $F_N$. The test function will be exactly $V_f$. To be explicit, we compute that
    \begin{align*}
        \frac{\ud}{\ud t} H(F_N|f^{\otimes N})=&\, -\frac{1}{N}\sum_{i,j=1}^N \int_{\R^{3N}} a(v^i-v^j): \Big( \nabla_{v^i} \log \frac{F_N}{f^{\otimes N}}-\nabla_{v^j} \log \frac{F_N}{f^{\otimes N}} \Big)^{\otimes 2} F_N \ud v^{[N]}\\
        &\, +\frac{1}{N} \sum_{i,j=1}^N \int_{\R^{3N}} V_f(v^i,v^j) F_N \ud v^{[N]}.
    \end{align*}
    However, to control the smallness of the error term, the classical relative entropy method requires estimates on exponential integrals in the form of
    \begin{equation*}
        \int_{\R^3} \exp(\lambda \sup_w |V_f(v,w)|) f(v) \ud v,
    \end{equation*}
    which is much stronger and more difficult to obtain than the square-integrability condition (Proposition \ref{squareintegrable} below) in this article.
\end{rmk}
Therefore, propagation of chaos in Kac's sense has been reformulated to the following statements:
\begin{equation*}
    \lim_{N \rightarrow \infty} N \int_0^T \int_{\R^{3N}} V_f(v^1,v^2) f^{\otimes N} \Phi_N \ud v^{[N]} \ud t=0,
\end{equation*}
\begin{equation*}
    \lim_{N \rightarrow \infty} \int_0^T \int_{\R^{3N}} V_f(v^1,v^1) f^{\otimes N} \Phi_N \ud v^{[N]} \ud t=0.
\end{equation*}

\subsection{Cluster Expansion}

Following the strategy of \cite{bresch2024duality}, we develop the cluster expansion formula of the backward dual solution $\Phi_N$ using the so-called linear correlation functions, or cumulant functions, $C_{N,n}$. The brief idea is to expand $\Phi_N$ into the sum of orthogonal series:
\begin{equation*}
    \Phi_N(v^{[N]})=\sum_{n=0}^N \sum_{\{i_1, \cdots, i_n \} \subset [N]} C_{N,n}(v^{i_1}, \cdots, v^{i_n}) 
\end{equation*}
with $C_{N,n}$ having zero expectation with respect to $f$ on each component. This expansion will simplify the two integrals in the end of the previous subsection. To this end, first we define the weighted marginals
\begin{equation*}
    M_{N,n} (v^{[n]})=\int_{\R^{3(N-n)}} \Phi_N(v^{[N]}) f^{\otimes (N-n)} (v^{n+1}, \cdots, v^{N}) \ud v^{n+1} \cdots \ud v^{N}. 
\end{equation*}
The cumulant functions are defined by an alternating series of the weighted marginals:
\begin{equation*}
    C_{N,n} (v^{[n]})= \sum_{m=0}^n (-1)^{n-m} \sum_{\sigma \in P_m^n} M_{N,m} (v^\sigma),
\end{equation*}
where the set $P_m^n$ contains all subsets of $[n]=\lbrace 1, \cdots, n \rbrace$ with exactly $m$ elements, and the notation $v^\sigma$ stands for $(v^{i_1}, \cdots, v^{i_m})$ with $\sigma=\lbrace i_1, \cdots, i_m \rbrace \in P_m^n$. For $n=0$ the two quantities are constants. One may easily check that this correlation function $C_{N,n}$ has zero expectation with respect to $f$ on each component:
\begin{equation*}
    \int_{\R^3} C_{N,n}(v^{[n]}) f(v^i) \ud v^i=0,
\end{equation*}
for each $1 \leq i \leq n$. Furthermore, the weighted marginal $M_{N,n}$ and the correlation function $C_{N,n}$ are both symmetric or exchangeable with respect to its $n$ variables. The cluster expansion of the weighted marginal is a version of an inversion formula from the correlation functions $C_{N,m}$ to obtain $M_{N,n}$:
\begin{equation*}
    M_{N,n}(v^{[n]})=\sum_{m=0}^n \sum_{\sigma \in P_m^n} C_{N,m}(v^\sigma).
\end{equation*}
In this sense, our duality formulation can be further simplified. Observe that
\begin{align*}
      &\,  N \int_0^T \int_{\R^{3N}} V_f(v^1,v^2) f^{\otimes N} \Phi_N \ud v^{[N]} \ud t \\
      = &\,  N \int_0^T \iint_{\R^6} V_f(v,w) f(v)f(w) M_{N,2}(v,w) \ud v \ud w \ud t \\
      = &\, N \int_0^T \iint_{\R^6} V_f(v,w) f(v)f(w) C_{N,2}(v,w) \ud v \ud w \ud t .
\end{align*}
Similarly, noticing that $\int_{\R^3} V_f(v,v) f(v) \ud v=0$, we have
\begin{align*}
    &\, \int_0^T \int_{\R^{3N}} V_f(v^1,v^1) f^{\otimes N} \Phi_N \ud v^{[N]} \ud t\\
    =&\, \int_0^T \int_{\R^3} V_f(v,v) f(v) M_{N,1}(v) \ud v \ud t\\
    =&\, \int_0^T \int_{\R^3} V_f(v,v) f(v) C_{N,1}(v) \ud v \ud t.
\end{align*}
From now on, we reduce the problem of propagation of Kac's chaos to the next two propositions, which concern the square-integrability of the test function $V_f$ and the weak-$\ast$ convergence to $0$ of the rescaled correlation functions $N C_{N,2}$ and $C_{N,1}$.

\begin{prop}[Square-integrability]\label{squareintegrable}
    The test function $V_f$ belongs to the functional space $L^1([0,T];L^2(f^{\otimes 2} \ud v \ud w))$.
    \begin{align*}
        \int_0^T \iint_{\R^6} |V_f(v,w)|^2 f(v)f(w) \ud v \ud w \ud t <\infty.
    \end{align*}
    Also we have
    \begin{equation*}
        \int_0^T \int_{\R^3} |V_f(v,v)|^2 f(v) \ud v \ud t<\infty.
    \end{equation*}
\end{prop}

\begin{prop}[Weak-$\ast$ Convergence]\label{weakstarconvergence}
    The rescaled correlation functions $N C_{N,2}$ and $C_{N,1}$ converge weakly-$\ast$ to $0$ in the weighted functional space.
    \begin{equation*}
        N C_{N,2} \rightharpoonup^\ast 0 \text{ in } L^\infty([0,T];L^2(f^{\otimes 2} \ud v \ud w)),
    \end{equation*}
    \begin{equation*}
        C_{N,1} \rightharpoonup^\ast  0 \text{ in } L^\infty([0,T];L^2(f \ud v)).
    \end{equation*}
\end{prop}

The proof of Proposition \ref{squareintegrable} and Proposition \ref{weakstarconvergence} will occupy the next two sections respectively, thus completing the proof of our first main result Theorem \ref{thm:poc}. To conclude this section, we prove the following \textit{a priori} boundedness property of the rescaled correlation functions, which is an essential preparation for the derivation of the weak limit. This property is an important corollary of the \textit{a priori} $L^\infty$ estimate of the dual solution $\Phi_N$ given by the parabolic maximum principle.

\begin{lemma}[\textit{A Priori} Boundedness]\label{aprioribounded}
    The correlation functions $C_{N,n}$ satisfy the following \textit{a priori} boundedness condition
    \begin{equation*}
        \| C_{N,n} \|_{L^\infty([0,T];L^2(f^{\otimes n} \ud v^{[n]}))} \leq (C_N^n)^{-\frac{1}{2}} \|\varphi\|_{L^\infty(\R^3)}^k.      
    \end{equation*}
\end{lemma}
\begin{proof}
    Recall the cluster expansion formula for $M_{N,N}=\Phi_N$:
    \begin{equation*}
        \Phi_N(v^{[N]})=\sum_{n=0}^N \sum_{\sigma \in P_n^N}  C_{N,n}(v^\sigma).
    \end{equation*}
    Take squares on both sides and integrate on $\R^{3N}$ with respect to $f^{\otimes N}$:
    \begin{equation*}
        \int_{\R^{3N}} |\Phi_N|^2 f^{\otimes N} \ud v^{[N]}=\int_{\R^{3N}} \Big| \sum_{n=0}^N \sum_{\sigma \in P_n^N}  C_{N,n}(v^\sigma) \Big|^2 f^{\otimes N} \ud v^{[N]}.
    \end{equation*}
    Expanding the square, we rewrite the right-hand side into the summation of terms in the form of
    \begin{equation*}
        \int_{\R^{3N}} C_{N,n}(v^\sigma) C_{N,m}(v^\tau) f^{\otimes N} \ud v^{[N]},
    \end{equation*}
    where $\sigma \in P_n^N$ and $\tau \in P_m^N$. Once $\sigma \neq \tau$, we have some index $i$ that appears exactly in one of the two sets $\sigma$ and $\tau$. Without loss of generality, let $i \in \sigma$. We compute the integral above by first integrating with respect to $v^i$, which leads to
    \begin{equation*}
        \int_{\R^{3(N-1)}} C_{N,m}(v^\tau) f^{\otimes (N-1)}(v^{[N]-\lbrace i \rbrace})\ud v^{[N]-\lbrace i \rbrace} \int_{\R^3} C_{N,n}(v^\sigma) f(v^i) \ud v^i.
    \end{equation*}
    Since the correlation function $C_{N,n}$ has zero expectation with respect to $f$ on each component, the last integral and thus the entire integral vanishes. Therefore, all cross terms that appear in the square expansion vanish, and we use the symmetry of $C_{N,n}$ to deduce that
    \begin{equation*}
        \int_{\R^{3N}} |\Phi_N|^2 f^{\otimes N} \ud v^{[N]}=\sum_{n=0}^N C_N^n \int_{\R^{3n}} |C_{N,n}|^2 f^{\otimes n} \ud v^{[n]}.
    \end{equation*}
    Moreover, using Proposition \ref{maximum}, we have the following \textit{a priori} $L^\infty$ bound of $\Phi_N$ from the parabolic maximum principle:
    \begin{equation*}
        \| \Phi_N \|_{L^\infty((-\infty, T];L^\infty(\R^{3N}))}=\| \tilde F_N \|_{L^\infty([0,\infty);L^\infty(\R^{3N}))} \leq \| \tilde F_N(0,\cdot) \|_{L^\infty(\R^{3N})} \leq \| \varphi \|_{L^\infty(\R^3)}^k.
    \end{equation*}
    This gives our desired \textit{a priori} bound.
\end{proof}

This lemma has already justified the convergence of $\|C_{N,1} \|_{L^\infty[0,T];L^2(f \ud v))}$ to $0$ in the second part of Proposition \ref{weakstarconvergence}. We only need to deal with the weak-$\ast$ convergence in the first part. By the Banach-Alaoglu theorem, since the functional space $L^\infty([0,T];L^2(f^{\otimes n} \ud v^{[n]}))$ is the dual space of $L^1([0,T];L^2(f^{\otimes n} \ud v^{[n]}))$, we can extract a weakly-$\ast$ convergent subsequence.

\begin{lemma}\label{weakconvergence}
    There exists a subsequence $N_l \rightarrow \infty$ such that for all $n \geq 0$,
    \begin{equation*}
        N_l^{\frac{n}{2}} C_{N_l,n} \rightharpoonup^\ast (n!)^{\frac{1}{2}} \bar C_n
    \end{equation*}
    in the functional space $L^\infty([0,T];L^2(f^{\otimes n} \ud v^{[n]}))$, with the weak limit $\bar C_n$ belongs to the same functional space with the norm
    \begin{equation*}
        \| \bar C_n \|_{L^\infty([0,T];L^2(f^{\otimes n} \ud v^{[n]}))} \leq \| \varphi \|_{L^\infty(\R^3)}^k.
    \end{equation*}
\end{lemma}
\begin{proof}
    We rescale the correlation functions by
    \begin{equation*}
        \bar C_{N,n}=(C_N^n)^{\frac{1}{2}} C_{N,n}.
    \end{equation*}
    By Lemma \ref{aprioribounded}, the sequence of rescaled correlation functions $\lbrace \bar C_{N,n} \rbrace_{N \geq 1}$ is uniformly bounded in $L^\infty([0,T];L^2(f^{\otimes n} \ud v^{[n]}))$. By the Banach-Alaoglu theorem, there exists a subsequence $N_l \rightarrow \infty$ such that $\bar C_{N_l,n} \rightharpoonup^\ast \bar C_n$ for some $\bar C_n \in L^\infty([0,T];L^2(f^{\otimes n} \ud v^{[n]}))$ as $l \rightarrow \infty$. The desired convergence follows immediately from the Stirling formula. By weak lower semi-continuity, we have
    \begin{equation*}
        \| \bar C_n \|_{L^\infty([0,T];L^2(f^{\otimes n} \ud v^{[n]}))} \leq \liminf_{l \rightarrow \infty} (C_{N_l}^n)^{\frac{1}{2}} \| C_{N_l,n} \|_{L^\infty([0,T];L^2(f^{\otimes n} \ud v^{[n]}))} \leq \| \varphi \|_{L^\infty(\R^3)}^k.
    \end{equation*}
\end{proof}
Therefore, to prove the remaining part of Proposition \ref{weakstarconvergence}, it suffices to prove the uniqueness of the weak limit and that $\bar C_2=0$. We will finish this proof in Section \ref{sec:hierarchy}.

\section{Key Functional Estimates}\label{keyfunctional}

This section is devoted to the key functional estimates that we use in the duality approach, mainly the square-integrability of the test function $V_f$ in Proposition \ref{squareintegrable}. To provide convenience for checking our main results for the Landau equation with general potentials, we shall present the complete estimates for the entire range of parameters $\gamma \in [-3,1]$. However, as stated in Subsection \ref{subsection:methodology}, we emphasize that the major novelty of our functional estimates lives in the regime of very soft potentials, especially the Coulomb interactions, which not only corresponds to the real physical picture, but also broadens the area of application of these estimates in further study.

\subsection{Preliminaries}\label{preliminaries}

In this subsection, we gather some classical estimates that are well known in the study of the Landau equation. These results include the well-posedness of the Landau equation, basic estimates of the coefficients, and propagation of elementary weighted norms. 

The first lemma concerns the ellipticity of the coefficient matrix that appears in the Landau equation, which has been standard in the literature. For detailed proofs, we refer to Desvillettes-Villani \cite[Proposition 4 (i)]{desvillettes2000spatiallyI} for hard potentials, Villani \cite[Proposition 1]{villani1998spatially} for Maxwellian molecules, and Silvestre \cite[proof of Lemma 3.1]{silvestre2017upper} for soft potentials.

\begin{lemma}[Ellipticity]\label{ellipticity}
    For any $\gamma \in [-3,1]$, assume that $f:\R^3 \rightarrow [0,\infty)$ satisfies
    \begin{equation*}
        \int_{\R^3} f(v) \ud v =1, \quad \int_{\R^3} |v|^2f(v) \ud v =3, \quad \int_{\R^3} f(v)|\log f(v)| \ud v \leq H,
    \end{equation*}
    for some constants $H$. Then we have the ellipticity estimate for the coefficient matrix
    \begin{equation*}
        \bar a(v)=a \ast f(v) \geq c_0\langle v \rangle^\gamma \uid,
    \end{equation*}
    for some constant $c_0=c_0(H)>0$.
\end{lemma}

\begin{rmk}
    The lemma above is stated in a general version. In this article, since we always consider the density $f$ with finite second-order moment (kinetic energy), one can replace the assumption of $f \in L\log L $ simply by the assumption of finite entropy $\int_{\R^3} f(v) 
    \log f(v) \ud v < \infty$. See the discussions in Appendix \ref{appendixB}.
\end{rmk}

The next lemma summarizes the $L^\infty$ estimates for the coefficients that appear in the Landau equation. Since the coefficients are given by explicit convolution formulas, these estimates are obtained easily from the classical Young inequality.

\begin{lemma}[Coefficient $L^\infty$ Estimates]\label{coefficient1}
    Assume that $f \in L^1 \cap L^\infty(\R^3)$. Then for any $\theta \in [0,3)$, we have the convolution bounds:
    \begin{equation*}
        \| | \cdot |^{-\theta} \ast f \|_{L^\infty} \leq C\|f\|_{L^1}^{\frac{3-\theta}{3}} \cdot \|f\|_{L^\infty}^{\frac{\theta}{3}}
    \end{equation*}
    for some $C=C(\theta)$. Specifically, for the cases of very soft potentials $\gamma \in (-3,0]$, we have the $L^\infty$ bound on the coefficient:
    \begin{equation*}
        \| \bar c \|_{L^\infty} \leq C\|f\|_{L^1}^{\frac{3+\gamma}{3}} \cdot \|f\|_{L^\infty}^{-\frac{\gamma}{3}}
    \end{equation*}
    for some $C=C(\gamma)$. Moreover, for the case of Coulomb interactions $\gamma=-3$, since $\bar c=-8\pi f$, the above bound also holds for $\gamma=-3$.
\end{lemma}

If we have higher-order weighted norms, we can go beyond the $L^\infty$ estimates and deduce point-wise estimates for the coefficients.

\begin{lemma}[Coefficient Point-wise Estimates, \cite{carrapatoso2017landau}]\label{coefficient2}
    For any $\theta \in [-2,0]$, assume that $f \in L_n^p(\R^3)$ for some $p>3$ and $n>3(p-1)>6$. Then we have the convolution bound:
    \begin{equation*}
        ||\cdot|^{\theta} \ast f(v)| \leq C \langle v \rangle^{\theta} \|f\|_{L_n^p}
    \end{equation*}
    for some $C=C(\theta,n,p)$. For any $\theta \in [0,6]$, assume that $f \in L_6^1(\R^3)$. Then we have the convolution bound:
    \begin{equation*}
        ||\cdot|^\theta \ast f(v)| \leq C(\langle v \rangle^\theta+\|f\|_{L_\theta^1})
    \end{equation*}
    for some $C=C(\theta)$.
\end{lemma}

The following result gives the propagation of $L^1$ moments of all orders.

\begin{lemma}[Propagation of $L^1$ Moments, \cite{carrapatoso2017estimates}]\label{moment}
    Assume that $f_0 \in L_m^1 \cap L\log L(\R^3)$  for some $m>2$ satisfies the normalized condition \eqref{normalize}. Let $f:[0, \infty) \times \R^3 \rightarrow [0, \infty)$ be any weak solution to the Landau equation with initial data $f_0$. Then we have the propagation of $L^1$ moments.
    \begin{equation*}
        \int_{\R^3} \langle v \rangle^m f(t,v) \ud v \leq C(1+t),
    \end{equation*}
    for some $C=C(m,\|f_0\|_{L_m^1})$.
\end{lemma}

For the uniformly bounded solution $f$ considered in this article, we have by interpolation that for any $1 < p <\infty$ and $n < m$,
\begin{equation}\label{lp}
    \|f\|_{L^p} \leq C=C(p,\|f\|_{L^\infty}), \quad \|f\|_{L_n^p} \leq C(1+t)^{\frac{1}{p}}
\end{equation}
with $C=C(n,p,\|f\|_{L^\infty},\|f_0\|_{L_m^1})$. We also recall from Remark \ref{sketch} that $\|f\|_{L^\infty}$ can be controlled by initial quantities $m,I(f_0),\|f_0\|_{L_m^1},\|f_0\|_{L^\infty}$.

\subsection{Reduction to Second-order Fisher Information}


The structure of the test function $V_f$ is rather complicated, and estimating it directly by relating it to familiar quantities through the Landau equation is delicate. Instead, by employing purely functional inequalities of Fokker-Planck type, we can reduce the problem to estimating the second-order Fisher information for the Landau equation. All arguments in this subsection rely solely on functional analysis and do not require $f$ to be a solution of the Landau equation. Since the main difficulties stem from the singularity of Coulomb interactions rather than from the explicit form of the coefficients, these estimates may also find applications in the study of more general models. Additional remarks will be provided alongside the proofs below.

To proceed, we first recall two elementary but standard lemmas in the literature. These lemmas establish comparison relations between the weighted first- and second-order Fisher information, as well as among the different formulations of the weighted second-order Fisher information.

\begin{lemma}[Weighted Second-order Fisher Information]\label{fokkerplanck}
    For any $\theta \in \R$, the following functional inequality holds:
    \begin{equation*}
        \int \langle v \rangle^\theta |\nabla \log f(v)|^4 f(v) \leq 24 \int \langle v \rangle^\theta |\nabla^2 \log f(v)|^2 f(v)+16\theta^4 \|f\|_{L_{\theta-4}^1}.
    \end{equation*}
    In particular, when $\theta=0$, we have
    \begin{equation}\label{gamma}
        \int |\nabla \log f|^4 f \leq 24 \int |\nabla^2 \log f|^2 f.
    \end{equation}
    Moreover, we have
    \begin{equation}\label{gamma2}
        \int |\nabla^2 \sqrt{f}|^2 \leq 4 \int |\nabla^2 \log f|^2f, \quad \int \frac{|\nabla^2 f|^2}{f} \leq 50 \int |\nabla^2 \log f|^2 f.
    \end{equation}
\end{lemma}
\begin{proof}
    We use the trivial identity $(\nabla \log f)f=\nabla f$ and then do integrate by parts
    \begin{align*}
        &\, \int \langle v \rangle^\theta |\nabla \log f(v)|^4 f(v) = \int \langle v \rangle^\theta |\nabla \log f(v)|^2 \nabla \log f(v) \cdot \nabla f(v)\\
        =&\, -\int \langle v \rangle^\theta |\nabla \log f(v)|^2 (\Delta \log f(v) )f(v)-\int \theta \langle v \rangle^{\theta-2} |\nabla \log f(v)|^2 (v \cdot \nabla \log f(v)) f(v)\\
        &\, -2\int \langle v \rangle^\theta (\nabla \log f(v) \cdot \nabla^2 \log f(v) \cdot \nabla \log f(v)) f(v)\\
        =&\, \text{Term}_1+\text{Term}_2+\text{Term}_3.
    \end{align*}
    We bound those three terms respectively by the Cauchy-Schwarz inequality to introduce the weighted second-order Fisher information 
    \begin{align*}
        &\, \text{Term}_1 \leq \frac{1}{8} \int \langle v \rangle^\theta |\nabla \log f(v)|^4 f(v)+2\int \langle v \rangle^\theta |\nabla^2 \log f(v)|^2 f(v), \\
        &\, \text{Term}_2 \leq \frac{1}{8} \int \langle v \rangle^\theta |\nabla \log f(v)|^4 f(v) +2\theta^2 \int \langle v \rangle^{\theta-2} |\nabla \log f(v)|^2 f(v), \\
        &\, \text{Term}_3 \leq \frac{1}{4} \int \langle v \rangle^\theta |\nabla \log f(v)|^4 f(v) +4\int \langle v \rangle^\theta |\nabla^2 \log f(v)|^2f(v).
    \end{align*}
    Summing up the three inequalities and rearranging terms, we have
    \begin{equation*}
        \int \langle v \rangle^\theta |\nabla \log f(v)|^4 f(v) \leq 12 \int \langle v \rangle^\theta |\nabla^2 \log f(v)|^2 f(v)+4\theta^2 \int \langle v \rangle^{\theta-2} |\nabla \log f(v)|^2 f(v).
    \end{equation*}
    Next, we apply the Cauchy-Schwarz inequality to the weighted Fisher information.
    \begin{equation*}
        4\theta^2 \int \langle v \rangle^{\theta-2} |\nabla \log f(v)|^2 f(v) \leq \frac{1}{2} \int \langle v \rangle^{\theta} |\nabla \log f(v)|^4 f(v)+8\theta^4 \int \langle v \rangle^{\theta-4} f(v).
    \end{equation*}
    Inserting this estimate to the previous functional inequality, the desired result follows from rearranging the terms. The inequality \eqref{gamma} is straightforward. And \eqref{gamma2} follows from the following elementary identities
    \begin{equation*}
        \nabla^2 \sqrt{f}=\frac{1}{2}\frac{\nabla^2 f}{\sqrt{f}}-\frac{1}{4}\frac{\nabla f \otimes \nabla f}{f^{\frac{3}{2}}}=\frac{1}{2}\sqrt{f} \, \nabla^2 \log f+\frac{1}{4}\frac{\nabla f \otimes \nabla f}{f^{\frac{3}{2}}},
    \end{equation*}
    \begin{equation*}
        \frac{\nabla^2 f}{f}=\nabla^2 \log f+\nabla \log f \otimes \nabla \log f,
    \end{equation*}
    and the Cauchy-Schwarz inequality
    \begin{equation*}
        \int |\nabla^2 \sqrt{f}|^2 \leq \frac{1}{2}\int |\nabla^2 \log f|^2 f+\frac{1}{8} \int |\nabla \log f|^4 f \leq 4\int|\nabla^2 \log f|^2 f,
    \end{equation*}
    \begin{equation*}
        \int \frac{|\nabla^2 f|^2}{f} \leq 2\int |\nabla^2 \log f|^2 f+2\int |\nabla \log f|^4 f \leq 50 \int |\nabla^2 \log f|^2 f.
    \end{equation*}
    This concludes our desired estimates.
\end{proof}

\begin{lemma}[Weighted Fisher Information]\label{fokkerplanck2}
     For any $\theta \in \R$ and $\lambda \geq 0$, the following functional inequality holds:
    \begin{equation*}
        \int \langle v \rangle^\theta |\nabla \log f(v)|^2 f(v) \leq 6 \int \langle v \rangle^{\theta-\lambda} |\nabla^2 \log f(v)|^2 f(v)+(1+4\theta^4) \|f\|_{L_{\theta+\lambda}^1}.
    \end{equation*}
\end{lemma}
\begin{proof}
    We apply the Cauchy-Schwarz inequality to the weighted Fisher information 
    \begin{equation*}
        \int \langle v \rangle^\theta |\nabla \log f(v)|^2 f(v) \leq \frac{1}{4} \int \langle v \rangle^{\theta-\lambda} |\nabla \log f(v)|^4 f(v)+\int \langle v \rangle^{\theta+\lambda} f(v).
    \end{equation*}
    Combining with Lemma \ref{fokkerplanck}, we have
    \begin{equation*}
        \int \langle v \rangle^\theta |\nabla \log f(v)|^2 f(v) \leq 6 \int \langle v \rangle^{\theta-\lambda} |\nabla^2 \log f(v)|^2 f(v)+(1+4\theta^4) \|f\|_{L_{\theta+\lambda}^1}.
    \end{equation*}
    Here we have used the trivial inequalities $(\theta-\lambda)^4 \leq \theta^4$ and $\|f\|_{L_{\theta-\lambda- 4}^1} \leq \|f\|_{L_{\theta+\lambda}^1}$. This concludes our desired estimate.
\end{proof}

\begin{rmk}
    The constants appearing in the above inequalities are far from optimal, only manifesting some sense of comparison relationship. It is of significant importance and independent interest to determine the optimal constants in these functional inequalities in the study of diffusion processes.
\end{rmk}

Now we turn to the functional estimate for the test function, for which we recall the explicit form:
\begin{equation*}
    \begin{aligned}
        V_f(v,w)=&\, -a(v-w): \Big( \frac{\nabla^2 f}{f}(v)-\nabla \log f(v) \otimes \nabla \log f(w) \Big)\\
        &\, -b(v-w) \cdot (\nabla \log f(v)-\nabla \log f(w))\\
        &\, +a \ast f(v):\frac{\nabla^2 f}{f}(v)-c \ast f(v).
    \end{aligned}
\end{equation*}
The coefficients are given by
\begin{equation*}
    a(z)=|z|^\gamma (|z|^2 \Id-z \otimes z), \quad b(z)=-2|z|^\gamma z, \quad c(z)=
    \begin{cases}
        -2(\gamma+3)|z|^\gamma, \quad \gamma>-3;\\
        -8\pi \delta_0, \, \qquad \qquad \gamma=-3.
    \end{cases}
\end{equation*}
For the convenience of computations, we perform the estimate for the cases of very soft potentials ($\gamma \leq -2$) and other potentials ($\gamma \geq -2$) respectively.

\begin{lemma}[Very Soft Potentials]\label{verysoft}
    For $\gamma \in [-3,-2]$, the following functional inequality holds:
    \begin{align*}
        \iint_{\R^6} |V_f(v,w)|^2 f(v)f(w) \ud v \ud w 
        \leq  C \Big( 1+\int_{\R^3} (|\nabla^2 \log f|^2+|\nabla \log f|^2+|f|^2) f \ud v \Big).
    \end{align*}
    Here $C=C(\gamma,n,p,\|f\|_{L^\infty},\|f\|_{L_n^p})$ for some $p>3$ and $n>3(p-1)>6$. The dependence can be made explicit from the proof.
\end{lemma}
\begin{proof}
    It suffices to control the square integral of each line in the expression of $V_f$. 

    For the first line, we use the elementary identity
    \begin{equation*}
        \frac{\nabla^2 f}{f}(v)=\nabla^2 \log f(v)+\nabla \log f(v) \otimes \nabla \log f(v)
    \end{equation*}
    and the coefficient upper bound $|a(z)| \leq 2|z|^{\gamma+2}$ to control by
    \begin{equation*}
        K_1=4\iint |v-w|^{2\gamma+4} \Big(|\nabla^2 \log f(v)|^2+|\nabla \log f(v)|^2 |\nabla \log f(v)-\nabla \log f(w)|^2 \Big) f(v) f(w).
    \end{equation*}
    By the Cauchy-Schwarz inequality, we expand the square of the difference
    \begin{equation*}
        |\nabla \log f(v)|^2 |\nabla \log f(v)-\nabla \log f(w)|^2 \leq 3|\nabla \log f(v)|^4+|\nabla \log f(w)|^4.
    \end{equation*}
    By symmetry, inserting either $|\nabla \log f(w)|^4$ or $|\nabla \log f(v)|^4$ into the integral yields the same result, which allows us to bound $K_1$ by
    \begin{equation*}
        4\iint |v-w|^{2\gamma+4} \Big( |\nabla^2 \log f(v)|^2+4|\nabla \log f(v)|^4 \Big) f(v) f(w).
    \end{equation*}
    We integrate with respect to $w$ first by using the point-wise estimate from Lemma \ref{coefficient2}
    \begin{equation*}
        \int |v-w|^{2\gamma+4} f(w) \lesssim \langle v \rangle^{2\gamma+4} \|f\|_{L_n^p},
    \end{equation*}
    for some $p>3$ and $n>3(p-1)>6$, since $2\gamma+4 \in [-2,0]$. Therefore, we have
    \begin{equation*}
        K_1 \lesssim \|f\|_{L_n^p} \int \langle v \rangle^{2\gamma+4} \Big(|\nabla^2 \log f(v)|^2+|\nabla \log f(v)|^4 \Big) f(v).
    \end{equation*}
    By Lemma \ref{fokkerplanck} with $\theta=2\gamma+4 \in [-2,0]$, the first line is eventually bounded by
    \begin{equation*}
        C\|f\|_{L_n^p}\Big( 1+\int \langle v \rangle^{2\gamma+4} |\nabla^2 \log f(v)|^2 f(v) \Big)
    \end{equation*}
    with $C=C(\gamma,n,p)$. Here we have used $\|f\|_{L_{\theta-4}^1} \leq 1$.

    For the second line, we use the coefficient upper bound $|b(z)| \leq 2|z|^{\gamma+1}$ to control by
    \begin{equation*}
        K_2=4\iint |v-w|^{2\gamma+2}|\nabla \log f(v)-\nabla \log f(w)|^2 f(v) f(w).
    \end{equation*}
    The control of $K_2$ provides the essential difficulty in the functional estimates when treating the regime of very soft potentials, especially the Coulomb interactions. Notice that $2\gamma+2 \in [-4,-2]$, which implies that $|v-w|^{2\gamma+2}$ is not locally integrable when the singularity is sufficiently large. Fortunately, the difference structure of $|\nabla \log f(v)-\nabla \log f(w)|^2$ leaves a possibility to control $K_2$, otherwise we just integrate with respect to $w$ first and by the Fubini-Tonelli theorem we must have $K_2=+\infty$.

    However, it is not immediately clear whether this observation resolves the problem. If we admit that $\nabla \log f \in W^{1,\infty}$, as in the final remark of \cite[Theorem 3]{bresch2024duality}, then it is straightforward to use the mean-value inequality to bound $K_2$ by
    \begin{equation*}
        K_2 \leq 4 \|\nabla \log f\|_{W^{1,\infty}} \iint |v-w|^{2\gamma+4} f(v) f(w),
    \end{equation*}
    where the singularity has been lowered to be locally integrable. This basic idea, however, does not align with our objectives, since we need to apply the estimates to the solutions $f$ of the Landau equation with the velocity variable $v \in \mathbb{R}^3$. To date, no $L^\infty$ bound on $\nabla^2 \log f$ has been established; the only related result is the recent preprint \cite[Proposition 3.14]{tabary2025weak}, which provides polynomial bounds under the same conditions on the initial data. In our work, we neither impose such assumptions on the initial data nor rely on this estimate.

    Instead, we establish an extended version of the second-order commutator estimate for instance in Nguyen-Rosenzweig-Serfaty \cite{nguyen2022mean}, where they studied the mean-field limit of Riesz-type flows and derived the estimate of the following functional:
    \begin{equation*}
        \iint (v(x)-v(y))^{\otimes 2}: \nabla^2 g(x-y) f_1(x)f_2(y),
    \end{equation*}
    where $v$ is some vector field, $f_1,f_2$ are Schwarz functions and $g$ typically corresponds to the Riesz potential. Here the singularity of $\nabla^2 g(x-y)$ is also too large to be locally integrable. In view of this, the authors applied integration by parts and related the quantity to the fractional Sobolev norm. Both steps essentially require that $v$ is Lipschitz continuous. Notice that $K_2$ shares a similar structure with the functional above, under the explicit choice of $v(x)=\nabla \log f(x)$ and $f_1=f_2=f$ and $g(x)=|x|^{-2}$ with $x \in \R^3$. However, a direct application of the estimates in \cite{nguyen2022mean} is not available, since again we encounter the non-Lipschitz vector field $\nabla \log f$. Our result can also be viewed as a special case of an extension of the commutator estimate to non-Lipschitz vector fields.

    Our method avoids the $W^{1,\infty}$ bound of the vector field $\nabla \log f$, but uses its special structure. Inspired by the standard well-known identity for the Fisher information:
    \begin{equation*}
        |\nabla \log f|^2 f=\frac{|\nabla f|^2}{f}=4|\nabla \sqrt{f}|^2, \quad I(f)=\int |\nabla \log f|^2 f=4\int |\nabla \sqrt{f}|^2=4\|\sqrt{f}\|_{\Dot{H}^1}^2,
    \end{equation*}
    we rewrite the integrand in $K_2$ into
    \begin{equation*}
        |\nabla \log f(v)-\nabla \log f(w)|^2 f(v) f(w)=\frac{|\nabla f(v)|^2}{f(v)}f(w)-2\nabla f(v) \cdot \nabla f(w)+\frac{|\nabla f(w)|^2}{f(w)}f(v),
    \end{equation*}
    in which we add some square terms to the middle part and extract from the sided parts:
    \begin{align*}
        =&\, \Big( \frac{|\nabla f(v)|^2}{f(v)}-\frac{|\nabla f(w)|^2}{f(w)} \Big) (f(w)-f(v))+|\nabla f(v)-\nabla f(w)|^2\\
        =&\, 4\Big( |\nabla \sqrt{f}(v)|^2-|\nabla \sqrt{f}(w)|^2 \Big)(f(w)-f(v))+|\nabla f(v)-\nabla f(w)|^2\\
        \leq &\, 2\Big( |\nabla \sqrt{f}(v)|^2-|\nabla \sqrt{f}(w)|^2\Big)^2+2|f(v)-f(w)|^2+|\nabla f(v)-\nabla f(w)|^2.
    \end{align*}
    Let us now insert this estimate into the integral in $K_2$. Recall that $\gamma \in [-3,-2]$ and $2\gamma+2 \in [-4,-2]$. Therefore, it suffices to estimate the cases with critical values $\gamma=-3$ and $\gamma=-2$, and the intermediate cases follow from the Cauchy-Schwarz inequality. For $\gamma=-3$, we apply the integral expression for the fractional Sobolev norm:
    \begin{equation*}
        \|g\|_{\Dot{H}^{1/2}}^2 \sim \iint \frac{|g(v)-g(w)|^2}{|v-w|^4}.
    \end{equation*}
    Take $g=f$, $g=\nabla f$ and $g=|\nabla \sqrt{f}|^2$ respectively, and we control $K_2$ by
    \begin{equation*}
        K_2 \lesssim \|\nabla f\|_{\Dot{H}^{1/2}}^2 +\|f\|_{\Dot{H}^{1/2}}^2+ \| |\nabla \sqrt{f}|^2 \|_{\Dot{H}^{1/2}}^2.
    \end{equation*}
    Combining with the standard interpolation inequality between the Sobolev spaces $L^2$, $H^{1/2}$ and $H^1$, and using the elementary identity $\nabla |\nabla \sqrt{f}|^2=2\nabla^2 \sqrt{f} \cdot \nabla \sqrt{f}$ and the Cauchy-Schwarz inequality, we have
    \begin{equation*}
        K_2 \lesssim \|f\|_{H^2}^2+\|\sqrt{f}\|_{H^2}^2.
    \end{equation*}
    We return to the second-order Fisher information using \eqref{gamma2} in Lemma \ref{fokkerplanck}:
    \begin{equation*}
        \|\nabla f\|_{L^2}^2 \leq \|f\|_{L^\infty} \cdot I(f), \quad \|\nabla \sqrt{f}\|_{L^2}^2=\frac{1}{4}I(f);
    \end{equation*}
    \begin{equation*}
        \|\nabla^2 f\|_{L^2}^2 \leq 50\|f\|_{L^\infty} \cdot \int |\nabla^2 \log f|^2 f, \quad \|\nabla^2 \sqrt{f}\|_{L^2}^2 \leq 4\int |\nabla^2 \log f|^2 f.
    \end{equation*}
    Therefore, when $\gamma=-3$ we bound the second line by
    \begin{equation*}
        C(1+\|f\|_{L^\infty}) \int (|\nabla^2 \log f|^2+|\nabla \log f|^2+|f|^2)f.
    \end{equation*}
    For $\gamma=-2$, since the singularity $|v-w|^{-2}$ is locally integrable now, similar to the argument in the estimate of $K_1$, we bound $K_2$ by integrating with respect to $w$ first and arrive at
    \begin{equation*}
        K_2 \lesssim \|f\|_{L_n^p} \int \langle v \rangle^{-2} |\nabla \log f(v)|^2 f(v).
    \end{equation*}
    Therefore, for any $\gamma \in [-3,-2]$, we interpolate between the two critical cases by the Cauchy-Schwarz inequality, and the second line is eventually bounded by
    \begin{equation*}
        C(1+\|f\|_{L^\infty}+\|f\|_{L_n^p}) \int (|\nabla^2 \log f|^2+|\nabla \log f|^2+|f|^2) f
    \end{equation*}
    with $C=C(\gamma,n,p)$.
    
    For the third line, we combine the point-wise estimate from Lemma \ref{coefficient2}
    \begin{equation*}
        |a \ast f(v)| \lesssim \int |v-w|^{\gamma+2} f(w) \lesssim \langle v \rangle^{\gamma+2}\|f\|_{L_n^p}, 
    \end{equation*}
    which holds thanks to $\gamma+2 \in [-1,0]$, the $L^\infty$ estimate from Lemma \ref{coefficient1}
    \begin{equation*}
        |c \ast f(v)| \lesssim \|f\|_{L^\infty}^{-\frac{\gamma}{3}},
    \end{equation*}
    with the elementary identity
    \begin{equation*}
        \frac{\nabla^2 f}{f}(v)=\nabla^2 \log f(v)+\nabla \log f(v) \otimes \nabla \log f(v)
    \end{equation*}
    to control by
    \begin{equation*}
        K_3=C\|f\|_{L^\infty}^{-\frac{2}{3}\gamma}+C\|f\|_{L_n^p}^2 \int \langle v \rangle^{2\gamma+4} \Big(|\nabla^2 \log f(v)|^2+|\nabla \log f(v)|^4 \Big) f(v).
    \end{equation*}
    By Lemma \ref{fokkerplanck} with $\theta=2\gamma+4 \in [-2,0]$, the third line is eventually bounded by
    \begin{equation*}
        C(1+\|f\|_{L_n^p}^2+\|f\|_{L^\infty}^{-\frac{2}{3}\gamma}) \Big( 1+\int \langle v \rangle^{2\gamma+4} |\nabla^2 \log f(v)|^2 f(v) \Big)
    \end{equation*}
    with $C=C(\gamma,n,p)$. Here we have used $\|f\|_{L_{\theta-4}^1} \leq 1$.

    A simple bookkeeping gives the desired estimate.
\end{proof}

\begin{rmk}
    It is straightforward to check that the second functional inequality in Proposition \ref{squareintegrable}
    \begin{equation*}
        \int_0^T \int_{\R^3} |V_f(v,v)|^2 f(v) \ud v \ud t <\infty
    \end{equation*}
    coincides exactly with the estimate of the third line in the square integral above, and thus is under control.
\end{rmk}

\begin{lemma}[Other Potentials]\label{other}
    For $\gamma \in [-2,1]$, the following functional inequality holds:
    \begin{align*}
        \iint_{\R^6} |V_f(v,w)|^2 f(v)f(w) \ud v \ud w 
        \leq  C \Big(1+\int_{\R^3} \langle v \rangle^{2\gamma+4}|\nabla^2 \log f(v)|^2 f(v) \ud v \Big).
    \end{align*}
    Here $C=C(\gamma,n,p,\|f\|_{L_6^1},\|f\|_{L_n^p})$ for some $p>3$ and $n>3(p-1)>6$. The dependence can be made explicit from the proof.
\end{lemma}
\begin{proof}
     The estimate follows from similar steps as in the case of very soft potentials, but the computations are much simpler since the singularity is small enough.

     For the first line, we control again by
     \begin{equation*}
         K_1=4\iint |v-w|^{2\gamma+4} \Big( |\nabla^2 \log f(v)|^2+4|\nabla \log f(v)|^4 \Big) f(v)f(w).
     \end{equation*}
     We integrate with respect to $w$ using the point-wise estimate from Lemma \ref{coefficient2}
     \begin{equation*}
         \int |v-w|^{2\gamma+4} f(w) \lesssim \langle v \rangle^{2\gamma+4}+\|f\|_{L_6^1},
     \end{equation*}
     since $2\gamma+4 \in [0,6]$. Therefore, we have
     \begin{equation*}
         K_1 \lesssim (1+\|f\|_{L_6^1}) \int \langle v \rangle^{2\gamma+4} \Big(|\nabla^2 \log f(v)|^2+|\nabla \log f(v)|^4 \Big) f(v).
     \end{equation*}
     By Lemma \ref{fokkerplanck} with $\theta=2\gamma+4 \in [0,6]$, the first line is eventually bounded by
     \begin{equation*}
         C(1+\|f\|_{L_6^1}) \Big(1+\int \langle v \rangle^{2\gamma+4} |\nabla^2 \log f(v)|^2f(v) \Big)
     \end{equation*}
     with $C=C(\gamma)$. Here we have used $\|f\|_{L_{\theta-4}^1} \leq \|f\|_{L_2^1}=4$.

     For the second line, we control again by
     \begin{equation*}
         K_2=4\iint |v-w|^{2\gamma+2} |\nabla \log f(v)-\nabla \log f(w)|^2 f(v)f(w).
     \end{equation*}
     Similar to the argument in the estimate of $K_1$, we bound $K_2$ by integrating with respect to $w$ first and arrive at
     \begin{equation*}
         K_2 \lesssim (1+\|f\|_{L_6^1}+\|f\|_{L_n^p}) \Big( 1+\int \langle v \rangle^{2\gamma+2} |\nabla \log f(v)|^2 f(v) \Big),
     \end{equation*}
     since $2\gamma+2 \in [-2,4]$. By Lemma \ref{fokkerplanck2} with $\theta=2\gamma+2 \in [-2,4], \lambda=0$, the second line is eventually bounded by
     \begin{equation*}
         C(1+\|f\|_{L_6^1}+\|f\|_{L_n^p}) \Big( 1+\|f\|_{L_4^1}+\int \langle v \rangle^{2\gamma+2} |\nabla^2 \log f(v)|^2 f(v) \Big)
     \end{equation*}
     with $C=C(\gamma,n,p)$. Here we have used $\|f\|_{L_{\theta+\lambda}^1} \leq \|f\|_{L_4^1}$.

     For the third line, we combine again the point-wise estimate from Lemma \ref{coefficient2}
    \begin{equation*}
        |a \ast f(v)| \lesssim \int |v-w|^{\gamma+2} f(w) \lesssim \langle v \rangle^{\gamma+2}+\|f\|_{L_6^1}, 
    \end{equation*}
    \begin{equation*}
        |c \ast f(v)| \lesssim 
        \begin{cases}
            \langle v \rangle^\gamma \|f\|_{L_n^p} \quad \, \gamma \in [-2,0);\\
            1+\langle v \rangle^\gamma \qquad  \gamma \in [0,1],
        \end{cases}
    \end{equation*}
    which hold thanks to $\gamma+2 \in [0,3]$, with the elementary identity
    \begin{equation*}
        \frac{\nabla^2 f}{f}(v)=\nabla^2 \log f(v)+\nabla \log f(v) \otimes \nabla \log f(v)
    \end{equation*}
    to control by
    \begin{equation*}
        K_3=C(1+\|f\|_{L_n^p}^2)+C(1+\|f\|_{L_6^1}^2) \int \langle v \rangle^{2\gamma+4} \Big( |\nabla^2 \log f(v)|^2+|\nabla \log f(v)|^4 \Big) f(v).
    \end{equation*}
    Here we have used $\|f\|_{L_{2\gamma}^1} \leq \|f\|_{L_2^1}=4$. By Lemma \ref{fokkerplanck} with $\theta=2\gamma+4 \in [0,6]$, the third line is eventually bounded by
    \begin{equation*}
        C(1+\|f\|_{L_n^p}^2+\|f\|_{L_6^1}^2) \Big( 1+\int \langle v \rangle^{2\gamma+4} |\nabla^2 \log f(v)|^2 f(v) \Big)
    \end{equation*}
    with $C=C(\gamma,n,p)$. Here we have used $\|f\|_{L_{\theta-4}^1} \leq \|f\|_{L_2^1}=4$.

    A simple bookkeeping gives the desired estimate.
\end{proof}

\begin{rmk}
    It is straightforward to check that the second functional inequality in Proposition \ref{squareintegrable}
    \begin{equation*}
        \int_0^T \int_{\R^3} |V_f(v,v)|^2 f(v) \ud v \ud t <\infty
    \end{equation*}
    coincides exactly with the estimate of the third line in the square integral above, and thus is under control.
\end{rmk}

\noindent \textit{Application to First-order Systems with 3D Coulomb/Riesz Interaction.}  Our extended version of the second-order commutator estimate in the proof of Lemma \ref{verysoft} has other applications. Apart from Kac's program for the Landau equation, we can also combine this estimate with the arguments in \cite{bresch2024duality} to derive propagation of chaos for first-order systems with 3D Coulomb/Riesz interaction and additive noise in the whole space.

    Recall the first-order particle system given by the canonical form \eqref{cannSDE}:
    \begin{equation*}
        \ud X_t^i = \frac{1}{N} \sum_{j \ne i}^N K(X_t^i - X_t^j) \ud t + \sqrt{2 \sigma} \ud B_t^i, \quad i = 1, 2, \cdots, N,
    \end{equation*}
    which is expected to converge as $N \rightarrow \infty$ to its mean-field limit, given by the McKean-Vlasov equation:
    \begin{equation*}
        \p_tf+\nabla \cdot(f(K \ast f))=\sigma \Delta f.
    \end{equation*}
    We require $X_t^i \in \R^3$ to be positioned in the whole space and consider the interacting kernel $K$ given by the opposite of the gradient of the 3D Riesz potential:
    \begin{equation*}
        K=-\nabla g, \quad g(x)=\begin{cases}
            |x|^{-s}, \qquad \, \, 0< s<d=3;\\
            -\log|x|, \quad s=0.
        \end{cases}
    \end{equation*}
    The critical case $s=d-2=1$ corresponds to the Coulomb interaction. For $0 \leq s <1/2$, we have $K \in L_{\text{loc}}^2(\R^3)$, which meets the regularity assumption in \cite[Theorem 3]{bresch2024duality}; while for larger $s$, the regularity assumption shall be relaxed to
    \begin{equation}\label{functional}
        \int_0^T \iint_{\R^6} |K(x-y) \cdot (\nabla \log f(x)-\nabla \log f(y))|^2 f(x) f(y) \ud x \ud y \ud t<\infty.
    \end{equation}
    In \cite[Theorem 3]{bresch2024duality} the authors established this functional inequality provided that $\nabla \log f \in W^{1,\infty}(\R^3)$, which, however, is still not known to hold true—even if it is satisfied initially, which itself represents a substantial restriction.

    We apply our commutator estimate in the proof of Lemma \ref{verysoft} to establish this functional inequality for a large class of potentials and initial data. Indeed, as long as $2s+2 \in (d,d+2)=(3,5)$, or equivalently $s \in (1/2,3/2)$, we control the spatial integral on the left-hand side of \eqref{functional} by
    \begin{equation*}
        \iint_{\R^6} |x-y|^{-(2s+2)} |\nabla \log f(x)-\nabla \log f(y)|^2 f(x)f(y) \ud x \ud y.
    \end{equation*}
    Using the commutator estimate in the proof of Lemma \ref{verysoft}, we can bound it by
    \begin{equation*}
        C\int (|\nabla^2 \log f|^2 +|\nabla \log f|^2 +|f|^2)f
    \end{equation*}
    with $C=C(s,\|f\|_{L^\infty})$. Then it suffices to prove
    \begin{equation*}
        \int_0^T \int |\nabla^2 \log f|^2 f \ud t<\infty, \quad \int_0^T \int |\nabla \log f|^2 f \ud t<\infty.
    \end{equation*}
    Thanks to the dissipative estimates
    \begin{equation*}
        \frac{\ud}{\ud t} \int f \log f=-\sigma \int |\nabla \log f|^2 f-\int (\nabla g \ast f) \cdot \nabla f,
    \end{equation*}
    \begin{equation*}
        \frac{\ud}{\ud t} \int |\nabla \log f|^2 f=-2\sigma \int |\nabla^2 \log f|^2 f-2\int (\nabla^2 \log f:\nabla^2 g \ast f)f,
    \end{equation*}
    we use the Cauchy-Schwarz inequality and the Newton-Leibniz formula to conclude that, under the standard regularity assumption $f \in L^\infty$ and $\nabla^2 g \ast f \in L^\infty$ and the initial assumptions 
    $H(f_0)<\infty$ and $I(f_0)<\infty$, the functional inequality \eqref{functional} holds. This yields  propagation of chaos for all $s \in (1/2,3/2)$, especially covering the 3D Coulomb case.

\subsection{Estimate on Second-order Fisher Information}

From \cite{guillen2023landau}, we know that the $L^3$ norm and the Fisher information of the solution $f$ of the Landau equation have uniform bounds. Therefore, to prove the square-integrability of $V_f$ in Proposition \ref{squareintegrable}, it remains to control the time integral of the (weighted) second-order Fisher information.

Quantities like entropy, Fisher information and second-order Fisher information have been well studied in the classical theory of diffusion equations, among which the most famous one must be the heat equation. It is well known that along the heat flow, the Fisher information equals the opposite of the time derivative of the entropy (classical de Bruijn identity), and the second-order Fisher information equals the opposite of the time derivative of the Fisher information. As a direct consequence, the entropy and Fisher information are monotonically decreasing along the heat flow. Higher-order derivatives, however, are not obviously non-positive.

Estimates of such functionals have been widely studied in the context of Fokker-Planck equations, as they share some similarities with the heat equation.  Concerning the Boltzmann equation and the Landau equation, in view of the Boltzmann $H$-Theorem, the estimates of entropy and entropy dissipation have been deeply investigated. We mention typical results that aim to control the entropy dissipation functional from below by the relative entropy or the weighted Fisher information between the solution and the equilibrium, including for instance \cite{cercignani1982h,desvillettes2011celebrating,desvillettes2000spatiallyI,desvillettes2000spatiallyII,toscani1999sharp,toscani2000on,villani2003cercignani}. Such estimates are essentially useful for deriving quantitative convergence to equilibrium and spectral gap estimates. 

Estimates of Fisher information, however, have attracted relatively less interest before the recent breakthrough \cite{guillen2023landau}. It was first conjectured by McKean \cite{mckean1966speed} that the Fisher information of the Boltzmann equation is monotonically decreasing. McKean proved this result for the $1$-dimensional Kac's caricature of the Boltzmann equation. Subsequent results include for instance \cite{toscani1992new,villani1998fisher,villani2000decrease}. The full power of Fisher information in the relevant study has been revealed by \cite{guillen2023landau,imbert2024monotonicity} for the Landau equation and the Boltzmann equation separately. We refer to the recent review \cite{villani2025fisher} for comprehensive discussions.

Indeed, McKean has conjectured that all consecutive higher-order derivatives of entropy shall be monotone along the Boltzmann equation, and thus the Landau equation by passing to the grazing limit. It is likely that we can try the same lifting technique as in \cite{guillen2023landau,imbert2024monotonicity} with the second-order Fisher information, since it also satisfies the special property ($\langle \cdot, \cdot \rangle$ stands for the Gateaux derivative)
\begin{equation*}
    J(f)=\frac{1}{2}J(F), \quad \langle J^\prime (f), q(f) \rangle=\frac{1}{2} \langle J^\prime (F), Q(F) \rangle
\end{equation*}
by direct computations, where we denote by $J(f)$ the second-order Fisher information
\begin{equation*}
    J(f)=\int |\nabla^2 \log f(v)|^2 f(v).
\end{equation*}
However, the layer decomposition and the derivative computation in \cite{guillen2023landau} are much more complicated for $J(f)$. We have tried this idea but cannot solve it. Fortunately, for the purpose of running Kac's program in this article, we only need the time integrability.

\begin{prop}\label{propfunctional}
    Under the same assumptions as in Theorem \ref{thm:poc}, for any $T>0$, we have the following functional estimate: if $\gamma \in [-3,-2]$, then
    \begin{equation*}
        \int_0^T \int_{\R^3} |\nabla^2 \log f(v)|^2 f(v) \ud v \ud t \leq C(1+T)^4 <\infty;
    \end{equation*}
    if $\gamma \in [-2,1]$, then
    \begin{equation*}
        \int_0^T \int_{\R^3} \langle v \rangle^{2\gamma+4} |\nabla^2 \log f(v)|^2 f(v) \ud v \ud t \leq C(1+T)^4 <\infty.
    \end{equation*}
\end{prop}

The proof of Proposition \ref{propfunctional} will be divided into several lemmas. The main idea is to compute the time derivative of the weighted Fisher information, defined by
\begin{equation*}
    I_\varphi(f)=\int |\nabla \log f|^2 f \varphi,
\end{equation*}
along the Landau equation. The lower bound of the non-negative definite coefficient matrix from Lemma \ref{ellipticity} and the diffusive structure of the Landau equation will produce the weighted second-order Fisher information. We summarize it in the following lemma.

\begin{lemma}\label{derivative}
    The time derivative of the weighted Fisher information along the Landau equation is given by
    \begin{align*}
        &\, \frac{\ud}{\ud t} I_\varphi(f)\\
        =&\, -2\int \nabla^2 \log f:(\bar a \cdot \nabla^2 \log f) f\varphi\\
         &\, -2\int \nabla \log f \cdot \nabla \bar a:(\nabla \log f \otimes \nabla \varphi)f+\int |\nabla \log f|^2 (\bar a:\nabla^2 \varphi)f+2\int |\nabla \log f|^2 (\bar b \cdot \nabla \varphi)f\\
         &\, -2\int \nabla^2 \log f:(\nabla \bar a \cdot \nabla \log f)f \varphi-2\int (\nabla \log f \cdot \nabla \bar b \cdot \nabla \log f) f \varphi-2\int (\nabla \log f \cdot \nabla \bar c) f \varphi.
     \end{align*}
     We remark that this formula holds for general weight function $\varphi$ and all $\gamma \in [-3,1]$.
\end{lemma}

\begin{proof}
    We recall the divergence form of the Landau equation
    \begin{equation*}
        \p_t f=\nabla \cdot (\bar a \cdot \nabla f-\bar bf).
    \end{equation*}
    We compute directly that
    \begin{align*}
        \frac{\ud}{\ud t} \int |\nabla \log f|^2 f \varphi &\, =\int |\nabla \log f|^2 \nabla \cdot (\bar a \cdot \nabla f-\bar bf) \varphi \\
        &\, \qquad \quad +2\int \nabla \log f \cdot \nabla \frac{\nabla \cdot (\bar a \cdot \nabla f-\bar bf)}{f} f\varphi \\
        &\, =\text{Term}_1+ \text{Term}_2.
    \end{align*}
    For the first term, we remove the divergence operator in front of $\bar a \cdot \nabla f$ using integration by parts, but we keep the divergence operator in front of $\bar b f$ unchanged.
    \begin{align*}
        \text{Term}_1=&\, -2\int (\nabla \log f \cdot \nabla^2 \log f \cdot \bar a  \cdot \nabla f) \varphi-\int |\nabla \log f|^2 (\nabla \varphi \cdot \bar a \cdot \nabla f)\\
        &\, -\int |\nabla \log f|^2 (\bar c f)\varphi -\int |\nabla \log f|^2 (\bar b \cdot \nabla f) \varphi.
    \end{align*}
    For the second term, we simplify one step before integrating by parts.
    \begin{align*}
        \text{Term}_2=&\, 2\int \nabla \log f \cdot \nabla \Big(\nabla \cdot (\bar a \cdot \nabla \log f) \Big) f \varphi+2\int \nabla \log f \cdot \nabla (\nabla \log f \cdot \bar a \cdot \nabla \log f )f \varphi\\
        &\, -2\int (\nabla \log f \cdot \nabla \bar c) f \varphi-2\int \nabla \log f \cdot \nabla (\bar b \cdot \nabla \log f) f \varphi\\
        =&\, \text{Term}_{2-1}+ \text{Term}_{2-2}- \text{Term}_{2-3}.
    \end{align*}
    We integrate by parts in $\text{Term}_{2-1}$ to remove the divergence operator.
    \begin{align*}
        \text{Term}_{2-1}=&\, -2\int \nabla^2 \log f: \nabla(\bar a \cdot \nabla \log f) f \varphi-2\int \Big(\nabla \log f \cdot \nabla(\bar a \cdot \nabla \log f) \cdot \nabla f \Big) \varphi\\
        &\, -2\int \Big( \nabla \log f \cdot \nabla( \bar a \cdot \nabla \log f) \cdot \nabla \varphi \Big) f.
    \end{align*}
    We expand the brackets in $\nabla (\bar a \cdot \nabla \log f)$.
    \begin{align*}
      \text{Term}_{2-1}=&\, -2\int \nabla^2 \log f: (\bar a \cdot \nabla^2 \log f) f\varphi-2\int \nabla^2 \log f:(\nabla \bar a \cdot \nabla \log f) f \varphi\\
        &\, -2\int (\nabla \log f \cdot \nabla^2 \log f \cdot \bar a \cdot \nabla f) \varphi-2\int \Big( \nabla \bar a:(\nabla \log f)^{\otimes 3} \Big) f \varphi\\
        &\, -\int (\nabla |\nabla \log f|^2 \cdot \bar a \cdot \nabla \varphi)f-2\int (\nabla \log f \cdot \nabla \bar a :(\nabla \log f \otimes \nabla \varphi))f.
    \end{align*}
    We expand $\text{Term}_{2-2}$ directly.
    \begin{equation*}
        \text{Term}_{2-2}=4\int (\nabla \log f \cdot \nabla^2 \log f \cdot \bar a \cdot \nabla f) \varphi+2\int \Big( \nabla \bar a:(\nabla \log f)^{\otimes 3} \Big) f\varphi.
    \end{equation*}
    We also expand the bracket in $\text{Term}_{2-3}$ and integrate by parts.
     \begin{align*}
         \text{Term}_{2-3}=&\, 2\int(\nabla \log f \cdot \nabla \bar c) f \varphi+2\int (\nabla \log f \cdot \nabla \bar b \cdot \nabla \log f) f \varphi+\int (\bar b \cdot \nabla |\nabla \log f|^2) f \varphi\\
         =&\, 2\int(\nabla \log f \cdot \nabla \bar c) f \varphi+2\int (\nabla \log f \cdot \nabla \bar b \cdot \nabla \log f) f \varphi-\int |\nabla \log f|^2 (\bar c f) \varphi\\
         &\, -\int |\nabla \log f|^2 (\bar b \cdot \nabla f) \varphi-\int |\nabla \log f|^2 (\bar b \cdot \nabla \varphi) f.
     \end{align*}

     After a little bookkeeping, we notice some cancellations happen in $\text{Term}_1$, $\text{Term}_{2-1}$, $\text{Term}_{2-2}$, $\text{Term}_{2-3}$ and rewrite the formula into
     \begin{align*}
        &\, \frac{\ud}{\ud t} \int |\nabla \log f|^2 f \varphi=\text{Term}_1+\text{Term}_{2-1}+\text{Term}_{2-2}-\text{Term}_{2-3}\\
         =&\, -2\int \nabla^2 \log f:(\bar a \cdot \nabla^2 \log f) f\varphi\\
         &\, -\int (\nabla |\nabla \log f|^2 \cdot \bar a \cdot \nabla \varphi)f-\int |\nabla \log f|^2 (\nabla \varphi \cdot \bar a\cdot \nabla f)+\int |\nabla \log f|^2 (\bar b \cdot \nabla \varphi) f\\
         &\, -2\int \nabla^2 \log f:(\nabla \bar a \cdot \nabla \log f)f \varphi-2\int \nabla \log f \cdot \nabla \bar a:(\nabla \log f \otimes \nabla \varphi)f\\
         &\, -2\int (\nabla \log f \cdot \nabla \bar b \cdot \nabla \log f) f \varphi-2\int (\nabla \log f \cdot \nabla \bar c) f \varphi.
     \end{align*}
     Finally, we integrate by parts on the first term in the second line on the right-hand side, which simplifies the expression in the second line. We also regroup the terms appropriately to obtain the following formula.
     \begin{align*}
        &\, -2\int \nabla^2 \log f:(\bar a \cdot \nabla^2 \log f) f\varphi\\
         &\, -2\int \nabla \log f \cdot \nabla \bar a:(\nabla \log f \otimes \nabla \varphi)f+\int |\nabla \log f|^2 (\bar a:\nabla^2 \varphi)f+2\int |\nabla \log f|^2 (\bar b \cdot \nabla \varphi)f\\
         &\, -2\int \nabla^2 \log f:(\nabla \bar a \cdot \nabla \log f)f \varphi-2\int (\nabla \log f \cdot \nabla \bar b \cdot \nabla \log f) f \varphi-2\int (\nabla \log f \cdot \nabla \bar c) f \varphi.
     \end{align*}
\end{proof}

We comment on the formula above before making further estimates. The first line is the non-positive term thanks to the positive definiteness of the diffusion matrix $\bar a$. The second line contains terms involving derivatives of $\varphi$, which automatically vanish if $\varphi=1$ when computing the time derivative of the pure Fisher information. The last line contains quantities involving derivatives of $\log f$ and gradients of coefficients $\bar a, \bar b, \bar c$, which is the main error term that shall be controlled.

Now we estimate the above integrals under our specific choice that $\varphi(v)=\langle v \rangle^\theta$ with some $\theta \geq 0$ to be determined.

\begin{lemma}\label{lemma:finalderivative}
    Let $\varphi(v)=\langle v \rangle^\theta$. There exists some $p>3$ and $3(p-1)<n<m$ such that
    \begin{equation}\label{finalderivative}
        \frac{\ud}{\ud t} I_\varphi(f) \leq -c_0 \int \langle v \rangle^{\gamma+\theta} |\nabla^2 \log f(v)|^2 f(v)+C(1+t) I_{\langle v \rangle^\theta+\langle v \rangle^{\theta+\gamma+2}}(f)+C(1+t).
    \end{equation}
\end{lemma}
\begin{proof}
    The standard ellipticity estimate from Lemma \ref{ellipticity} gives
     \begin{equation*}
         \bar a \geq c_0 \langle v \rangle^{\gamma} \text{ Id}
     \end{equation*}
     for some positive constant $c_0$. Therefore, the non-positive term in the first line of the formula in Lemma \ref{derivative} is bounded by
     \begin{equation*}
         -2c_0\int \langle v \rangle^{\gamma+\theta} |\nabla^2 \log f(v)|^2 f(v).
     \end{equation*}
     
     For the second line of the formula in Lemma \ref{derivative}, we use the point-wise estimate from Lemma \ref{coefficient2}
     \begin{equation*}
         |\bar a(v)| \lesssim (\|f\|_{L_n^p}+\|f\|_{L_6^1})\langle v \rangle^{\gamma+2}, \quad |\bar b(v)|, |\nabla \bar a(v)| \lesssim (\|f\|_{L_n^p}+\|f\|_{L_6^1}) \langle v \rangle^{\gamma+1},
     \end{equation*}
     which covers the whole range of parameters $\gamma \in [-3,1]$. Therefore, the second line of the formula in Lemma \ref{derivative} is bounded by
     \begin{equation*}
         C (\|f\|_{L_n^p}+\|f\|_{L_6^1}) \int \langle v \rangle^{\gamma+\theta} |\nabla \log f(v)|^2 f(v)
     \end{equation*}
     with $C=C(\gamma,n,p,\theta)$.

     For the third line of the formula in Lemma \ref{derivative}, we shall estimate the three integrals respectively. The first integral
     \begin{equation*}
         A_1=-2\int \nabla^2 \log f:(\nabla \bar a \cdot \nabla \log f )f \varphi
     \end{equation*}
     is controlled by the Cauchy-Schwarz inequality and the point-wise estimate from Lemma \ref{coefficient2}
     \begin{align*}
         A_1 &\, \leq \frac{c_0}{2}\int \langle v \rangle^{\gamma+\theta} |\nabla^2 \log f(v)|^2f(v) +\frac{2}{c_0}\int \langle v \rangle^{-\gamma-\theta}|\nabla \log f(v)|^2 f(v) |\nabla \bar a|^2 \varphi^2\\
         &\, \leq \frac{c_0}{2}\int \langle v \rangle^{\gamma+\theta} |\nabla^2 \log f(v)|^2f(v) +C(\|f\|_{L_n^p}^2+\|f\|_{L_6^1}^2) \int \langle v \rangle^{\gamma+\theta+2} |\nabla \log f(v)|^2 f(v)
     \end{align*}
     with $C=C(\gamma,n,p,\theta)$. The second integral
     \begin{equation*}
         A_2=-2\int(\nabla \log f \cdot \nabla \bar b \cdot \nabla \log f)f\varphi
     \end{equation*}
     is controlled using integration by parts with $b=\nabla \cdot a$ and the point-wise estimate from Lemma \ref{coefficient2} and Lemma \ref{fokkerplanck}
     \begin{align*}
         A_2=&\, 4\int (\nabla^2 \log f \cdot \nabla \bar a \cdot \nabla \log f)f\varphi+2\int\Big( \nabla \bar a : (\nabla \log f)^{\otimes 3}\Big) f\varphi\\
         +&\, 2\int \nabla \log f \cdot \nabla \bar a:(\nabla \log f \otimes \nabla \varphi)f\\
         \leq &\, \frac{c_0}{4}\int \langle v \rangle^{\gamma+\theta} |\nabla^2 \log f(v)|^2f(v)\\
         +&\, C(\|f\|_{L_n^p}^2+\|f\|_{L_6^1}^2) \int \langle v \rangle^{\gamma+\theta+2} |\nabla \log f(v)|^2f(v)+C\|f\|_{L_{\gamma+\theta-4}^1}
     \end{align*}
     with $C=C(\gamma,n,p,\theta)$. Finally, for $\gamma \in [-3,0]$, the third integral
     \begin{equation*}
         A_3=-2\int (\nabla \log f \cdot \nabla \bar c) f \varphi
     \end{equation*}
     is controlled using integration by parts and the $L^\infty$ estimate from Lemma \ref{coefficient1}
     \begin{align*}
         A_3=&\, 2\int (\Delta \log f)\bar c f\varphi+2\int |\nabla \log f|^2 \bar c f\varphi+2\int (\nabla \log f \cdot \nabla \varphi) \bar c f\\
         \leq &\, \frac{c_0}{4}\int \langle v \rangle^{\gamma+\theta} |\nabla^2 \log f(v)|^2 f(v)\\
         +&\, C\|f\|_{L^\infty}^{-\frac{\gamma}{3}}\int \langle v \rangle^\theta |\nabla \log f(v)|^2 f(v)+C\|f\|_{L^\infty}^{-\frac{2}{3}\gamma} \cdot \|f\|_{L_{\theta-\gamma}^1}+C\|f\|_{L^\infty}^{-\frac{\gamma}{3}} \cdot \|f\|_{L_{\theta-2}^1}
     \end{align*}
     with $C=C(\gamma,\theta)$; while for $\gamma \in [0,1]$, it is controlled by the Cauchy-Schwarz inequality and the point-wise estimate from Lemma \ref{coefficient2}
     \begin{align*}
         A_3 \leq C\|f\|_{L_n^p}^2\int \langle v \rangle^{\gamma+\theta+2} |\nabla \log f(v)|^2 f(v)+C\|f\|_{L_{\gamma+\theta-4}^1}
     \end{align*}
     with $C=C(\gamma,n,p,\theta)$. Here we have used $|\nabla \bar c(v)| \lesssim \langle v \rangle^{\gamma-1} \|f\|_{L_n^p}$.

     A simple bookkeeping concludes that
     \begin{align*}
         \frac{\ud}{\ud t} I_\varphi(f) \leq &\, -c_0 \int \langle v \rangle^{\gamma+\theta} |\nabla^2 \log f(v)|^2 f(v)\\
         &\, +C(\|f\|_{L_n^p}+\|f\|_{L_6^1}+\|f\|_{L_n^p}^2+\|f\|_{L_6^1}^2+\|f\|_{L^\infty}^{-\frac{\gamma}{3}}) I_{\langle v \rangle^\theta+\langle v \rangle^{\gamma+\theta+2}}(f)\\
         &\, +C(\|f\|_{L^\infty}^{-\frac{2}{3}\gamma} \cdot \|f\|_{L_{\theta-\gamma}^1}+\|f\|_{L^\infty}^{-\frac{\gamma}{3}} \cdot \|f\|_{L_{\theta-2}^1}+\|f\|_{L_{\gamma+\theta-4}^1})
     \end{align*}
     with $C=C(\gamma,n,p,\theta)$. Thanks to \eqref{lp}, all the weighted norms that appear above are controlled by $C(1+t)$ with $C$ depending on the parameters and the initial norms.
\end{proof}

     Finally, we use the Fokker-Planck type inequality in Lemma \ref{fokkerplanck2} to control the weighted Fisher information on the right-hand side of \eqref{finalderivative}. We shall treat the different cases of very soft potentials $\gamma \leq -2$ and the other potentials $\gamma \geq -2$ respectively.

     For the case of very soft potentials $\gamma \in [-3,-2]$, we choose $\theta=-\gamma \in [2,3]$ in Lemma \ref{lemma:finalderivative} and notice $\gamma+\theta+2 \leq \theta$:
     \begin{equation*}
         C(1+t)I_{\langle v \rangle^{-\gamma}}(f) \leq \frac{c_0}{2}\int |\nabla^2 \log f(v)|^2 f(v)+C(1+t)^2\|f\|_{L_{-2\gamma}^1}.
     \end{equation*}
     Combining with \eqref{finalderivative} and using $\|f\|_{L_{-2\gamma}^1} \leq \|f\|_{L_6^1} \leq C(1+t)$, we conclude by
     \begin{equation*}
         \frac{\ud}{\ud t} I_{\langle v \rangle^{-\gamma}}(f) \leq -\frac{c_0}{2}\int |\nabla^2 \log f(v)|^2 f(v)+C(1+t)^3.
     \end{equation*}
     Integrating from $0$ to $T$ gives that
     \begin{equation*}
         \int_0^T \int_{\R^3} |\nabla^2 \log f(v)|^2 f(v) \ud v \ud t \leq C(1+T)^4+CI_{\langle v \rangle^{-\gamma}}(f_0) \leq C(1+T)^4
     \end{equation*}
     by the initial assumption. Here $C$ depends on the parameters $\gamma,m$ and the initial quantities $\|f_0\|_{L_m^1},\|f_0\|_{L^\infty},I_{\langle v \rangle^{-\gamma}}(f_0)$.

     For the case of other potentials $\gamma \in [-2,1]$, we choose $\theta=\gamma+4 \in [2,5]$ in Lemma \ref{lemma:finalderivative} and notice $\gamma+\theta+2 \geq \theta$:
     \begin{equation*}
         C(1+t)I_{\langle v \rangle^{2\gamma+6}}(f) \leq \frac{c_0}{2}\int \langle v \rangle^{2\gamma+4} |\nabla^2 \log f(v)|^2 f(v)+C(1+t)^2\|f\|_{L_{2\gamma+8}^1}.
     \end{equation*}
     Combining with \eqref{finalderivative} and using $\|f\|_{L_{2\gamma+8}^1} \leq C(1+t)$, we conclude by
     \begin{equation*}
         \frac{\ud}{\ud t} I_{\langle v \rangle^{2\gamma+6}}(f) \leq -\frac{c_0}{2}\int \langle v \rangle^{2\gamma+4} |\nabla^2 \log f(v)|^2 f(v)+C(1+t)^3.
     \end{equation*}
     Integrating from $0$ to $T$ gives that
     \begin{equation*}
         \int_0^T \int_{\R^3} \langle v \rangle^{2\gamma+4} |\nabla^2 \log f(v)|^2 f(v) \ud v \ud t \leq C(1+T)^4+CI_{\langle v \rangle^{2\gamma+6}}(f_0) \leq C(1+T)^4
     \end{equation*}
     by the initial assumption. Here $C$ depends on the parameters $\gamma,m$ and the initial quantities $\|f_0\|_{L_m^1},\|f_0\|_{L^\infty},I_{\langle v \rangle^{2\gamma+6}}(f_0)$. This finishes the proof of Proposition \ref{propfunctional}.

\section{Hierarchy Estimates for Correlation Functions}\label{sec:hierarchy}

Now we return to the proof of Proposition \ref{weakstarconvergence}. To derive the uniqueness of the weak limit, we need to establish the hierarchy solved by the weak limit $\bar C_n$ to derive our result. We start by presenting the BBGKY-type hierarchy solved by the weighted marginals $M_{N,n}$.

\begin{lemma}\label{marginalhierarchy}
    The weighted marginals $M_{N,n}$ solve the following hierarchy.
    \begin{align*}
        \p_t M_{N,n}=&\, -\frac{1}{2N} \sum_{i,j=1}^n (\nabla_{v^i}-\nabla_{v^j}) \cdot \Big( a(v^i-v^j) \cdot (\nabla_{v^i}-\nabla_{v^j})M_{N,n} \Big)\\
        &\, -\frac{N-n}{N} \sum_{i=1}^N \int_{\R^3} (\nabla_{v^i}-\nabla_v) \cdot \Big( a(v^i-v) \cdot (\nabla_{v^i}-\nabla_v) M_{N,n+1}(v^{[n]},v) \Big) f(v) \ud v\\
        &\, +\frac{(N-n)(N-n-1)}{N} \iint_{\R^6} V_f(v,w) M_{N,n+2}(v^{[n]},v,w) f(v) f(w) \ud v \ud w\\
        &\, +\frac{(n+1)(N-n)}{N} \int_{\R^3} \Big( a\ast f(v) :\frac{\nabla^2 f}{f}(v)-c \ast f(v) \Big) M_{N,n+1}(v^{[n]},v) f(v) \ud v.
    \end{align*}
\end{lemma}

\begin{proof}

By definition, we compute directly that
\begin{align*}
    \p_t M_{N,n} = &\, -\frac{1}{2N}\sum_{i,j=1}^{N} \int_{\R^{3(N-n)}} \left( \nabla_{v^i} - \nabla_{v^j} \right) \cdot \Big( a(v^i - v^j) \cdot \left( \nabla_{v^i} - \nabla_{v^j} \right) \Phi_N \Big)\\
    &\, \qquad \qquad \qquad \qquad \qquad \qquad \qquad \qquad  f^{\otimes (N-n)} \ud v^{n+1} \cdots \ud v^{N} \\
    &\, + \sum_{i=n+1}^{N} \int_{\R^{3(N-n)}} \left( a*f(v^i): \frac{\nabla^2 f}{f}(v^i) - c*f(v^i) \right) f^{\otimes (N-n)} \Phi_N \ud v^{n+1} \cdots \ud v^{N}.
\end{align*}
Here $f^{\otimes (N-n)}=f^{\otimes (N-n)}(v^{n+1},\cdots,v^{N})$.

Now we divide the first sum into three parts according to whether $i \leq n$ or $j \leq n$. 

\textbf{For the first case where $i,j \leq n$}, we can directly move the factorized term $f^{\otimes (N-n)}$ into the big bracket since it is independent of $v^i$ and $v^j$, and integrate with respect to the last $N-n$ variables, which leads to
\begin{equation*}
    -\frac{1}{2N}\sum_{i,j=1}^{n}\left( \nabla_{v^i} - \nabla_{v^j} \right) \cdot \Big( a(v^i - v^j) \cdot \left( \nabla_{v^i} - \nabla_{v^j} \right) M_{N,n} \Big).
\end{equation*}

\textbf{For the second case where either $i \leq n, j >n$ or $i>n, j \leq n$}, we first reduce to the former condition $i \leq n, j>n$ due to the symmetry of the expression. Now, the divergence with respect to $v^i$ commutes with the integral, while the divergence with respect to $v^j$ is treated by integration by parts. Finally we integrate with respect to the last $N-n$ variables except $v^j$ to obtain that
\begin{align*}
    -\frac{1}{N} &\, \sum_{i=1}^n \sum_{j=n+1}^N \int_{\R^3}\\
    \Big[ &\, \nabla_{v^i} \cdot \Big( a(v^i - v^j) \cdot \nabla_{v^i} M_{N,n+1} \Big) + a(v^i - v^j): (\nabla \log f(v^j) \otimes \nabla_{v^i} M_{N,n+1})\\
    +&\, \nabla_{v^i} \cdot \Big( a(v^i - v^j) \cdot \nabla \log f(v^j) M_{N,n+1} \Big)- \nabla_{v^i} \cdot \Big( b(v^i - v^j) M_{N,n+1} \Big)\\
    +&\,  a(v^i - v^j): \frac{\nabla^2 f}{f}(v^j) M_{N,n+1} - b(v^i - v^j) \cdot \nabla \log f(v^j) M_{N,n+1} \Big] f(v^j) \ud v^j.
\end{align*}
In order to simplify the above expression, we apply integration by parts to the second line of the integrand with respect to $v^j$ and exploit the relationship $\nabla \cdot a=b$ :
\begin{align*}
    &\, \int_{\R^3} \nabla_{v^i} \cdot \Big( a(v^i-v^j) \cdot \nabla f(v^j) M_{N,n+1} \Big) -\nabla_{v^i} \cdot \Big( b(v^i-v^j) M_{N,n+1} \Big) f(v^j) \ud v^j\\
    =&\, -\int_{\R^3} \nabla_{v^i} \cdot \Big( a(v^i-v^j) \cdot \nabla_{v^j} M_{N,n+1} \Big) \ud v^j.
\end{align*}
We also apply integration by parts to the last line with respect to $v^j$ and exploit the relationship $\nabla \cdot a=b$:
\begin{align*}
    &\,  \int_{\R^3} a(v^i - v^j): \nabla^2 f(v^j) M_{N,n+1} - b(v^i - v^j) \cdot \nabla f(v^j) M_{N,n+1} \ud v^j\\
    =&\, -\int_{\R^3} a(v^i-v^j): (\nabla \log f(v^j) \otimes \nabla_{v^j} M_{N,n+1}) f(v^j) \ud v^j.
\end{align*}
Combining the two partial expressions with the two terms in the first line of the integrand respectively, and using once more integration by parts with respect to $v^j$, and replacing $v^j$ with $v$, we arrive at for the second case:
\begin{align*}
    -\frac{N-n}{N} \sum_{i=1}^n \int_{\R^3} (\nabla_{v^i}-\nabla_{v}) \cdot \Big( a(v^i-v) \cdot (\nabla_{v^i}-\nabla_{v}) M_{N,n+1} \Big) f(v) \ud v.
\end{align*}
Here $M_{N,n+1}=M_{N,n+1}(v^{[n]},v)$.

\textbf{For the third case where $i,j>n$}, we use integration by parts to rewrite into
\begin{align*}
    \frac{1}{2N} \sum_{i,j=n+1}^{N} \int_{\R^{3(N-n)}} a(v^i - v^j) : \Big( \left( \nabla_{v^i} - \nabla_{v^j} \right) \Phi_N \otimes &\, ( \nabla \log f(v^i) - \nabla \log f(v^j)) \Big)\\
    &\, f^{\otimes(N-n)} \ud v^{n+1} \cdots \ud v^{N}.
\end{align*}
Playing once more integration by parts to remove the gradients in front of $\Phi_N$, and integrating with respect to the last $N-n$ variables except $v^i$ and $v^j$, we derive that
\begin{align*}
    &\, -\frac{1}{N} \sum_{i,j=n+1}^{N} \iint_{\R^6} b(v^i - v^j) \cdot \left( \nabla \log f(v^i) - \nabla \log f(v^j) \right) M_{N,n+2} \ud v^i \ud v^j\\
    &\, -\frac{1}{N} \sum_{i,j=n+1}^{N} \iint_{\R^6} a(v^i - v^j) : \nabla^2 \log f(v^i) M_{N,n+2} \ud v^i \ud v^j\\
    &\, -\frac{1}{N} \sum_{i,j=n+1}^{N} \iint_{\R^6} a(v^i - v^j) : \left( \nabla \log f(v^i) - \nabla \log f(v^j) \right) \otimes \nabla \log f(v^i) M_{N,n+2} \ud v^i \ud v^j.
\end{align*}
Replacing $v^i$ and $v^j$ with respectively $v$ and $w$, and recalling the definition of the test function $V_f(v,w)$, we arrive at
\begin{align*}
    &\, \frac{(N-n)(N-n-1)}{N} \iint_{\R^6} V_f(v,w) M_{N,n+2}(v^{[n]},v,w) f(v)f(w) \ud v \ud w\\
    -&\, \frac{(N-n)(N-n-1)}{N} \int_{\R^3} \Big( a \ast f(v) : \frac{\nabla^2 f}{f}(v)-c\ast f(v) \Big) f(v) M_{N,n+1}(v^{[n]},v) \ud v.
\end{align*}

Finally, we deal with the second sum in the original expression of $\p_t M_{N,n}$, which is much easier. Observing that $f^{\otimes (N-n)} \Phi_N$ is symmetric in $v^{n+1}, \cdots, v^{N}$, each integral in the summation equals to the same value. Together with the definition of $M_{N,n+1}$, we have the second sum coincides with
\begin{equation*}
    (N-n)\int_{\R^3} \Big( a \ast f(v) : \frac{\nabla^2 f}{f}(v)-c\ast f(v) \Big) f(v) M_{N,n+1}(v^{[n]},v) \ud v.
\end{equation*}
A simple bookkeeping would lead to the final result.

\end{proof}

Using the definition and the inversion formula, we derive the corresponding hierarchy for the correlation functions $C_{N,n}$.

\begin{lemma}\label{cumulanthierarchy}
    The correlation functions $C_{N,n}$ solve the following hierarchy.
    \begin{align*}
        \p_t C_{N,n}=T_1+T_2+T_3+T_4,
    \end{align*}
    where the four parts of expressions on the right-hand side are given respectively by
    \begin{align*}
        T_1=-\frac{1}{2N}\sum_{i,j=1}^n &\, (\nabla_{v^i}-\nabla_{v^j}) \cdot \Big[ a(v^i-v^j) \cdot (\nabla_{v^i}-\nabla_{v^j})\\
        &\, \Big( C_{N,n}(v^{[n]})+C_{N,n-1}(v^{[n]-\{i\}})+C_{N,n-1}(v^{[n]-\{j\}}) \Big) \Big],
    \end{align*}
    which is an inner term and contains derivatives;
    \begin{align*}
        T_2=\sum_{i=1}^n \int_{\R^3} &\, f(v) (\nabla_{v^i}-\nabla_v) \cdot \Big[ a(v^i-v) \cdot (\nabla_{v^i}-\nabla_v)\\
        \Big( &\, -\frac{N-n}{N} C_{N,n+1}(v^{[n]},v)+\frac{1}{N} \sum_{j \neq i}^n C_{N,n} (v^{[n]-\{j\}},v)\\
        &\, -\frac{N-n}{N} C_{N,n}(v^{[n]-\{i\}},v)+\frac{1}{N} \sum_{j \neq i}^n C_{N,n-1} (v^{[n]-\{i,j\}},v)\\
        &\, -\frac{N-n}{N} C_{N,n}(v^{[n]})+\frac{1}{N} \sum_{j \neq i}^n C_{N,n-1} (v^{[n]-\{j\}}) \Big) \Big] \ud v,
    \end{align*}
    which is an outer term and contains derivatives;
    \begin{align*}
        T_3= &\, \frac{(N-n)(N-n-1)}{N} \iint_{\R^6} V_f(v,w) f(v)f(w) C_{N,n+2}(v^{[n]},v,w) \ud v \ud w\\
        &\, -\frac{2(N-n)}{N} \sum_{i=1}^n \iint_{\R^6} V_f(v,w) f(v)f(w) C_{N,n+1}(v^{[n]-\{i\}},v,w) \ud v \ud w\\
        &\, +\frac{2}{N} \sum_{i<j}^n \iint_{\R^6} V_f(v,w) f(v)f(w) C_{N,n}(v^{[n]-\{i,j\}},v,w) \ud v \ud w,
    \end{align*}
    which does not contain derivatives and is an outer term with two more variables;
    \begin{align*}
        T_4= &\, \frac{(n+1)(N-n)}{N} \int_{\R^3} \Big( a\ast f(v) :\frac{\nabla^2 f}{f}(v)-c \ast f(v) \Big) f(v) C_{N,n+1}(v^{[n]},v) \ud v\\
        &\, +\frac{N-2n}{N} \sum_{i=1}^n \int_{\R^3} \Big( a\ast f(v) :\frac{\nabla^2 f}{f}(v)-c \ast f(v) \Big) f(v) C_{N,n}(v^{[n]-\{i\}},v) \ud v\\
        &\, -\frac{2}{N} \sum_{i<j}^n \int_{\R^3} \Big( a\ast f(v) :\frac{\nabla^2 f}{f}(v)-c \ast f(v) \Big) f(v) C_{N,n-1}(v^{[n]-\{i,j\}},v) \ud v,
    \end{align*}
    which does not contain derivatives and is an outer term with one more variable.
\end{lemma}

\begin{proof}
    The explicit formula of the hierarchy above is quite complicated, but the proof only relies on identical deformations of linear expressions. Recall the definition formula of correlation functions:
    \begin{equation*}
        C_{N,n}(v^{[n]})=\sum_{m=0}^n (-1)^{n-m} \sum_{\sigma \in P_m^n} M_{N,m}(v^\sigma).
    \end{equation*}
    We differentiate both sides with respect to $t$ and use the BBGKY-type hierarchy solved by the weighted marginals $M_{N,n}$ in Lemma \ref{marginalhierarchy} to derive the hierarchy solved by $C_{N,n}$.
    \begin{align*}
        &\, \p_t C_{N,n}=T_1+T_2+T_3+T_4\\
        \triangleq &\, - \sum_{k=0}^n (-1)^{n-k} \sum_{\sigma \in P_k^n} \frac{1}{2N} \sum_{i,j \in \sigma} (\nabla_{v^i}-\nabla_{v^j}) \cdot \Big( a(v^i-v^j) \cdot (\nabla_{v^i}-\nabla_{v^j}) M_{N,k}(v^\sigma) \Big)\\
        &\, - \sum_{k=0}^n (-1)^{n-k} \sum_{\sigma \in P_k^n} \frac{N-k}{N} \sum_{i \in \sigma} \int_{\R^3} (\nabla_{v^i}-\nabla_v) \cdot \Big( a(v^i-v) \cdot (\nabla_{v^i}-\nabla_v) M_{N,k+1}(v^\sigma,v) \Big) f(v) \ud v \\
        &\, + \sum_{k=0}^n (-1)^{n-k} \sum_{\sigma \in P_k^n} \frac{(N-k)(N-k-1)}{N} \iint_{\R^6} V_f(v,w) M_{N,k+2}(v^{\sigma},v,w) f(v)f(w) \ud v \ud w\\
        &\, + \sum_{k=0}^n (-1)^{n-k} \sum_{\sigma \in P_k^n} \frac{(k+1)(N-k)}{N} \int_{\R^3} \Big( a \ast f(v) : \frac{\nabla^2 f}{f}(v)-c \ast f(v) \Big) f(v) M_{N,k+1}(v^\sigma,v) \ud v.
    \end{align*}
    We shall compute from $T_1$ to $T_4$ explicitly. The common idea follows from \cite{bresch2024duality}, which is to expand the weighted marginals $M_{N,k}$ into correlation functions via the inversion formula. Then we change the order of summations carefully and make use of some combinatorial identities to derive the result.

    \textbf{Computation for $T_1$:} We have by the inversion formula
    \begin{equation*}
        T_1=-\frac{1}{2N} \sum_{k=0}^n (-1)^{n-k} \sum_{\sigma \in P_k^n} \sum_{i,j \in \sigma} (\nabla_{v^i}-\nabla_{v^j}) \cdot \Big( a(v^i-v^j) \cdot (\nabla_{v^i}-\nabla_{v^j}) \sum_{l=0}^k \sum_{\tau \in P_l^\sigma} C_{N,l}(v^\tau) \Big).
    \end{equation*}
    We manage to put the summation in $k$ and $\sigma$ inside, which leads to
    \begin{align*}
        -\frac{1}{2N} \sum_{i,j=1}^n (\nabla_{v^i}-\nabla_{v^j}) \cdot &\, \Big( a(v^i-v^j) \cdot (\nabla_{v^i}-\nabla_{v^j}) \\
        &\, \sum_{l=0}^n \sum_{\tau \in P_l^n}C_{N,l}(v^\tau) \cdot \sum_{k=l}^n (-1)^{n-k} \# \{ \sigma \in P_k^n: \tau \cup \lbrace i,j \rbrace \subset \sigma \} \Big).
    \end{align*}
    The cardinality of the set $\sigma$ can be easily computed, but it requires discussions among different cases, say $i \in \tau$ or not and $j \in \tau$ or not. Hence we divide the summation in the second line into four parts:
    \begin{align*}
        &\, \sum_{l=0}^n \sum_{i,j \notin \tau \in P_l^n} C_{N,l}(v^\tau) \cdot \sum_{k=l}^n (-1)^{n-k} C_{n-l-2}^{k-l-2}\\
        +&\, \sum_{l=0}^n \sum_{i \in \tau \in P_l^n, j \notin \tau} C_{N,l}(v^\tau) \cdot \sum_{k=l}^n (-1)^{n-k} C_{n-l-1}^{k-l-1}\\
        +&\, \sum_{l=0}^n \sum_{j \in \tau \in P_l^n, i \notin \tau} C_{N,l}(v^\tau) \cdot \sum_{k=l}^n (-1)^{n-k} C_{n-l-1}^{k-l-1}\\
        +&\, \sum_{l=0}^n \sum_{i,j \in \tau \in P_l^n} C_{N,l}(v^\tau) \cdot \sum_{k=l}^n (-1)^{n-k} C_{n-l}^{k-l}.
    \end{align*}
    Using the elementary combinatorial identity
    \begin{equation*}
        \sum_{k=l}^n (-1)^{n-k} C_{n-l-m}^{n-k}=
        \begin{cases}
            1 \quad l=n-m ;\\
            0 \quad else,
        \end{cases}
    \end{equation*}
    and noticing that the first part vanishes under $\nabla_{v^i}-\nabla_{v^j}$ since it is independent with $v^i$ and $v^j$, we arrive at
    \begin{align*}
        T_1=-\frac{1}{2N}\sum_{i,j=1}^n &\, (\nabla_{v^i}-\nabla_{v^j}) \cdot \Big[ a(v^i-v^j) \cdot (\nabla_{v^i}-\nabla_{v^j})\\
        &\, \Big( C_{N,n}(v^{[n]})+C_{N,n-1}(v^{[n]-\{i\}})+C_{N,n-1}(v^{[n]-\{j\}}) \Big) \Big].
    \end{align*}

    \textbf{Computation for $T_2$:} We have by the inversion formula
    \begin{align*}
        T_2=-\sum_{k=0}^n (-1)^{n-k} \sum_{\sigma \in P_k^n} \sum_{i \in \sigma} \frac{N-k}{N} \int_{\R^3} (\nabla_{v^i}-\nabla_v) \cdot \Big[ a(v^i-v) \cdot (\nabla_{v^i}-\nabla_v)\\
        \Big( \sum_{l=0}^k \sum_{\tau \in P_l^\sigma} C_{N,l+1}(v^\tau,v)+\sum_{l=0}^k \sum_{\tau \in P_l^\sigma} C_{N,l}(v^{\tau}) \Big) \Big] f(v) \ud v.
    \end{align*}
    We manage to put the summation in $k$ and $\sigma$ inside, which leads to
    \begin{align*}
        -\sum_{i=1}^n \int_{\R^3} f(v) (\nabla_{v^i}&\, -\nabla_v) \cdot \Big[ a(v^i-v) \cdot (\nabla_{v^i}-\nabla_v)\\ \sum_{l=0}^n \sum_{\tau \in P_l^n}
        &\, \Big( C_{N,l+1}(v^\tau,v) \cdot \sum_{k=l}^n (-1)^{n-k} \frac{N-k}{N} \# \{ \sigma \in P_k^n: \tau \cup \{i\} \subset \sigma \}\\
        &\, \quad +C_{N,l}(v^\tau) \cdot \sum_{k=l}^n (-1)^{n-k} \frac{N-k}{N} \# \{ \sigma \in P_k^n: \tau \cup \{i\} \subset \sigma \} \Big) \Big] \ud v.
    \end{align*}
    The cardinality of the set $\sigma$ can be easily computed, but it requires discussions between different cases, say $i \in \tau$ or not. Hence we divide the summation in the last two lines into four parts:
    \begin{align*}
        &\, \sum_{l=0}^n \sum_{i \in \tau \in P_l^n} C_{N,l+1}(v^\tau,v) \cdot \sum_{k=l}^n (-1)^{n-k} \frac{N-k}{N} C_{n-l}^{k-l}\\
        +&\, \sum_{l=0}^n \sum_{i \notin \tau \in P_l^n} C_{N,l+1}(v^\tau,v) \cdot \sum_{k=l}^n (-1)^{n-k} \frac{N-k}{N} C_{n-l-1}^{k-l-1}\\
        +&\, \sum_{l=0}^n \sum_{i \in \tau \in P_l^n} C_{N,l}(v^\tau) \cdot \sum_{k=l}^n (-1)^{n-k} \frac{N-k}{N} C_{n-l}^{k-l}\\
        +&\, \sum_{l=0}^n \sum_{i \notin \tau \in P_l^n} C_{N,l}(v^\tau) \cdot \sum_{k=l}^n (-1)^{n-k} \frac{N-k}{N} C_{n-l-1}^{k-l-1}.
    \end{align*}
    Using the elementary combinatorial identity
    \begin{equation*}
        \sum_{k=l}^n (-1)^{n-k} \frac{N-k}{N} C_{n-l-m}^{n-k}=\begin{cases}
        \frac{N-n}{N} \quad l=n-m;\\
        -\frac{1}{N} \quad l=n-m-1;\\
        0 \quad else,
        \end{cases}
    \end{equation*}
    and noticing that the last part vanishes under $\nabla_{v^i}-\nabla_v$ since it is independent with $v^i$ and $v$, we arrive at
    \begin{align*}
        T_2=\sum_{i=1}^n \int_{\R^3} &\, f(v) (\nabla_{v^i}-\nabla_v) \cdot \Big[ a(v^i-v) \cdot (\nabla_{v^i}-\nabla_v)\\
        \Big( &\, -\frac{N-n}{N} C_{N,n+1}(v^{[n]},v)+\frac{1}{N} \sum_{j \neq i}^n C_{N,n} (v^{[n]-\{j\}},v)\\
        &\, -\frac{N-n}{N} C_{N,n}(v^{[n]-\{i\}},v)+\frac{1}{N} \sum_{j \neq i}^n C_{N,n-1} (v^{[n]-\{i,j\}},v)\\
        &\, -\frac{N-n}{N} C_{N,n}(v^{[n]})+\frac{1}{N} \sum_{j \neq i}^n C_{N,n-1} (v^{[n]-\{j\}}) \Big) \Big] \ud v.
    \end{align*}

    \textbf{Computation for $T_3$:} We have by the inversion formula
    \begin{align*}
        T_3=\sum_{k=0}^n (-1)^{n-k} \sum_{\sigma \in P_k^n} \frac{(N-k)(N-k-1)}{N} &\, \iint_{\R^6} V_f(v,w) f(v) f(w)\\
        &\, \sum_{l=0}^k \sum_{\tau \in P_l^\sigma} C_{N,l+2}(v^\tau,v,w) \ud v \ud w.
    \end{align*}
    We manage to put the summation in $k$ and $\sigma$ inside, which leads to
    \begin{align*}
        \sum_{l=0}^n \sum_{\tau \in P_l^n} \iint_{\R^6} &\, V_f(v,w) f(v) f(w) C_{N,l+2}(v^\tau,v,w) \\
        &\, \sum_{k=l}^n (-1)^{n-k} \frac{(N-k)(N-k-1)}{N} \# \{ \sigma \in P_k^n: \tau \subset \sigma \} \ud v \ud w.
    \end{align*}
    The cardinality of the set $\sigma$ can be easily computed, which is equal to $C_{n-l}^{k-l}$. Hence the summation in the last line becomes
    \begin{align*}
        \sum_{k=l}^n (-1)^{n-k} \frac{(N-k)(N-k-1)}{N} C_{n-l}^{k-l}.
    \end{align*}
    Using the elementary combinatorial identity, this summation equals $\frac{(N-n)(N-n-1)}{N}$ if $l=n$, equals $\frac{-2(N-n)}{N}$ if $l=n-1$, equals $\frac{2}{N}$ if $l=n-2$, and vanishes in the remaining cases. We arrive at
    \begin{align*}
        T_3= &\, \frac{(N-n)(N-n-1)}{N} \iint_{\R^6} V_f(v,w) f(v)f(w) C_{N,n+2}(v^{[n]},v,w) \ud v \ud w\\
        &\, -\frac{2(N-n)}{N} \sum_{i=1}^n \iint_{\R^6} V_f(v,w) f(v)f(w) C_{N,n+1}(v^{[n]-\{i\}},v,w) \ud v \ud w\\
        &\, +\frac{2}{N} \sum_{i<j}^n \iint_{\R^6} V_f(v,w) f(v)f(w) C_{N,n}(v^{[n]-\{i,j\}},v,w) \ud v \ud w.
    \end{align*}

    \textbf{Computation for $T_4$:} We have by the inversion formula
    \begin{align*}
        T_4=\sum_{k=0}^n (-1)^{n-k} \sum_{\sigma \in P_k^n} \frac{(k+1)(N-k)}{N} \int_{\R^3} \Big( a \ast f(v):\frac{\nabla^2 f}{f}(v)-c \ast f(v) \Big) f(v) \\
        \sum_{l=0}^k \sum_{\tau \in P_l^\sigma} C_{N,l+1}(v^\tau,v) \ud v.
    \end{align*}
    We manage to put the summation in $k$ and $\sigma$ inside, which leads to
    \begin{align*}
        \sum_{l=0}^n \sum_{\tau \in P_l^n} \int_{\R^3} &\, \Big( a \ast f(v):\frac{\nabla^2 f}{f}(v)-c \ast f(v) \Big) f(v) C_{N,l+1}(v^\tau,v) \\
        &\, \sum_{k=l}^n (-1)^{n-k} \frac{(k+1)(N-k)}{N} \# \{ \sigma \in P_k^n: \tau \subset \sigma \} \ud v.
    \end{align*}
    The cardinality of the set $\sigma$ can be easily computed, which is equal to $C_{n-l}^{k-l}$. Hence the summation in the last line becomes
    \begin{equation*}
        \sum_{k=l}^n (-1)^{n-k} \frac{(k+1)(N-k)}{N} C_{n-l}^{k-l}.
    \end{equation*}
    Using the elementary combinatorial identity, this summation equals $\frac{(n+1)(N-n)}{N}$ if $l=n$, equals $\frac{N-2n}{N}$ if $l=n-1$, equals $-\frac{2}{N}$ if $l=n-2$, and vanishes in the remaining cases. We arrive at
    \begin{align*}
        T_4= &\, \frac{(n+1)(N-n)}{N} \int_{\R^3} \Big( a\ast f(v) :\frac{\nabla^2 f}{f}(v)-c \ast f(v) \Big) f(v) C_{N,n+1}(v^{[n]},v) \ud v\\
        &\, +\frac{N-2n}{N} \sum_{i=1}^n \int_{\R^3} \Big( a\ast f(v) :\frac{\nabla^2 f}{f}(v)-c \ast f(v) \Big) f(v) C_{N,n}(v^{[n]-\{i\}},v) \ud v\\
        &\, -\frac{2}{N} \sum_{i<j}^n \int_{\R^3} \Big( a\ast f(v) :\frac{\nabla^2 f}{f}(v)-c \ast f(v) \Big) f(v) C_{N,n-1}(v^{[n]-\{i,j\}},v) \ud v.
    \end{align*}

    A simple bookkeeping finishes the proof.
\end{proof}

For our purpose, we do not need to analyze the complicated hierarchy directly, but only apply this result to derive the limit hierarchy by passing $N$ to infinity. Since we have the square-integrability of the test function $V_f(v,w)$ and $V_f(v,v)$ in Proposition \ref{squareintegrable}, we can pass to the weak limit in the BBGKY-type hierarchy solved by the correlation functions $C_{N,n}$ to obtain the limit hierarchy, which is solved by $\bar C_n$ in the weak sense.

\begin{cor}[Limit Hierarchy]\label{cor:limithierarchy}
    The weak limit $\bar C_n$ of the rescaled correlation functions solves the following limit hierarchy.
    \begin{equation}\label{limithierarchy}
        \begin{aligned}
            \p_t \bar C_n=&\, -\sum_{i=1}^n \int_{\R^3} (\nabla_{v^i}-\nabla_{v}) \cdot \Big( a(v^i-v) \cdot \nabla_{v^i} \bar C_n(v^{[n]}) \Big) f(v) \ud v\\
        &\, +\sum_{i=1}^n \int_{\R^3} (\nabla_{v^i}-\nabla_v) \cdot \Big( a(v^i-v) \cdot \nabla_v \bar C_n(v^{[n]-\{i\}},v) \Big) f(v) \ud v\\
        &\, -\sum_{i=1}^n \int_{\R^3} a(v^i-v): (\nabla \log f(v)-\nabla \log f(v^i)) \otimes \nabla_v \bar C_n(v^{[n]-\{i\}},v) f(v) \ud v\\
        &\, +\sqrt{(n+1)(n+2)} \iint_{\R^6} V_f(v,w) \bar C_{n+2}(v^{[n]},v,w) f(v) f(w) \ud v \ud w\\
        &\, +\sum_{i=1}^n \int_{\R^3} V_f(v,v^i) \bar C_n(v^{[n]-\{i\}},v) f(v) \ud v.
        \end{aligned}
    \end{equation}
\end{cor}
\begin{rmk}
    The limit hierarchy \eqref{limithierarchy} has the structural advantage as stated in Subsection \ref{subsection:methodology}. It does not contain any derivatives of higher order terms like $\nabla \bar C_{n+1}$, as opposed to the classical BBGKY hierarchy. Moreover, the last two lines are easy to be controlled since they do not contain derivatives, and the derivatives that appear in the second and third lines can be well controlled by the dissipation term in the first line.
\end{rmk}
\begin{proof}
    We pass to the limit in Lemma \ref{cumulanthierarchy}, which shows that
    \begin{equation*}
        \begin{aligned}
            \p_t \bar C_n=&\, -\sum_{i=1}^n \int_{\R^3} (\nabla_{v^i}-\nabla_{v}) \cdot \Big( a(v^i-v) \cdot \nabla_{v^i} \bar C_n(v^{[n]}) \Big) f(v) \ud v\\
        &\, +\sum_{i=1}^n \int_{\R^3} (\nabla_{v^i}-\nabla_v) \cdot \Big( a(v^i-v) \cdot \nabla_v \bar C_n(v^{[n]-\{i\}},v) \Big) f(v) \ud v\\
        &\, +\sqrt{(n+1)(n+2)} \iint_{\R^6} V_f(v,w) \bar C_{n+2}(v^{[n]},v,w) f(v) f(w) \ud v \ud w\\
        &\, +\sum_{i=1}^n \int_{\R^3} \Big( a\ast f(v) :\frac{\nabla^2 f}{f}(v)-c \ast f(v) \Big) \bar C_n(v^{[n]-\{i\}},v) f(v) \ud v.
        \end{aligned}
    \end{equation*}
    We also recall the explicit form of the test function
    \begin{equation*}
        \begin{aligned}
        V_f(v,w)=&\, -a(v-w): \Big( \frac{\nabla^2 f}{f}(v)-\nabla \log f(v) \otimes \nabla \log f(w) \Big)\\
        &\, -b(v-w) \cdot (\nabla \log f(v)-\nabla \log f(w))\\
        &\, +a \ast f(v):\frac{\nabla^2 f}{f}(v)-c \ast f(v).
        \end{aligned}
    \end{equation*}
    In order to meet the last line of \eqref{limithierarchy}, we wish to change the term in the big bracket in the last line of the limit hierarchy into $V_f(v,v^i)$. This will create an extra term
    \begin{align*}
        &\, \sum_{i=1}^n \int_{\R^3} a(v-v^i): \Big( \frac{\nabla^2 f}{f}(v)-\nabla \log f(v) \otimes \nabla \log f(v^i) \Big) f(v) \bar C_n(v^{[n]-\{i\}},v) \ud v\\
        +&\, \sum_{i=1}^n \int_{\R^3} b(v-v^i) \cdot (\nabla \log f(v)-\nabla \log f(v^i)) f(v) \bar C_n(v^{[n]-\{i\}},v) \ud v.
    \end{align*}
    Simple observation gives that it equals
    \begin{equation*}
        \sum_{i=1}^n \int_{\R^3} \nabla_v \cdot \Big[ a(v-v^i) \cdot \Big( \nabla \log f(v)-\nabla \log f(v^i) \Big) f(v) \Big] \bar C_n(v^{[n]-\{i\}},v) \ud v.
    \end{equation*}
    Using integration by parts, the extra term becomes
    \begin{align*}
        -\sum_{i=1}^n \int_{\R^3} a(v-v^i): (\nabla \log f(v)-\nabla \log f(v^i)) \otimes \nabla_v \bar C_n(v^{[n]-\{i\}},v) f(v) \ud v.
    \end{align*}
    A little bookkeeping leads to the desired limit hierarchy.
\end{proof}

We also investigate the final data for the limit hierarchy \eqref{limithierarchy}.

\begin{lemma}[Final Data]\label{limitfinal}
    The final data of the limit hierarchy \eqref{limithierarchy} is given by
    \begin{equation*}
        \bar C_n(T,v^{[n]})=
        \begin{cases}
            \Big( \int_{\R^3} \varphi(v) f(T,v) \ud v \Big)^k \quad &\, n=0,\\
            0 \quad &\, n \geq 1.
        \end{cases}
    \end{equation*}
    We emphasize that the final data are constant.
\end{lemma}

\begin{proof}
    Let us start by computing the final data for the weighted marginals $M_{N,n}(T)$. 
    
    Recall that the final data for the backward dual solutions $\Phi_N(T)$ are given by
    \begin{equation*}
        \Phi_N(T)=(C_N^k)^{-1} \sum_{1 \leq i_1< \cdots <i_k \leq N} \varphi(v^{i_1}) \cdots \varphi(v^{i_k}).
    \end{equation*}
    Substituting into the integral defining $M_{N,n}$, we obtain that
    \begin{equation*}
        M_{N,n}(T,v^{[n]})=(C_N^k)^{-1} \sum_{1 \leq i_1< \cdots <i_k \leq N} \int_{\R^{3(N-n)}} \varphi(v^{i_1}) \cdots \varphi(v^{i_k}) f^{\otimes (N-n)} \ud v^{n+1} \cdots \ud v^N.
    \end{equation*}
    By considering the number of indices in $i_1, \cdots, i_k$ which are less than or equal to $n$, denoted by $0 \leq l \leq k \wedge n$, the expression above is simplified into
    \begin{equation*}
        (C_N^k)^{-1} \sum_{l=0}^{k \wedge n} \Big( \int_{\R^3} \varphi(v) f(T,v) \ud v \Big)^{k-l} \cdot C_{N-n}^{k-l} \sum_{\sigma \in P_l^n} \varphi^{\otimes l}(v^\sigma).
    \end{equation*}

    Now we insert the above terms into the definition of the correlation functions $C_{N,n}$. The final data read
    \begin{align*}
        C_{N,n}(T,v^{[n]})=&\, \sum_{m=0}^n (-1)^{n-m} \sum_{\sigma \in P_m^n} M_{N,m}(T,v^\sigma)\\
        =&\, \sum_{m=0}^n (-1)^{n-m} \sum_{\sigma \in P_m^n} (C_N^k)^{-1} \sum_{l=0}^{k \wedge m} \Big( \int_{\R^3} \varphi(v) f(T,v) \ud v \Big)^{k-l} \cdot C_{N-m}^{k-l} \sum_{\tau \in P_l^\sigma} \varphi^{\otimes l}(v^\tau)\\
        =&\, \sum_{l=0}^{k \wedge n} \sum_{\tau \in P_l^n} \Big( \int_{\R^3} \varphi(v) f(T,v) \ud v \Big)^{k-l} \varphi^{\otimes l}(v^\tau)\\
        &\, \qquad \qquad \cdot (C_N^k)^{-1} \sum_{m=l}^n (-1)^{n-m} C_{N-m}^{k-l} \# \{ \sigma \in P_m^n: \tau \subset \sigma \}.
    \end{align*}
    The cardinality of the set $\sigma$ equals $C_{n-l}^{m-l}$, and the summation in $m$ equals $(-1)^{n+l} C_{N-n}^{k-n}$ by the elementary combinatorial identities as above. Hence, the final data of the correlation functions are given by
    \begin{equation*}
        C_{N,n}(T,v^{[n]})=
        \begin{cases}
            (C_N^n)^{-1} C_k^n \sum_{l=0}^n (-1)^{n+l} \sum_{\tau \in P_l^n} \Big( \int_{\R^3} \varphi(v) f(T,v) \ud v \Big)^{k-l} \varphi^{\otimes l}(v^\tau) \quad &\, n \leq k,\\
            0 \quad &\, n>k.
        \end{cases}
    \end{equation*}

    Finally, we pass to the limit by taking $N=N_l \rightarrow \infty$ and obtain the final data for the weak limits $\bar C_n$. Notice that for any fixed $k,n$, the above expression is $O(N^{-n})$. Hence the only nontrivial case for the final data of the weak limits is $n=0$, and the explicit result is given by
    \begin{equation*}
        \bar C_n(T,v^{[n]})=
        \begin{cases}
            \Big( \int_{\R^3} \varphi(v) f(T,v) \ud v \Big)^k \quad &\, n=0,\\
            0 \quad &\, n \geq 1.
        \end{cases}
    \end{equation*}
\end{proof}

Finally, we apply the classical energy method to establish the uniqueness for the limit hierarchy \eqref{limithierarchy}. The natural energy here is given by the weighted square norm in the functional space $L^2(f^{\otimes n} \ud v^{[n]})$.

\begin{lemma}[Uniqueness of the Limit Hierarchy]\label{uniqueness}
    Assume that $\bar C_n$ solves the limit hierarchy \eqref{limithierarchy} in the weak sense on any time interval $[0,T]$, with the zero final data $\bar C_n(T)=0$, and the following \textit{a priori} uniform bound holds:
    \begin{equation*}
        \sup_{n \geq 0} \| \bar C_n \|_{L^\infty([0,T]; L^2(f^{\otimes n} \ud v^{[n]}))} <\infty.
    \end{equation*}
    Then we have $\bar C_n \equiv 0$.
\end{lemma}

\begin{proof}
    Recall that the limit hierarchy \eqref{limithierarchy} is given by
    \begin{equation*}
        \begin{aligned}
            \p_t \bar C_n=&\, -\sum_{i=1}^n \int_{\R^3} (\nabla_{v^i}-\nabla_{v}) \cdot \Big( a(v^i-v) \cdot \nabla_{v^i} \bar C_n(v^{[n]}) \Big) f(v) \ud v\\
        &\, +\sum_{i=1}^n \int_{\R^3} (\nabla_{v^i}-\nabla_v) \cdot \Big( a(v^i-v) \cdot \nabla_v \bar C_n(v^{[n]-\{i\}},v) \Big) f(v) \ud v\\
        &\, -\sum_{i=1}^n \int_{\R^3} a(v^i-v): (\nabla \log f(v)-\nabla \log f(v^i)) \otimes \nabla_v \bar C_n(v^{[n]-\{i\}},v) f(v) \ud v\\
        &\, +\sqrt{(n+1)(n+2)} \iint_{\R^6} V_f(v,w) \bar C_{n+2}(v^{[n]},v,w) f(v) f(w) \ud v \ud w\\
        &\, +\sum_{i=1}^n \int_{\R^3} V_f(v,v^i) \bar C_n(v^{[n]-\{i\}},v) f(v) \ud v.
        \end{aligned}
    \end{equation*}
    Applying the weak-strong argument, we compute the time derivative of the weighted square norm of $\bar C_n$.
    \begin{equation}\label{derivativenorm}
        \begin{aligned}
            \frac{1}{2}\frac{\ud}{\ud t} \int_{\R^{3n}} |\bar C_n|^2 &\, f^{\otimes n} \ud v^{[n]}= \int_{\R^{3n}} (\bar C_n \cdot \p_t \bar C_n) f^{\otimes n} \ud v^{[n]}\\
         +&\, \frac{1}{2}\int_{\R^{3n}} |\bar C_n|^2 \sum_{i=1}^n \nabla \cdot \Big( (a \ast f) \cdot \nabla f-(b \ast f)f \Big)(v^i) f^{\otimes (n-1)}(v^{[n]-\{i\}}) \ud v^{[n]}.
        \end{aligned}
    \end{equation}
    Integrating by parts on the second line gives
    \begin{align*}
        &\, -\sum_{i=1}^n \int_{\R^{3n}} a\ast f(v^i):\nabla_{v^i} \bar C_n(v^{[n]}) \otimes \nabla \log f(v^i) \bar C_n(v^{[n]}) f^{\otimes n} \ud v^{[n]}\\
        &\, +\sum_{i=1}^n \int_{\R^{3n}} b \ast f(v^i) \cdot \nabla_{v^i} \bar C_n(v^{[n]}) \bar C_n(v^{[n]}) f^{\otimes n} \ud v^{[n]}.
    \end{align*}
    
    As for the first line, we shall insert the five terms of the limit hierarchy into the time derivative $\p_t \bar C_n$ respectively. Roughly speaking, inserting the first term and integrating by parts will cancel the second line on the right-hand side of \eqref{derivativenorm} stated above and create a positive term. 
    \begin{align*}
        &\, -\sum_{i=1}^n \int_{\R^{3(n+1)}} \bar C_n (\nabla_{v^i}-\nabla_v) \cdot \Big( a(v^i-v) \cdot \nabla_{v^i} \bar C_n(v^{[n]}) \Big) f(v) f^{\otimes n} \ud v \ud v^{[n]}\\
        =&\, -\sum_{i=1}^n \int_{\R^{3(n+1)}} b(v^i-v) \cdot \nabla_{v^i} \bar C_n(v^{[n]}) \bar C_n(v^{[n]}) f(v) f^{\otimes n} \ud v \ud v^{[n]} \\
        &\, +\sum_{i=1}^n \int_{\R^{3(n+1)}} a(v^i-v): \nabla_{v^i} \bar C_n(v^{[n]}) \otimes \nabla_{v^i} \bar C_n(v^{[n]}) f(v) f^{\otimes n} \ud v \ud v^{[n]}\\
        &\, +\sum_{i=1}^n \int_{\R^{3(n+1)}} a(v^i-v): \nabla_{v^i} \bar C_n(v^{[n]}) \otimes \nabla \log f(v^i) \bar C_n(v^{[n]}) f(v) f^{\otimes n} \ud v \ud v^{[n]}\\
        =&\, -\sum_{i=1}^n \int_{\R^{3n}} b \ast f(v^i) \cdot \nabla_{v^i} \bar C_n(v^{[n]}) \bar C_n(v^{[n]}) f^{\otimes n} \ud v^{[n]} \\
        &\, +\sum_{i=1}^n \int_{\R^{3n}} a\ast f(v^i):\nabla_{v^i} \bar C_n(v^{[n]}) \otimes \nabla \log f(v^i) \bar C_n(v^{[n]}) f^{\otimes n} \ud v^{[n]}\\
        &\, +\sum_{i=1}^n \int_{\R^{3n}} a\ast f(v^i):\nabla_{v^i} \bar C_n(v^{[n]}) \otimes \nabla_{v^i} \bar C_n(v^{[n]}) f^{\otimes n} \ud v^{[n]}.
    \end{align*}
    Combining with the second line of \eqref{derivativenorm} which has been computed as above, the only term left is the final positive term.
    
    Inserting the second and third terms will also create some cancellations. Indeed, we compute the integral after inserting the second term using integration by parts.
    \begin{align*}
        &\, \sum_{i=1}^n \int_{\R^{3(n+1)}} (\nabla_{v^i}-\nabla_v) \cdot \Big( a(v^i-v) \cdot \nabla_v \bar C_n(v^{[n]-\{i\}},v) \Big) \bar C_n(v^{[n]}) f(v) f^{\otimes n} \ud v \ud v^{[n]}\\
        =&\, \sum_{i=1}^n \int_{\R^{3(n+1)}} \Big( a(v^i-v): \nabla_v \bar C_n(v^{[n]-\{i\}},v) \otimes \nabla f(v) \Big) \bar C_n(v^{[n]}) f^{\otimes n} \ud v \ud v^{[n]}    \\
        -&\, \sum_{i=1}^n \int_{\R^{3(n+1)}} \Big( a(v^i-v): \nabla_v \bar C_n(v^{[n]-\{i\}},v) \otimes \nabla \log f(v^i) \Big) \bar C_n(v^{[n]}) f(v) f^{\otimes n} \ud v \ud v^{[n]}\\
        -&\, \sum_{i=1}^n \int_{\R^{3(n+1)}} \Big( a(v^i-v): \nabla_v \bar C_n(v^{[n]-\{i\}},v) \otimes \nabla_{v^i} \bar C_n(v^{[n]}) \Big) f(v) f^{\otimes n} \ud v \ud v^{[n]}.
    \end{align*}
    The first two lines cancel exactly with the integral after inserting the third term into the time derivative $\p_t \bar C_n$, leaving only the final line.

    After inserting the remaining two terms into $\p_t \bar  C_n$ and a little bookkeeping, we compute that
    \begin{align*}
        &\, \frac{1}{2} \frac{\ud}{\ud t} \int_{\R^{3n}} |\bar C_n|^2 f^{\otimes n} \ud v^{[n]}\\
        =&\, \sum_{i=1}^n \int_{\R^{3n}} a\ast f(v^i):\nabla_{v^i} \bar C_n(v^{[n]}) \otimes \nabla_{v^i} \bar C_n(v^{[n]}) f^{\otimes n} \ud v^{[n]}\\
        -&\, \sum_{i=1}^n \int_{\R^{3(n+1)}} \Big( a(v^i-v): \nabla_v \bar C_n(v^{[n]-\{i\}},v) \otimes \nabla_{v^i} \bar C_n(v^{[n]}) \Big) f(v) f^{\otimes n} \ud v \ud v^{[n]}\\
        +&\, \sum_{i=1}^n \int_{\R^{3(n+1)}} V_f(v,v^i) \bar C_n(v^{[n]-\{i\}},v) \bar C_n(v^{[n]}) f(v) f^{\otimes n} \ud v \ud v^{[n]}\\
        +&\, \sqrt{(n+1)(n+2)} \int_{\R^{3(n+2)}} V_f(v,w) \bar C_n(v^{[n]}) \bar C_{n+2}(v^{[n]},v,w) f(v)f(w) f^{\otimes n} \ud v \ud w \ud v^{[n]}.
    \end{align*}
    Next we show that the sum of the first two terms is nonnegative. In fact, since the coefficient matrix $a(v^i-v)$ is nonnegative definite, we apply the Cauchy-Schwarz inequality to control the second term.
    \begin{align*}
        &\, \sum_{i=1}^n \int_{\R^{3(n+1)}} \Big( a(v^i-v): \nabla_v \bar C_n(v^{[n]-\{i\}},v) \otimes \nabla_{v^i} \bar C_n(v^{[n]}) \Big) f(v) f^{\otimes n} \ud v \ud v^{[n]}\\
        \leq &\, \frac{1}{2} \sum_{i=1}^n \int_{\R^{3(n+1)}} \Big( a(v^i-v): \nabla_v \bar C_n(v^{[n]-\{i\}},v) \otimes \nabla_v \bar C_n(v^{[n]-\{i\}},v) \Big) f(v) f^{\otimes n} \ud v \ud v^{[n]}\\
        +&\, \frac{1}{2} \sum_{i=1}^n \int_{\R^{3(n+1)}} \Big( a(v^i-v): \nabla_{v^i} \bar C_n(v^{[n]}) \otimes \nabla_{v^i} \bar C_n(v^{[n]}) \Big) f(v) f^{\otimes n} \ud v \ud v^{[n]}.
    \end{align*}
    For the right-hand side, we integrate first with respect to $v^i$ in the first term and integrate first with respect to $v$ in the second term.
    \begin{align*}
        =&\, \frac{1}{2} \sum_{i=1}^n \int_{\R^{3n}} \Big( a \ast f(v): \nabla_v \bar C_n(v^{[n]-\{i\}},v) \otimes \nabla_v \bar C_n(v^{[n]-\{i\}},v) \Big) f(v) f^{\otimes (n-1)}(v^{[n]-\{i\}}) \ud v \ud v^{[n]-\{i\}}\\
        +&\, \frac{1}{2} \sum_{i=1}^n \int_{\R^{3n}} \Big( a \ast f(v^i): \nabla_{v^i} \bar C_n(v^{[n]}) \otimes \nabla_{v^i} \bar C_n(v^{[n]}) \Big) f^{\otimes n} \ud v^{[n]}\\
        =&\, \sum_{i=1}^n \int_{\R^{3n}} \Big( a \ast f(v^i): \nabla_{v^i} \bar C_n(v^{[n]}) \otimes \nabla_{v^i} \bar C_n(v^{[n]}) \Big) f^{\otimes n} \ud v^{[n]},
    \end{align*}
    where the last step follows from simple change of notations $v \rightarrow v^i$.

    For the two remaining terms in the time derivative formula, we make the simple change of notations $v,w \rightarrow v^{n+1},v^{n+2}$ and use the symmetry property.
    \begin{align*}
        =&\, n \iint_{\R^6} V_f(v^n,v^{n+1}) f^{\otimes 2} \ud v^{n} \ud v^{n+1} \int_{\R^{3(n-1)}} \bar C_n(v^{[n-1]},v^{n+1}) \bar C_n(v^{[n-1]},v^{n}) f^{\otimes (n-1)} \ud v^{[n-1]}\\
        &\, +\sqrt{(n+1)(n+2)} \int_{\R^{3(n+2)}} V_f(v^{n+1},v^{n+2}) \bar C_n(v^{[n]}) \bar C_{n+2}(v^{[n+2]}) f^{\otimes (n+2)} \ud v^{[n+2]}.
    \end{align*}
    Denote by $\Lambda_f(t)=\Big (\int |V_f|^2 f^{\otimes 2} \Big)^{\frac{1}{2}}$. We use the Cauchy-Schwarz inequality to bound the first line from below by
    \begin{align*}
        &\, -n \Lambda_f \Bigg( \iint_{\R^6} \Bigg| \int_{\R^{3(n-1)}} \bar C_n(v^{[n-1]},v^{n}) \bar C_n(v^{[n-1]},v^{n+1})f^{\otimes (n-1)} \Bigg|^2 f(v^{n})f(v^{n+1}) \Bigg)^{\frac{1}{2}}\\
        \geq &\, -n \Lambda_f \Bigg( \iint_{\R^6} \Big( \int_{\R^{3(n-1)}} |\bar C_n(v^{[n-1]},v^{n})|^2 f^{\otimes(n-1)} \Big)\\
        &\, \qquad \qquad \qquad \, \Big( \int_{\R^{3(n-1)}} |\bar C_n(v^{[n-1]},v^{n+1})|^2 f^{\otimes (n-1)} \Big) f(v^{n})f(v^{n+1}) \Bigg)^{\frac{1}{2}}\\
        =&\, -n \Lambda_f \int_{\R^{3n}} |\bar C_n|^2 f^{\otimes n} \ud v^{[n]};
    \end{align*}
    and to bound the second line from below by
    \begin{align*}
        &\, -(n+2)  \Lambda_f \Bigg( \iint_{\R^6} \Bigg| \int_{\R^{3n}} \bar C_n(v^{[n]}) \bar C_{n+2}(v^{[n+2]}) f^{\otimes n} \Bigg|^2 f(v^{n+1})f(v^{n+2}) \Bigg)^{\frac{1}{2}}\\
        \geq &\, -(n+2) \Lambda_f \Bigg( \int_{\R^{3n}} |\bar C_n|^2 f^{\otimes n} \ud v^{[n]} \Bigg)^{\frac{1}{2}} \Bigg( \int_{\R^{3(n+2)}} |\bar C_{n+2}|^2 f^{\otimes (n+2)} \ud v^{[n+2]} \Bigg)^{\frac{1}{2}}.
    \end{align*}

    Combining all discussions above, we conclude that
    \begin{equation*}
        \frac{\ud}{\ud t} \| \bar C_n \|_{L^2(f^{\otimes n} \ud v^{[n]})} \geq -n \Lambda_f \| \bar C_n \|_{L^2(f^{\otimes n} \ud v^{[n]})} -(n+2) \Lambda_f \| \bar C_{n+2} \|_{L^2(f^{\otimes (n+2)} \ud v^{[n+2]})}.
    \end{equation*}
    We use the generating function method to solve this ODE hierarchy. For any $r \in [0,1)$, define the generating function
    \begin{equation*}
        P(t,r)=\sum_{n=0}^\infty r^n \| \bar C_n \|_{L^2(f^{\otimes n} \ud v^{[n]})}
    \end{equation*}
    which converges absolutely uniformly in $t$ due to the \textit{a priori} uniform bound. Now we have the differential inequality
    \begin{equation*}
        \p_t P(t,r) \geq -\Lambda_f(t) (r\p_r P(t,r)+r^{-1}\p_r P(t,r)) \geq -2 \Lambda_f(t) r^{-1} \p_r P(t,r),
    \end{equation*}
    with the final data $P(T,r)=0$. Fix any $0<r<\frac{1}{2}$, for any $t \in [T_0,T]$ with some $T_0$ to be defined, let
    \begin{equation*}
        r(t)= \Big(r^2+\int_{T_0}^t 4 \Lambda_f(s) \ud s \Big)^{\frac{1}{2}}
    \end{equation*}
    be the characteristic of the inequality. Thanks to the square-integrability condition \eqref{squareintegrable} of $V_f$, we have
    \begin{equation*}
        \int_{T_0}^t \Lambda_f(s) \ud s \leq C \sqrt{(1+T)(t-T_0)},
    \end{equation*}
    which allows us to choose $T_0<T$ uniformly in $r$ such that $r(T)<1$ . It is straightforward to check that
    \begin{equation*}
        \frac{\ud}{\ud t} P(t,r(t)) \geq 0,
    \end{equation*}
    which yields $P(t,r) \leq P(T,r(T))=0$. Hence $\bar C_n=0$ for any $t \in [T_0,T]$. Iterating this process will give $\bar C_n \equiv 0$.
\end{proof}

Notice that the constant solutions
\begin{equation*}
    \bar C_n(t,v^{[n]}) \equiv \bar C_n(T,v^{[n]})
\end{equation*}
solve the limit hierarchy \eqref{limithierarchy} trivially. We can now conclude that the weak limit $\bar C_2=0$, which completes the proof of Theorem \ref{thm:poc}.

\section{Entropic Chaos}\label{entropicchaos}

In this section, we investigate one step further to derive the entropic chaos. We shall make use of the entropy dissipation formula for the master equation \eqref{master} and the Landau equation \eqref{compactlandau}, and prove the lower semi-continuity of the entropy and the entropy dissipation functional. 

First we consider the lower semi-continuity of the entropy, which is quite standard in the literature.

\begin{lemma}[Lower Semi-continuity of Entropy]\label{entropycontinuity}
    Under the assumptions of Theorem \ref{thm:poc}, we have
    \begin{equation*}
        \liminf_{N \rightarrow \infty} \frac{1}{N} H(F_N) \geq H(f).
    \end{equation*}
\end{lemma}
\begin{proof}
    We first express the entropy in terms of the relative entropy with respect to some Gaussian variable. Denote that
    \begin{equation*}
        \gamma=\frac{1}{Z} \exp(-|v|^2), \quad \gamma_N=\frac{1}{Z_N} \exp \Big(-\sum_{i=1}^N |v^i|^2 \Big),
    \end{equation*}
    where $Z$ and $Z_N$ are partition functions satisfying $Z_N=Z^N$. We rewrite the entropy into
    \begin{align*}
        H(F_N)=&\, H(F_N|\gamma_N)+\int F_N \Big(-\sum_{i=1}^N |v^i|^2 -\log Z_N \Big)\\
        =&\, H(F_N|\gamma_N)-3N -N \cdot \log Z.
    \end{align*}
    Similarly we have
    \begin{align*}
        H(f)=H(f|\gamma)+\int f(-|v|^2-\log Z)=H(f|\gamma)-3-\log Z.
    \end{align*}
    Now we apply the Donsker-Varadhan variational formula for the relative entropy:
    \begin{equation*}
        H(F_N|\gamma_N)=\sup_{\Phi \in C_b} \Big( \int \Phi F_N-\log \int e^\Phi \gamma_N \Big)= \sup_{\Phi \in C_b, \int e^\Phi \gamma_N=1} \int \Phi F_N.
    \end{equation*}
    \begin{equation*}
        H(f|\gamma)=\sup_{\phi \in C_b} \Big( \int \phi f-\log \int e^\phi \gamma \Big)=\sup_{\phi \in C_b, \int e^\phi \gamma=1} \int \phi f.
    \end{equation*}
    For any $\eps>0$, choose $\phi \in C_b$ with $\int e^\phi \gamma=1$ such that $\int \phi f>H(f|\gamma)-\eps$. Then we let $\Phi(v^1,\cdots,v^N)=\phi(v^1)+\cdots+\phi(v^N)$, which satisfies $\int e^\Phi \gamma_N=1$ :
    \begin{align*}
        H(F_N) &\, \geq \int (\phi(v^1)+ \cdots+\phi(v^N)) F_N-3N-N \cdot \log Z\\
        &\, =N \int \phi F_{N,1}-3N-N \cdot \log Z\\
        &\, >NH(f)-N\eps +N \int \phi(F_{N,1}-f).
    \end{align*}
    By the weak convergence from Theorem \ref{thm:poc}, we have $\liminf_{N \rightarrow \infty} \frac{1}{N} H(F_N) \geq H(f)-\eps$. Let $\eps \rightarrow 0$ and the result is concluded.
\end{proof}

Next we prove the lower semi-continuity of the entropy dissipation functional. The basic idea follows similarly from \cite[Section 5]{carrapatoso2016propagation}, but due to the singularity of the coefficients in the cases of very soft potentials, especially the Coulomb interactions, it requires more careful estimates.

\begin{lemma}[Lower Semi-continuity of Entropy Dissipation Functional]\label{dissipationlower}
    Define the entropy dissipation functional
    \begin{equation*}
        D(F_N)=\frac{1}{2N} \sum_{i,j=1}^N \int_{\R^{3N}} a(v^i-v^j):(\nabla_{v^i} \log F_N-\nabla_{v^j} \log F_N)^{\otimes 2} F_N \ud v^{[N]},
    \end{equation*}
    \begin{equation*}
        D(f)=\frac{1}{2}\iint_{\R^6} a(v-w):(\nabla \log f(v)-\nabla \log f(w))^{\otimes 2} f(v)f(w) \ud v \ud w.
    \end{equation*}
    Under the assumptions of Theorem \ref{thm:poc}, we have
    \begin{equation*}
        \liminf_{N\rightarrow \infty} \frac{1}{N} D(F_N) \geq D(f).
    \end{equation*}
\end{lemma}
\begin{proof}
    We rewrite the entropy dissipation functional in the variational form:
    \begin{align*}
        D(f)=\sup_{\phi \in C_c^\infty,\phi(v,w)=-\phi(w,v)} &\, \Big( \iint a(v-w):((\nabla \log f(v)-\nabla \log f(w)) \otimes \phi(v,w) f(v)f(w)\\
        -\frac{1}{2}&\, \iint a(v-w):\phi(v,w) \otimes \phi(v,w) f(v)f(w) \Big).
    \end{align*}
    Indeed we can construct such approximating sequence of $\phi$ by mollifying $\nabla \log f(v)-\nabla \log f(w)$ and truncating at infinity. For any $\eps>0$, choose $\phi \in C_c^\infty$ with $\phi(v,w)=-\phi(v,w)$ such that
    \begin{align*}
        D(f)-\eps< &\, \iint a(v-w):((\nabla \log f(v)-\nabla \log f(w)) \otimes \phi(v,w) f(v)f(w)\\
        -\frac{1}{2}&\, \iint a(v-w):\phi(v,w) \otimes \phi(v,w) f(v)f(w).
    \end{align*}
    Then we let $\Phi(v^1,\cdots,v^N)=\phi(v^1,v^2)$ and compute
    \begin{align*}
        D(F_N)&\, =\frac{N-1}{2}\int a(v^1-v^2) :(\nabla_{v^1} \log F_N-\nabla_{v^2} \log F_N)^{\otimes 2} F_N\\
        &\, \geq (N-1) \Big(\int a(v^1-v^2):(\nabla_{v^1} \log F_N-\nabla_{v^2} \log F_N) \otimes \Phi F_N-\frac{1}{2} \int a(v^1-v^2):\Phi \otimes \Phi F_N \Big)\\
        &\, =-(N-1)\iint \Big( 2b(v-w) \cdot \phi +a(v-w):(\nabla_v \phi-\nabla_w \phi)+\frac{1}{2}a(v-w):\phi \otimes \phi\Big) F_{N,2}(v,w).
    \end{align*}
    Hence, we have
    \begin{equation}\label{lowerbound}
        \begin{aligned}
            &\, \liminf_{N \rightarrow \infty} \frac{1}{N}D(F_N) \geq\\
            -&\, \liminf_{N \rightarrow \infty}\iint \Big( 2b(v-w) \cdot \phi +a(v-w):(\nabla_v \phi-\nabla_w \phi)+\frac{1}{2}a(v-w):\phi \otimes \phi\Big) F_{N,2}(v,w).
        \end{aligned}
    \end{equation}
    In the cases of hard potentials or Maxwellian molecules, there are no singularities in $a$ and $b$, which allows us to pass to the limit using the weak convergence directly. In the cases of moderately soft potentials, we use the anti-symmetry of $\phi$ to rewrite the double integral into
    \begin{align*}
        -\iint \Big( &\, b(v-w) \cdot (\phi(v,w)-\phi(w,v))\\
        +&\, a(v-w):(\nabla_v \phi-\nabla_w \phi)+\frac{1}{2}a(v-w):\phi \otimes \phi \Big) F_{N,2}(v,w).
    \end{align*}
    Since the matrix $a$ has no singularity and the vector field $b$ has possible singularity at most $|\cdot|^{\gamma+1}$ with $\gamma \geq -2$, the integrand inside the big bracket is again bounded and regular. However, in the regime of very soft potentials, the test function multiplying with $F_{N,2}$ is essentially singular. This requires another truncation process. For any $\delta>0$, we define the partition of unity:
    \begin{equation*}
        1=p_1+p_2, \quad a_i(v)=a(v)p_i(|v|),\quad  b_i(v)=b(v)p_i(|v|), \quad i=1,2,
    \end{equation*}
    such that $p_1,p_2 \in C^\infty([0,\infty))$, $p_1$ is supported in $[0,\delta]$, $p_2$ vanishes in $[0,\delta/2]$.
    We apply this decomposition to \eqref{lowerbound}. For the term involving the singular quantities $a_1$ and $b_1$, we control from below by
    \begin{align*}
        &\, \iint -b_1(v-w) \cdot (\phi(v,w)-\phi(w,v)) F_{N,2}(v,w)-C\|\phi\|_{W^{1,\infty}} \iint_{|v-w| \leq \delta} \frac{F_{N,2}(v,w)}{|v-w|}\\
        \geq &\, -C\|\phi\|_{W^{1,\infty}} \iint_{|v-w| \leq \delta} |v-w|^{\gamma+2} F_{N,2}(v,w)\\
        \geq &\, -C\delta^{3+\gamma}\|\phi\|_{W^{1,\infty}} \iint_{|v-w| \leq \delta} |v-w|^{-1} F_{N,2}(v,w)\\
        \geq &\, -C_\phi \delta^{\frac{7}{2}+\gamma} I(F_{N,2}),
    \end{align*}
    where we use the anti-symmetry of $\phi$ and the Sobolev inequality and the technique in Fournier-Hauray-Mischler \cite{fournier2014propagation}. Using the fact that the Fisher information of $F_N$ is monotonically decreasing, we bound the singular part from below by $-C\delta^{7/2+\gamma}$ with $C$ depending on $\phi$ and $I(f_0)$. As for the regular part, we can directly use the weak convergence of $F_{N,2}$ to pass to the limit to obtain
    \begin{align*}
        -\iint \Big( 2b_2(v-w) \cdot \phi+a_2(v-w):(\nabla_v \phi-\nabla_w \phi)+\frac{1}{2} a_2(v-w):\phi \otimes \phi \Big)f(v)f(w).
    \end{align*}
    Notice that $\nabla \cdot a_i=bp_i+a \cdot \nabla p_i=b_i$, since $a(v) \cdot \nabla p_i(|v|)=a(v) \cdot v \cdot p_i^\prime(|v|)/|v|$ and $a(v) \cdot v=0$. Hence, we integrate by parts to rewrite the above double integral into
    \begin{align*}
        &\, \iint \Big(a_2(v-w) : (\nabla \log f(v)-\nabla \log f(w)) \otimes \phi-\frac{1}{2}a_2(v-w):\phi \otimes \phi \Big) f(v)f(w)\\
        \geq &\, D(f)-\eps-\iint \Big(a_1(v-w) : (\nabla \log f(v)-\nabla \log f(w)) \otimes \phi-\frac{1}{2}a_1(v-w):\phi \otimes \phi \Big) f(v)f(w)\\
        =&\, D(f)-\eps+\iint \Big( 2b_1(v-w) \cdot \phi +a_1(v-w):(\nabla_v \phi-\nabla_w \phi)+\frac{1}{2}a_1(v-w):\phi \otimes \phi\Big) f(v)f(w)\\
        \geq &\, D(f)-\eps-C \delta^{\frac{7}{2}+\gamma}
    \end{align*}
    by our choice of $\phi$. The last step follows from the estimate of the singular part as above. Now we conclude that $\liminf_{N \rightarrow \infty} \frac{1}{N}D(F_N) \geq D(f)-\eps-C\delta^{7/2+\gamma}$ with $C=C_\eps$. Let $\delta \rightarrow 0$ and then $\eps \rightarrow 0$, and the result is concluded.
\end{proof}

We combine the two lower semi-continuity results together with the entropy dissipation formulas for the master equation and the Landau equation:
\begin{equation}\label{dissipation1}
    H(F_N(t))+\int_0^t D(F_N(s)) \ud s \leq H(F_N(0)),
\end{equation}
\begin{equation}\label{dissipation2}
    H(f(t))+\int_0^t D(f(s)) \ud s=H(f_0),
\end{equation}
where \eqref{dissipation1} is derived in Appendix \ref{appendixB} in the sense of the construction of $H$-solutions in \cite{villani1998new} and \eqref{dissipation2} follows from classical results. Notice that $\frac{1}{N} H(F_N(0))=H(f_0)$ from our choice of initial data, we conclude that
\begin{equation*}
    \lim_{N \rightarrow \infty}\frac{1}{N} H(F_N)=H(f), \quad \lim_{N \rightarrow \infty}\frac{1}{N} D(F_N)=D(f).
\end{equation*}
This establishes the entropic chaos Theorem \ref{thm:entropic}.

Inspired by \cite[Proof of Theorem 2.13]{fournier2014propagation}, we use the entropic chaos to derive almost surely convergence for marginals, and then apply the compactness argument to conclude the $L^1$ propagation of chaos.

On the one hand, the probability measures $\{ F_{N,k} \}$ in $\R^{3k}$ have uniformly bounded entropy and kinetic energy:
\begin{equation*}
    \int_{\R^{3k}} F_{N,k} \log F_{N,k} \ud v^{[k]} \leq kH(f_0), \quad \int_{\R^{3k}} F_{N,k} \sum_{i=1}^k |v^i|^2 \ud v^{[k]}=k \int_{\R^3} |v|^2f_0(v) \ud v,
\end{equation*}
which shows that $F_{N,k}$ is uniformly bounded in $L \log L$, and thus uniformly integrable in $L^1$, see the discussions in Appendix \ref{appendixB}. By the Dunford-Pettis theorem, $F_{N,k}$ is weakly compact in $L^1(\R^{3k})$.

On the other hand, we already have that $F_{N,k}$ weakly converges to $f^{\otimes k}$ and that
\begin{equation*}
    \lim_{N \rightarrow \infty} \frac{1}{N}H(F_N)=H(f).
\end{equation*}
By the lower semi-continuity of the entropy, for any $k \geq 1$,
\begin{equation*}
    H(f^{\otimes k}) \leq \liminf_{N \rightarrow \infty} H(F_{N,k}) \leq \limsup_{N \rightarrow \infty} H(F_{N,k}) \leq \limsup_{N \rightarrow \infty} \frac{k}{N}H(F_N)=kH(f)=H(f^{\otimes k}),
\end{equation*}
where we have used the super-additivity of entropy (see for instance \cite[Lemma 3.3 (iv)]{hauray2014kac}) to obtain the last inequality. This shows that the convergence of the normalized entropy holds for the marginals: $H(F_{N,k}) \rightarrow H(f^{\otimes k})$. Now we introduce the average measures
\begin{equation*}
    G_{N,k}=\frac{1}{2}(F_{N,k}+f^{\otimes k}),
\end{equation*}
which obviously converges weakly to $f^{\otimes k}$. Then we have $\limsup_{N \rightarrow \infty} \frac{1}{2}H(F_{N,k})+\frac{1}{2}H(f^{\otimes k})-H(G_{N,k}) \leq 0$ by lower semi-continuity. But the function $s \mapsto s \log s$ is strictly convex, which yields a point-wise convergent subsequence
\begin{equation*}
    F_{N_l,k} \rightarrow f^{\otimes k} \quad \text{a.s.}.
\end{equation*}

Combining the two arguments above, we conclude that $F_{N,k}$ converges to $f^{\otimes k}$ in $L^1$, which is exactly the $L^1$ propagation of chaos. This completes the proof of the rest part of Theorem \ref{thm:entropic}.

\subsection*{Acknowledgements}
The authors would like to thank Pierre-Emmanuel Jabin for helpful discussions and valuable suggestions. This work was partially supported by the National Key R\&D Program of China (Project Nos. 2024YFA1015500 and 2021YFA1002800) and by the NSFC Grant No. 12171009. 

\appendix

\section{Well-posedness of the Master Equation}\label{appendixB}

Our duality approach for Kac's program works on the level of the Landau master equation and its backward dual equation, which necessitates the justification of the well-posedness of the Landau master equation \eqref{master}. In the recent work Carrillo-Guo \cite{carrillo2025fisher}, they have successfully verified the existence, uniqueness and regularity for the Landau master equation, by proving the Fisher information of the joint law $F_N$ is monotonically decreasing. For the sake of completeness, we go through the classical approximation method in this short section and give the needed well-posedness results.

Recall that our Landau master equation or forward Kolmogorov equation of Kac's particle system is given by
\begin{equation*}
    \p_t F_N=\frac{1}{2N}\sum_{i,j=1}^N (\nabla_{v^i}-\nabla_{v^j}) \cdot \Big( a(v^i-v^j) \cdot (\nabla_{v^i}-\nabla_{v^j})F_N \Big),
\end{equation*}
with the initial data $F_N(0,\cdot)=f_0^{\otimes N}$. This is a linear parabolic equation in divergence form, but there are two major difficulties for us to establish the existence of classical solutions: the first being the lack of uniform ellipticity for the operator, since the coefficient matrix $a(z)$ is degenerate in the $z$ direction, performing some anisotropy property; the second being the singularity of the coefficient matrix $a(z)$ at the origin. Hence, we turn to the regime of weak solutions.

We apply the classical regularization technique and compactness argument to derive the well-posedness of weak solutions to the master equation. The entire process is fairly standard. First, we truncate and mollify the coefficients and the initial data to overcome the two difficulties stated above, and classical solutions exist uniquely for the regularized system by standard linear parabolic PDE theory. Then we prove a series of \textit{a priori} estimates for the regularized system, giving the uniqueness of the weak solution and the sequential compactness of the regularized solutions. Finally, we pass to the limit and show that the weak limit solves the Landau master equation in the sense of distribution.

For any regularization parameter $\eps>0$, we introduce the following regularized system to overcome the singularity and degeneracy and mollify the initial data:
\begin{equation}\label{regularization}
    \begin{cases}
        \p_t F_{N,\eps}= \frac{1}{2N} \sum_{i,j=1}^N (\nabla_{v^i}-\nabla_{v^j}) \cdot \Big( a_\eps(v^i-v^j) \cdot (\nabla_{v^i}-\nabla_{v^j}) F_{N,\eps} \Big) +\eps \sum_{i=1}^N \Delta_{v^i} F_{N,\eps},\\
        F_{N,\eps}(0,\cdot)=f_{0,\eps}^{\otimes N}.
    \end{cases}
\end{equation}
Here, we replace the singular coefficient matrix $a(z)$ by the smooth one $a_\eps(z)$ using simple truncation near the origin: $a_\eps(z)=\chi_\eps(|z|)a(z)=\chi(|z|/\eps)a(z)$, where the cut-off function $\chi(r)$ satisfies
\begin{equation*}
    \chi(r) \in C^\infty([0,\infty)), \quad \chi|_{[0,1]}=0, \quad \chi|_{[2,\infty)}=1, \quad 0 \leq \chi \leq 1, \quad \|\nabla^n \chi \|_{L^\infty} \leq C_n.
\end{equation*}
We also add some small diffusion term to gain the uniform ellipticity. As for the initial data, we keep the factorized structure but mollify $f_0$ into $f_{0,\eps}$ by
\begin{equation*}
    f_{0,\eps}(v)=(f_0 \ast \eta_\eps(v)+\eps)(1-\chi(\eps|v|)+\chi(\eps|v|)\exp(-|v|^2).
\end{equation*}
Here, $\eta_\eps$ is a family of standard mollifiers. In this sense, we have $f_{0,\eps} \in C^\infty(\R^3)$ with Gaussian upper bound and lower bound. By the standard linear parabolic PDE theory, there exists a unique bounded strong solution $F_{N,\eps} \in C^\infty([0,\infty);C^\infty(\R^{3N}))$ for the regularized system \eqref{regularization}.

We need to establish some \textit{a priori} estimates for regularized solutions $F_{N, \eps}$ so that we can apply the compactness argument. We obtain easily that $\lbrace F_{N, \eps} \rbrace_{\eps}$ have uniform-in-$\eps$ upper bounds in mass, momentum, kinetic energy and entropy using integration by parts. Notice that for any probability density function $f$, let $\rho$ be the probability density of standard normal distribution, we have
\begin{align*}
    \int_{f \leq 1} f \log f =&\, \int_{f \leq 1} \Big(f \log \frac{f}{\rho}-f+\rho \Big) +f \log \rho +f-\rho\\
    \geq &\, \int_{f \leq 1} f \log \rho-\rho\\
    \geq &\, -C-\frac{1}{2}\int |v|^2 f.
\end{align*}
Hence
\begin{equation*}
    \int f |\log f| \leq \int f \log f+\int |v|^2 f+C.
\end{equation*}
This shows that the sequence of functionals $\Big\lbrace \int_{\R^{3N}} F_{N,\eps} |\log F_{N,\eps}| \ud v^{[N]} \Big\rbrace_\eps$ is uniformly bounded. Together with the Chebyshev inequality, we deduce that for each fixed $t \in [0,T]$, $\lbrace F_{N,\eps}(t,\cdot) \rbrace_{\eps}$ is uniformly integrable in $L^1(\R^{3N})$. By the Dunford-Pettis theorem, there exists some $F_N(t,\cdot)$ such that for some subsequence of regularization parameters $\eps_l \rightarrow 0$, $F_{N,\eps_l}(t, \cdot)$ weakly converges to $F_N(t, \cdot)$ in $L^1(\R^{3N})$ as $l \rightarrow \infty$. Note that the subsequence $\eps_l$ may vary for different $t \in [0,T]$.

Now we state the proposition concerning the maximum principle and the uniform $L^\infty$ norm estimate of $F_{N,\eps}$.
\begin{prop}[Maximum Principle]\label{maximum}
    We have $F_{N,\eps} \in L^\infty([0,\infty);L^1 \cap L^\infty(\R^{3N}))$ with the uniform estimate
\begin{equation*}
    \| F_{N,\eps} \|_{L^\infty([0,\infty);L^1 \cap L^\infty(\R^{3N}))} \leq \| F_{N,\eps}(0,\cdot) \|_{L^1 \cap L^\infty(\R^{3N})}.
\end{equation*}
\end{prop}
\begin{proof}
    The estimate of $L^1$ norm is trivial since the regularized master equation conserves mass. We sketch the proof of the $L^\infty$ norm estimate and the maximum principle.

    Denote by $M_0=\| F_{N,\eps}(0,\cdot) \|_{L^\infty(\R^{3N})}$. Since the regularized master equation is a linear parabolic PDE with smooth and bounded coefficients, it admits a unique smooth and bounded classical solution $F_{N,\eps}$. Let $M=\| F_{N,\eps}\|_{L^\infty([0,\infty);L^\infty(\R^{3N}))}$. It suffices to prove $M=M_0$.

    For any $\delta >0$, consider the auxiliary function
    \begin{equation*}
        G_N(t,v^{[N]})=F_{N,\eps}(t,v^{[N]})-\delta(3N\eps t+\sum_{i=1}^N |v^i|^2)-M_0.
    \end{equation*}
    It is straight forward to check that $G_N(0,\cdot) \leq 0$ by the definition of $M_0$ and that $G_N \leq 0$ for $|v^1|^2 + \cdots +|v^N|^2 \geq M/\delta$ by the definition of $M$. Hence, for any $T_0>0$, if $G_N>0$ holds somewhere inside the cylinder $[0,T_0] \times B_{M/\delta}$, it must achieve its positive maximum somewhere inside the cylinder but not on the parabolic boundary. At the point where the maximum is achieved, we have
    \begin{equation}\label{maximalpoint}
        \p_t G_N \geq 0, \quad \nabla_{v^i} G_N=0, \quad \Delta_{v^i} G_N \leq 0.
    \end{equation}
    However, we compute directly that
    \begin{align*}
        &\, \p_t G_N-\frac{1}{2N}\sum_{i,j=1}^N (\nabla_{v^i}-\nabla_{v^j}) \cdot \Big( a_\eps(v^i-v^j) \cdot (\nabla_{v^i}-\nabla_{v^j})G_N \Big)-\eps \sum_{i=1}^N \Delta_{v^i} G_N\\
        =&\, -3\delta N \eps+\frac{\delta}{N}\sum_{i,j=1}^N (\nabla_{v^i}-\nabla_{v^j}) \cdot \Big( a_\eps(v^i-v^j) \cdot (v^i-v^j) \Big)+2\delta N \eps\\
        =&\, -\delta N \eps<0,
    \end{align*}
    where we use that $a_\eps(z) \cdot z=0$. This contradicts \eqref{maximalpoint}. Hence we deduce that $G_N \leq 0$ inside the cylinder $[0,T_0] \times B_{M/\delta}$.
    
    For any fixed point $(t_0,v_0^1, \cdots, v_0^N)$, we choose $T_0$ large enough and $\delta$ small enough such that the point is inside $[0,T_0] \times B_{M/\delta}$:
    \begin{equation*}
        F_{N,\eps}(t_0,v_0^1, \cdots, v_0^N) \leq M_0+\delta(N \eps t_0+\sum_{i=1}^N |v_0^i|^2).
    \end{equation*}
    Let $\delta \rightarrow 0$, and we have $F_{N,\eps} \leq M_0$ at this point. Similarly we have $F_{N,\eps} \geq -M_0$ at this point. This deduces that $M=M_0$.
\end{proof}
By the Banach-Alaoglu theorem, since the Banach space $L^\infty([0,\infty);L^\infty(\R^{3N}))$ is the dual space of $L^1([0,\infty);L^1(\R^{3N}))$, we have weak compactness of $F_{N,\eps}$, i.e. there exists some subsequence of regularization parameters $\eps_l \rightarrow 0$, $F_{N,\eps_l}$ weakly-$\ast$ converges to $F_N$ in $L^\infty([0,\infty);L^\infty(\R^{3N}))$. In particular, since the coefficient matrix $a$ can be decomposed into two parts $a_1+a_2$ as in the proof of Lemma \ref{dissipationlower} with $\delta=2$, we can pass to the limit in the weak formulation of the regularized Landau master equation to deduce that $F_N$ solves the Landau master equation in the distributional sense. Detailed discussion is postponed to the end of this subsection.

Moreover, by direct computations and the uniform bound of Fisher information for $F_{N,\eps}$ in \cite{carrillo2025fisher} (in a similar sense to the proof of Lemma \ref{dissipationlower}), we can check that for any test function $\varphi \in C_c^2(\R^{3N})$, the sequence of functionals $\Big\lbrace \int_{\R^{3N}} \varphi F_{N,\eps} \ud v^{[N]} \Big\rbrace_\eps$ is uniformly bounded and Lipschitz equi-continuous on any finite time interval $[0,T]$:
\begin{equation*}
    \Big| \int_{\R^{3N}} \varphi F_{N,\eps} \ud v^{[N]} \Big| \leq C \| \varphi \|_{L^\infty},
\end{equation*}
\begin{equation*}
    \Big| \int_{\R^{3N}} \varphi F_{N,\eps}(t, \cdot) \ud v^{[N]}-\int_{\R^{3N}} \varphi F_{N,\eps}(s, \cdot) \ud v^{[N]} \Big| \leq C|t-s| \cdot \| \varphi \|_{W^{2,\infty}}.
\end{equation*}
By the Arzel\`a-Ascoli lemma, and, using that $C_c^2(\R^{3N})$ is dense in $C_c(\R^{3N})$, an extraction of the diagonal, we obtain some subsequence of regularization parameters $\eps_l \rightarrow 0$ such that, for all $\varphi \in C_c(\R^{3N})$, the sequence of functionals $\Big \lbrace \int_{\R^{3N}} \varphi F_{N,\eps_l} \ud v^{[N]} \Big \rbrace_l$ converges uniformly in time $t \in [0,T]$ to some functional $G_{N,\varphi}(t)$, which is Lipschitz continuous in $t$. Notice that for any fixed $t$, the map
\begin{equation*}
    \varphi \mapsto G_{N,\varphi}(t)
\end{equation*}
is linear and positive on $C_c(\R^{3N})$. By the Riesz-Markov-Kakutani representation theorem, the limiting functional $G_{N,\varphi}(t)$ is given by testing $\varphi$ on some Borel measure. Combing with the weak convergence in the previous paragraph, we deduce that this Borel measure is given exactly by $F_N(t,\cdot)$.

Now we have already derived the existence of the weak limit $F_N(t,\cdot)$ up to some subsequence of regularization parameters $\eps_l \rightarrow 0$. This also provides us with the regularity for the weak limit $F_N \in L^\infty([0,\infty);L^1 \cap L^\infty(\R^{3N}))$ with conserved mass, momentum, energy, and with monotonically decreasing entropy and Fisher information. We are left to check that it solves the Landau master equation in the weak sense, hence leading to the existence of weak solutions. The uniqueness result will be derived directly from the master equation, which also shows that the weak limit holds for the whole family of regularization parameters $\eps>0$.

For any test function $\varphi \in C_c^\infty(\R^{3N})$, we rewrite the regularized system \eqref{regularization} into the weak form, say multiplying both sides with $\varphi$ and integrating in $[0,t] \times \R^{3N}$:
\begin{align*}
    \int_{\R^{3N}} \varphi F_{N,\eps}(t,\cdot) \ud v^{[N]}= &\, \int_{\R^{3N}} \varphi F_{N,\eps}(0,\cdot) \ud v^{[N]}\\
    -\frac{1}{2N} \int_0^t &\, \int_{\R^{3N}} \sum_{i,j=1}^N  a_\eps(v^i-v^j) : (\nabla_{v^i} \varphi-\nabla_{v^j} \varphi) \\
    &\, \otimes (\nabla_{v^i}-\nabla_{v^j})  F_{N,\eps} \ud v^{[N]} \ud s
    + \eps \int_0^t \int_{\R^{3N}} \sum_{i=1}^N F_{N,\eps} \Delta_{v^i} \varphi \ud v^{[N]} \ud s.
\end{align*}
By the uniform convergence, we can easily deduce that the integrals in the first line converge into the corresponding integrals with $F_N$ replacing $F_{N,\eps}$, and that the integral in the last term vanishes in the limit. For the remaining term, we integrate by parts and use the decomposition stated as above
\begin{align*}
    &\, \frac{1}{2N} \int_0^t \int_{\R^{3N}} \sum_{i,j=1}^N (\nabla_{v^i}-\nabla_{v^j}) \cdot \Big( a_{1,\eps}(v^i-v^j) \cdot (\nabla_{v^i}\varphi-\nabla_{v^j} \varphi) \Big) F_{N,\eps} \ud v^{[N]} \ud s\\
    +&\, \frac{1}{2N} \int_0^t \int_{\R^{3N}} \sum_{i,j=1}^N (\nabla_{v^i}-\nabla_{v^j}) \cdot \Big( a_{2,\eps}(v^i-v^j) \cdot (\nabla_{v^i}\varphi-\nabla_{v^j} \varphi) \Big) F_{N,\eps} \ud v^{[N]} \ud s.
\end{align*}
The test function in the first integral is uniformly-in-$\eps$ $L^1$ bounded, which allows us to pass to the limit using the weak-$\ast$ convergence of $F_{N,\eps}$ in $L^\infty$. The test function in the second integral is uniformly-in-$\eps$ smooth and compactly supported, which allows us to pass to the limit using the uniform convergence of $F_N$ testing on $C_c^2(\R^{3N})$.

Finally we present the entropy dissipation inequality \eqref{dissipation1} of the weak solution $F_N$, which is needed for establishing the entropic chaos in Section \ref{entropicchaos}, in the sense of the construction of $H$-solutions in \cite{villani1998new}. We make use of the entropy dissipation formula for the regularized system \eqref{regularization}:
\begin{equation}\label{dissipationregular}
    \begin{aligned}
        H(F_{N,\eps}(0,\cdot))=&\, H(F_{N,\eps}(t,\cdot))\\
    +&\, \frac{1}{2N} \sum_{i,j=1}^N \int_0^t \int_{\R^{3N}} a_\eps(v^i-v^j):\Big((\nabla_{v^i}-\nabla_{v^j}) \sqrt{F_{N,\eps}} \Big)^{\otimes 2} \ud v^{[N]} \ud s\\
    +&\, \eps \sum_{i=1}^N \int_0^t \int_{\R^{3N}} \frac{|\nabla_{v^i} F_{N,\eps}|^2}{F_{N,\eps}} \ud v^{[N]} \ud s.
    \end{aligned}
\end{equation}
Using the fact that the entropy of $F_{N,\eps}$ is uniformly bounded in $\eps$ and $t$, we deduce that
\begin{equation*}
    \frac{1}{2N} \sum_{i,j=1}^N \int_0^t \int_{\R^{3N}} a_\eps(v^i-v^j):\Big((\nabla_{v^i}-\nabla_{v^j}) \sqrt{F_{N,\eps}} \Big)^{\otimes 2} \ud v^{[N]} \ud s
\end{equation*}
is uniformly bounded in $\eps$ and $t$. Notice that the truncated coefficient matrix $a_\eps$ is monotonically increasing with respect to $\eps$. Fix any $\eps_0>0$ and choose the sequence of parameters $\eps_l \rightarrow 0$ such that $F_{N,\eps_l}$ converges weakly to $F_N$ uniformly in $t \in [0,T]$, and we have
\begin{align*}
     \liminf_{\eps_l \rightarrow 0} &\, \frac{1}{2N} \sum_{i,j=1}^N \int_0^t \int_{\R^{3N}} a_{\eps_0}(v^i-v^j):\Big((\nabla_{v^i}-\nabla_{v^j}) \sqrt{F_{N,\eps_l}} \Big)^{\otimes 2} \ud v^{[N]} \ud s\\
    \geq &\, \frac{1}{2N} \sum_{i,j=1}^N \int_0^t \int_{\R^{3N}} a_{\eps_0}(v^i-v^j):\Big((\nabla_{v^i}-\nabla_{v^j}) \sqrt{F_N} \Big)^{\otimes 2} \ud v^{[N]} \ud s
\end{align*}
by the convexity of the operator and the weak convergence. Then we have
\begin{align*}
    \liminf_{\eps_l \rightarrow 0} &\, \frac{1}{2N} \sum_{i,j=1}^N \int_0^t \int_{\R^{3N}} a_{\eps_l}(v^i-v^j):\Big((\nabla_{v^i}-\nabla_{v^j}) \sqrt{F_{N,\eps_l}} \Big)^{\otimes 2} \ud v^{[N]} \ud s\\
    \geq &\, \frac{1}{2N} \sum_{i,j=1}^N \int_0^t \int_{\R^{3N}} a_{\eps_0}(v^i-v^j):\Big((\nabla_{v^i}-\nabla_{v^j}) \sqrt{F_N} \Big)^{\otimes 2} \ud v^{[N]} \ud s
\end{align*}
by the monotone property of $a_\eps$. Let $\eps_0 \rightarrow 0$, we have
\begin{align*}
    \liminf_{\eps_l \rightarrow 0} &\, \frac{1}{2N} \sum_{i,j=1}^N \int_0^t \int_{\R^{3N}} a_{\eps_l}(v^i-v^j):\Big((\nabla_{v^i}-\nabla_{v^j}) \sqrt{F_{N,\eps_l}} \Big)^{\otimes 2} \ud v^{[N]} \ud s\\
    \geq &\, \frac{1}{2N} \sum_{i,j=1}^N \int_0^t \int_{\R^{3N}} a(v^i-v^j):\Big((\nabla_{v^i}-\nabla_{v^j}) \sqrt{F_N} \Big)^{\otimes 2} \ud v^{[N]} \ud s.
\end{align*}
This concludes the entropy dissipation inequality \eqref{dissipation1} if we take the inferior limit on both sides of \eqref{dissipationregular} and make use of the lower semi-continuity of the entropy.

\bibliographystyle{abbrv}
\bibliography{ref}

\end{document}